\documentclass[oneside,11pt]{amsart}%
\usepackage{amsthm,amsmath,amssymb,amsbsy,bbm,mathrsfs,supertabular}
\usepackage[authoryear]{natbib}

\RequirePackage[colorlinks,citecolor=blue,urlcolor=blue]{hyperref}

\newcommand{\bm}[1]{\mbox{\boldmath $#1$}}

\DeclareMathAlphabet{\mathonebb}{U}{bbold}{m}{n}
\newcommand{\1}{\ensuremath{\mathonebb{1}}}
\def\B{{\mathscr B}}

\def\E{{\mathbb{E}_{\gs}}}
\def\F{\mathbb{F}}

\def\IL{{\mathbb{L}}}
\def\N{{\mathbb{N}}}
\def\P{{\mathbb{P}}}
\def\Q{{\mathbb{Q}}} 
\def\R{{\mathbb{R}}}
\def\SS{{\mathbb{S}}}

\def\Z{{\mathbb{Z}}}

\def\A{{\mathscr A}}
\def\B{{{\mathscr B}}}
\def\CC{\mathscr C}

\def\EE{{\mathscr E}}
\def\FF{{\mathscr{F}}}
\def\GG{{\mathscr{G}}}
\def\HH{\mathscr H}
\def\I{{\mathcal I}}
\def\J{{\mathcal J}}
\def\LL{{\mathscr L}}
\def\M{\mathcal M}

\def\PP{\mathcal P}
\def\QQ{\mathcal Q}
\def\RR{{\mathscr R}}
\def\X{{\mathscr{X}}}
\def\S{{\mathbf S}}

\def\WW{\mathcal{W}}

\def\gf{\mathbf{f}}
\def\gg{\mathbf{g}}
\def\gh{\mathbf{h}}

\def\gP{{\mathbf{P}}}
\def\gp{\mathbf{p}}
\def\gq{\mathbf{q}}
\def\gr{{\mathbf{r}}}
\def\gu{{\mathbf{u}}}
\def\gs{{\mathbf{s}}}
\def\gt{{\mathbf{t}}}
\def\gT{{\mathbf{T}}}

\def\w{\mathbf{w}}
\def\gX{\boldsymbol{X}}
\def\gY{{\mathbf{Y}}}
\def\gx{{\mathbf{x}}}
\def\gZ{{\mathbf{Z}}}
\def\gK{\mathbf{K}}
\def\g0{{\mathbf{0}}}
\def\scr{\mathscr}

\def\gmu{\boldsymbol{\mu}}

\def\grho{\boldsymbol{\rho}}
\def\gvrho{\boldsymbol{\varrho}}

\def\gup{\boldsymbol{\Upsilon}}

\def\gW{{\mathbf{W}}}

\def\eps{{\varepsilon}}
\newcommand{\pen}{\mathop{\rm pen}\nolimits}

\def\<{{\langle}}
\def\>{{\rangle}}
\def\Var{{\rm Var}}

\newcommand{\eref}[1]{(\ref{#1})}
\newcommand{\pa}[1]{\left({#1}\right)}
\newcommand{\norm}[1]{\left\|{#1}\right\|}
\newcommand{\cro}[1]{\left[{#1}\right]}
\newcommand{\ab}[1]{\left|{#1}\right|}
\newcommand{\ac}[1]{\left\{{#1}\right\}}

\newtheorem{thm}{Theorem}
\newtheorem{lem}[thm]{Lemma}
\newtheorem{prop}[thm]{Proposition}
\newtheorem{cor}[thm]{Corollary}
\newtheorem{defi}[thm]{Definition}
\newtheorem{ass}[thm]{Assumption}

\def\1{1\hskip-2.6pt{\rm l}}

\def\telque{\big |}

\def\Y{\FF^{S}(\gs,\overline \gs,y)}
\def\Ent{{\overline\HH}}
\def\epp{{\epsilon}}
\newcommand{\Etc}[2]{#1_1,\ldots,#1_{#2}}
\usepackage{color}

\begin{document}


\title{A new method for estimation and model selection:\\
{\Large$\rho\,$}-Estimation}

\author{Y. Baraud}
\address{Univ. Nice Sophia Antipolis, CNRS,  LJAD, UMR 7351, 06100 Nice, France.}
\email{baraud@unice.fr}
\author{L. Birg\'e}
\address{Sorbonne Universit\'es, UPMC Univ.\ Paris 06, CNRS - UMR 7599, LPMA - Case courrier 188, 75252 Paris Cedex 05, France.}
\email{lucien.birge@upmc.fr}
\author{M. Sart}
\address{Univ Lyon, UJM-Saint-Etienne, CNRS, Institut Camille Jordan UMR 5208, F-42023, SAINT-ETIENNE, France.}
\email{mathieu.sart@univ-st-etienne.fr}
\date{\today}

\begin{abstract}
The aim of this paper is to present a new estimation procedure that can be applied in various 
statistical frameworks including density and regression and which leads to both robust and 
optimal (or nearly optimal) estimators. In density estimation, they asymptotically coincide with 
the celebrated maximum likelihood estimators at least when the statistical model is regular 
enough and contains the true density to estimate. For very general models of densities, 
including non-compact ones, these estimators are robust with respect to the Hellinger 
distance and converge at optimal rate (up to a possible logarithmic factor) in all cases 
we know. In the regression setting, our approach improves upon the classical least squares 
in many respects. In simple linear regression for example, it provides an estimation of the 
coefficients that are both robust to outliers and simultaneously rate-optimal (or nearly 
rate-optimal) for a large class of error distributions including Gaussian, Laplace, Cauchy 
and uniform among others.
\end{abstract}

\maketitle

\section{Introduction\label{I}}
The primary scope of this paper was to design a new and more or less universal estimation 
method for the regression framework where we observe $n$ independent real random variables 
$X_{1},\ldots,X_{n}$ of the form $X_i=f_i+\varepsilon_i$ where the $f_i$ are the unknown 
parameters of interest and the $\varepsilon_i$ i.i.d.\ real random errors with a partially unknown 
distribution which may be quite different from the usual Gaussian one. The problem arose
from a question by Oleg Lepski to the first author during his visit to Nice in January 2012. This
question was about the regression framework when the errors have rather unusual distributions, 
in which case the classical least squares method can be far from optimal. 
That was the starting point of our study which finally resulted in a much broader approach
and the design of a new class of estimators with several remarquable and partly unexpected
properties.

The regression frameworks that we shall consider here are of the form $Z_i=f(W_i)+\varepsilon_i$ 
for $1\le i\le n$, where the $Z_i$ are real observations, the $\varepsilon_i$ i.i.d.\ errors with 
density $p$ with respect to the Lebesgue measure $\mu$ on $\R$, $f$ is an unknown 
function from $\WW$ to $\R$ and the $W_i\in\WW$ are explanatory variables which may either 
be deterministic, in which case $W_i=x_i$ and $f(x_i)=f_i$, or random and i.i.d. This leads to 
the two classical regression frameworks on $\mathbb{R}^n$ that we shall consider in the sequel:
\[
X_i=f_i+\varepsilon_i\qquad\mbox{and}\qquad X_i=(W_i,Y_i)\quad\mbox{with}\;\;
Y_i=f(W_i)+\varepsilon_i\quad\mbox{for }1\le i\le n.
\]
The first case corresponds to {\em fixed design regression} for which $X_i$ has density 
$p(\cdot-f_i)$ with respect to $\mu$, the second case to {\em random design regression} 
with i.i.d.\ random explanatory variables $W_i$ independent of the $\varepsilon_i$. 

Both examples can be set in the more general framework of independent observations with 
a distribution that may vary with $i$ and that we shall now describe more precisely. We 
observe $n$ independent random variables $X_{1},\ldots,X_{n}$ each $X_i$ with an 
unknown distribution $P_i$ on a measurable space ${(\X,\A)}$ and our aim is to use the 
vector $\gX=(X_{1},\ldots,X_{n})$ of observations to estimate their joint distribution 
$\gP=\bigotimes_{i=1}^nP_i$, that is to find a random approximation $\widehat{\gP}(\gX)=
\bigotimes_{i=1}^n\widehat{P}_i(\gX)$ of $\gP$ based on the observed variables $X_i$. 
To measure the quality of the approximation of $\gP$ by 
$\widehat{\gP}$ we need a distance on the set of product measures on $\X^n$. It is known 
from Le~Cam's work --- see for instance Le~Cam~\citeyearpar{MR856411} and Le~Cam and Yang~\citeyearpar{MR1066869} --- 
that a very convenient one is that (here denoted by $\gh$) derived from the 
Hellinger distance $h$ and introduced in Le~Cam~\citeyearpar{MR0395005}:
\[
\gh^2\left(\bigotimes_{i=1}^nP_i,\bigotimes_{i=1}^nQ_i\right)=\sum_{i=1}^nh^2(P_i,Q_i)=
\frac{1}{2}\sum_{i=1}^n\int\left(\sqrt{dP_i}-\sqrt{dQ_i}\right)^2.
\]
We recall that the Hellinger distance $h$ is the bounded distance on the set of all probabilities 
on $\X$ given by
\begin{equation}
h^2(R,T)=\frac{1}{2}\int\left(\sqrt{dR/d\mu}-\sqrt{dT/d\mu}\right)^2d\mu\le1,
\label{Eq-Hel9}
\end{equation}
where $\mu$ is an arbitrary positive measure which dominates both $R$ and $T$, the result 
being independent of the choice of $\mu$. This is why one writes symbolically $h^2(R,T)=
(1/2)\int(\sqrt{dR}-\sqrt{dT})^2$. 

The distance $\gh$ between product measures provides an indicator of the quality of an 
estimator $\widehat{\gP}$ of $\gP$ via their distance $\gh(\widehat{\gP},\gP)$ and our aim 
is to design estimators $\widehat{\gP}$ such that, with a probability close to one, 
$\gh(\widehat{\gP},\gP)$ is as small as possible. We shall in particular 
often measure the quality of $\widehat{\gP}$ by its {\em quadratic risk}\, $\mathbb{E}_{\gP}
[\gh^2(\widehat{\gP}(\gX),\gP)]$ which is a bounded function of 
$\gP$ since $\gh\le\sqrt{n}$, the notation $\mathbb{E}_{\gP}$ meaning that $\gX$ has the 
distribution $\gP$. As previously mentioned, we shall put a special emphasis on regression 
frameworks on $\mathbb{R}^n$ and on the particularily simple example of a constant 
function $f$, which corresponds to a {\em translation family} for i.i.d.\ observations.

\subsection{Translation families\label{I0}}
The simplest case of a general regression framework $Z_i=f(W_i)+\varepsilon_i$ occurs 
when the function $f$ is constant and equal to $\theta\in\Theta\subset\mathbb{R}$. It also 
corresponds to fixed design regression with $f_i=\theta$ for all $i$, in which case the 
observations $X_i$ are i.i.d.\ with density $p(\cdot-\theta)$ and distribution $P_{\theta}$, 
$\gP=\gP_{\theta}=P_{\theta}^{\otimes n}$ and $\gh^2(\gP_{\theta},\gP_{\theta'})=
nh^2(P_\theta,P_{\theta'})$ for $\theta,\theta' \in \Theta$. When $p$ is known, this is a 
parametric family with a single translation parameter $\theta$ for which the problem is 
to find an estimator $\widehat{\theta}_n=\widehat{\theta}_n(\gX)$ for $\theta$ so that 
$\widehat{\gP}=P_{\widehat{\theta}_n}^{\otimes n}$.
For all densities $p$ and $\Theta$ an interval of positive length, it follows from 
Le~Cam \citeyearpar{MR0334381} that, for all $\pi\in(0,1/2)$ and some constant $c(\pi)$ 
depending on $\pi$,   
\[
\sup_{\theta\in\Theta}\gP_{\theta}\left[h\left(P_{\widehat{\theta}_n},P_{\theta}\right)\ge c(\pi)n^{-1/2}\right]\ge\pi,\quad\mbox{whatever the estimator }\widehat{\theta}_n.
\]
When there exists a local relationship between the parameter distance and the corresponding
Hellinger distance of the form 
\[
a|\theta-\theta'|^\alpha\le h(P_\theta,P_{\theta'})\le A|\theta-\theta'|^\alpha\quad\mbox{for }
|\theta-\theta'|\le b\quad\mbox{with}\quad a,A,b>0,\;\;0<\alpha\le1,
\]
one cannot expect to build an estimator $\widehat{\theta}_n$ with convergence rate to 
the true $\theta$ better than $n^{-1/(2\alpha)}$ which we shall call the {\em optimal rate}.

In the past, various procedures have been considered for estimating $\theta$. Let us assume that 
the density $p$ is symmetric and have a look at three among the most classical ones:

i)  the {\em empirical mean} $\overline{X}_n=n^{-1}\sum_{i=1}^nX_i$ which is the minimizer with respect to $\theta$ of the squared empirical error $\sum_{i=1}^n(X_i-\theta)^2$ (least squares estimator);

ii) the {\em empirical median} $X_{(n/2)}$ or $X_{((n+1)/2)}$ according to the parity of $n$, where $X_{(i)}$ denotes the $i$-th element of the set $\{X_1,\cdots,X_n\}$ in ascending order;

iii) the {\em maximum likelihood estimator} (MLE for short) which maximizes the {\em likelihood function} $\theta\mapsto\prod_{i=1}^np(X_i-\theta)$.

Unfortunately, none of them is really satisfactory in the sense that each one may behave quite 
poorly for some densities $p$ as shown by the following examples. The empirical mean is only 
suitable when $\Bbb{E}[\varepsilon_i]=0$ and $\Bbb{E}[\varepsilon_i^2]<+\infty$ as in the
Gaussian case: $p(x)=\left(2\pi\sigma^2\right)^{-1/2}\exp\left[x^2/\left(2\sigma^2\right)\right]$ 
where it reaches the optimal rate $n^{-1/2}$. But it fails miserably when the density $p$ is 
Cauchy --- $p(x)=\left[\pi\left(1+x^2\right)\right]^{-1}$ --- in which case one could use instead 
the empirical median and get again the optimal rate $n^{-1/2}$. When the density 
$p=(1/2)\1_{[-1,1]}$ is uniform, both methods provide the rate $n^{-1/2}$ while the MLE 
converges at the optimal rate $n^{-1}$. It also provides the rate $n^{-1/2}$ for our two previous 
examples but, if $p(x)=(1/4)|x|^{-1/2}\1_{[-1,1]}(x)$ the likelihood function is unbounded and 
the MLE does not even exist! In this case the empirical mean and median do exist but none 
of them provides the optimal rate which is, in this last case, $n^{-2}$. It follows that none of 
the three methods reaches the optimal rate for all possible densities $p$. Actually, each $p$ 
requires the choice of a specific method depending on the characteristics of $p$.

There is, moreover, an additional problem which is due to the fact that our translation family 
is actually only a {\em model}, that is an approximation of the truth. This means that we 
pretend that our observations $X_i$ are i.i.d.\ with density $p(\cdot-\theta)$ and joint 
distribution $P_{\theta}^{\otimes n}$ for some unknown parameter $\theta\in\Theta$, 
therefore dealing with the statistical model 
\begin{equation}
\overline{S}=\left\{\gP_\theta=P_{\theta}^{\otimes n},\;\theta\in\Theta\right\}
\qquad\mbox{with}\qquad(dP_{\theta}/d\mu)=p(\cdot-\theta),
\label{Eq-trans}
\end{equation}
although the true distribution is $\gP=\bigotimes_{i=1}^nP_i$. Of course, if the distance 
$\inf_{\theta\in\Theta}\gh\left(\gP,P_{\theta}^{\otimes n}\right)$ from $\gP$ to our model is 
large, there is no hope to get a good estimation of $\gP$ by some $P_{\widehat{\theta}_n}
^{\otimes n}$. But when our model provides a reasonable approximation of $\gP$, 
one would like to derive an estimator $P_{\widehat{\theta}_n}^{\otimes n}$ which remains 
close to $\gP$. This is the so-called problem of {\em robustness} of estimators. It is known, 
for instance, that the replacement of the true $p$ by an approximation $q$, even if 
$h(p\cdot\mu,q\cdot\mu)$ is small, may considerably affect the value of the corresponding 
moments and makes methods based on moments estimation fail. The same phenomenon 
may happen with the MLE which should  be used with great caution as emphasized by 
Le~Cam~\citeyearpar{Lecam-MLE}.

The situation does not improve when we consider more general regression problems and 
it is well-known that both the method of least squares (the multidimensional analogue of 
the empirical mean) and the MLE suffer from the same weaknesses as for translation families. 

\subsection{What would be desirable?\label{I2}}
In view of the conclusions of the previous section, a natural question arises: is it possible to 
build an estimator that can be simultaneously optimal (in some suitable sense) when the 
model is true and also {\em robust}, that is not too sensitive to small differences between 
the true distribution and the chosen model? There are actually two distinct problems to be 
solved simultaneously: one of optimality and one of robustness. 

Let us first focus on optimality. We recall that we want to estimate an unknown distribution 
$\gP$ on $\X^n$ which belongs to the set $\PP$ of all product distributions 
$\bigotimes_{i=1}^nP_i$, that is of all possible joint distributions for the independent random 
variables $X_i$ and that we shall measure the quality of an estimator $\widehat{\gP}=
\bigotimes_{i=1}^n\widehat{P}_i\in\PP$ by its quadratic risk $\mathbb{E}_{\gP}
[\gh^2(\widehat{\gP},\gP)]\le n$. Most of the time we shall assume some prior information 
on $\gP$, for instance that it derives from some regression framework. We shall express this 
prior information by assuming that $\gP=\gP_{\gs}$ for some unknown parameter $\gs$ 
belonging to some given parameter set $\scr{S}$, often some subset of a linear space, 
either Euclidean (finite dimensional) or functional (infinite dimensional), with a one-to-one 
parametrization $\gs\mapsto\gP_{\gs}$. It follows that the metric $\gh$ on $\{\gP_{\gs},
\gs\in\scr{S}\}$ can be transfered to $\scr{S}$ and we shall write indifferently 
$\gh(\gP_{\gt},\gP_{\gu})$ or $\gh(\gt,\gu)$. Unfortunately, in many situations, the set
$\{\gP_{\gs}, \gs\in\scr{S}\}$ is too large for the existence of an estimator 
$\widehat{\gP}=\gP_{\widehat\gs}$ such that
\[
\sup_{\gs\in\scr{S}}\E\cro{\gh^{2}(\gP_{\gs},\gP_{\widehat\gs})}=\sup_{\gs\in\scr{S}}
\E\cro{\gh^{2}(\gs,\widehat{\gs})}\quad\left(\mbox{with }\:\E=\Bbb{E}_{\gP_\gs}\right)
\]
be substantially smaller than its maximal value $n$, so that the maximal quadratic risk does 
not provide a useful information on the quality of $\widehat{\gs}$ unless one focuses 
on some specific values of $\gs\in\scr{S}$. 

Following Birg\'e~\citeyearpar{MR2219712} and Baraud~\citeyearpar{MR2834722} 
but also much earlier contributions including the sieves method of Grenander~
\citeyearpar{MR599175} or the ones in Birg\'e and Massart~
\citeyearpar{MR1462939,MR1653272} and Barron, 
Birg\'e and Massart~\citeyearpar{MR1679028} among many others, 
our approach in this paper is based on {\em models}, that is subsets $\overline{S}$ of 
$\mathscr{S}$ of moderate size in order that there exists an estimator $\widehat{\gs}$ 
such that $\sup_{\gs\in \overline{S}}\E\cro{\gh^{2}(\gs,\widehat{\gs})}\ll n$. 
This means that we do as if $\gs$ did belong to $\overline{S}$ and design estimators 
$\widehat\gs(X_{1},\ldots,X_{n})$ with values in $\overline{S}$, although we do not 
necessarily assume that this is true. With this approach, a good indicator of the quality of 
the estimator $\widehat{\gs}$ under the assumption that $\gs$ truely belongs to 
$\overline{S}$ is its maximal risk $\sup_{\gs\in\overline{S}}\E\cro{\gh^{2}(\gs,\widehat\gs)}$,
as compared to the so-called {\em minimax risk} over $\overline{S}$, $R_M(\overline{S})=
\inf_{\widehat{\gs}}\sup_{\gs\in \overline{S}}\E[\gh^2(\gs,\widehat{\gs})]$ where the
infimum runs over all possible estimators $\widehat{\gs}$. An
{\em approximately optimal} estimation procedure $\widetilde{\gs}$ should satisfy 
\[
\sup_{\gs\in \overline{S}}\E
\cro{\gh^{2}(\gs,\widetilde\gs)}\le C_0R_M(\overline{S}),
\]
where $C_0$ (as well as all $C_j$'s with $j\in\mathbb{N}$ that we shall introduce below) 
denotes a positive universal constant (independent of $n$ and $\overline{S}$ and, ideally, 
not large).

Nevertheless, since there is no way to check precisely whether the true parameter value 
$\gs$ does actually belong to $\overline{S}$, one would like that the previous bound 
remains approximately true if the model $\overline{S}$ is slightly misspecified, that is when 
$\gs\not\in\overline{S}$ but $\gh(\gs,\overline{S})=\inf_{\gt\in\overline{S}}\gh(\gs,\gt)$ is small, 
in which case the estimator is robust. It is clear that whatever the estimator $\widehat{\gs}
\in \overline{S}$, $\E\left[\gh^2\left(\gs,\widehat{\gs}\right)\right]\ge\inf_{\gt\in \overline{S}}
\gh^2\left(\gs,\gt\right)$. In view of this fact and the definition of the minimax risk, an 
approximately optimal and robust estimator $\widetilde{\gs}$ based on the model 
$\overline{S}$ should satisfy
\begin{equation}
\E\cro{\gh^{2}(\gs,\widetilde\gs)}\le C_1\max\ac{R_M(\overline{S}),\ \inf_{\gt\in \overline{S}}
\gh^2\left(\gs,\gt\right)}\quad\mbox{for all }\gs\in\scr{S}.
\label{Eq-Opt}
\end{equation}
As already mentioned, most popular methods of estimation, in particular those based 
on moments estimation or the MLE, are not robust and minimum contrast estimators based 
on the $\IL_{2}$-contrast as well. As to classical methods which do possess some robustness 
properties with respect to misspecification, like methods based on the $\IL_{1}$-contrast 
in regression or quantile estimation, they may unfortunately lead to sub-optimal rates of estimation.

\subsection{The search for robust and optimal  estimators\label{I1}}
Even for the simple case of a translation parameter, finding robust estimators is definitely 
not obvious. An old result in this direction is from P. Huber~\citeyearpar{MR0161415}.
An important research activity about robustness developed in the 60's and 70's resulting 
in a large number of publications. For a summary, we refer the interested reader to 
Huber~\citeyearpar{MR606374}.

Attempts to design ``optimal" procedures of estimation in various settings have been 
made by Le~Cam \citeyearpar{MR0334381,MR0395005}, 
Birg\'e~\citeyearpar{MR722129,MR2219712}, Yang and Barron~
\citeyearpar{MR1742500} or Baraud~\citeyearpar{MR2834722} and the construction 
that we shall present here is in the line of these previous papers. Actually, the problem 
of estimating $\theta$ in the translation model as well as many other problems in density 
estimation can essentially be solved, modulo some weak assumptions, by using the 
methods developed in these papers. 

Things become more delicate when we turn to the regression framework and, more generally, 
to estimating the distribution of independent but not necessarily i.i.d.\ distributions, for which 
the number of unknown parameters $f_i$, $1\le i\le n$, is a priori equal to the number of 
observations. 

\subsubsection{T-estimators on a model\label{I1a}}
Birg\'e~\citeyearpar{MR2219712}, following ideas from Le~Cam~\citeyearpar{MR0334381,
MR0395005} and generalizing earlier constructions of Birg\'e~\citeyearpar{MR722129,
MR762855}, derived a general procedure for 
building new estimators (called T-estimators) that satisfy (\ref{Eq-Opt}) under some 
compactness assumptions on the model $\overline{S}$. The idea is first to build a finite 
discretization $S_\eta$ at scale $\eta$ of $\overline{S}$ (with respect to $\gh$) and then 
an estimator $\widetilde\gs$ with values in $S_\eta$, based on tests between balls 
in the metric space $(\scr{S},\gh)$, centered at the points of $S_\eta$. These are actually 
robust tests between the points of $S_\eta$, as described for instance in 
Birg\'e~\citeyearpar{Robusttests}, so that the estimator inherits from these tests 
its robustness properties. The performance of the estimator is driven by a function 
from $(0,+\infty)$ into $[1/2, +\infty]$, called the {\em metric dimension} 
$\widetilde{D}_{\overline{S}}$ of $\overline{S}$, that characterizes the number of points of 
$S_\eta$ that are contained in balls of radius $x\eta$ for $x\ge2$. As shown in Birg\'e~\citeyearpar{MR2219712}, with a convenient choice of $S_\eta$ the T-estimator 
$\widetilde\gs$ satisfies an analogue of (\ref{Eq-Opt}), namely that, for all $\gs\in\scr{S}$,
\begin{equation}
\E\cro{\gh^{2}(\gs,\widetilde\gs)}\le C_3\max\{\eta^2,
\inf_{\gt\in \overline{S}}\gh^2(\gs,\gt)\}\quad\mbox{if }\;\eta\mbox{ satisfies }
\eta^2\ge C_2\widetilde{D}_{\overline{S}}(\eta).
\label{Eq-Opt1}
\end{equation}
This implies in particular that $R_M(\overline{S})\le C_3\eta^2$ and that (\ref{Eq-Opt}) holds 
provided that $R_M(\overline{S})\ge C_4\eta^2>0$. In particular, if the function $\widetilde{D}_{\overline{S}}$ is bounded by the constant $\overline{D}_{\overline{S}}$ (the {\em finite 
dimensional case}), one can set $\eta^{2}=C_2\overline{D}_{\overline{S}}$ and the minimax 
risk $R_M(\overline{S})$ is bounded by $C_5\overline{D}_{\overline{S}}$.

\subsubsection{Several models\label{I1c}} 
As we can immediately see from (\ref{Eq-Opt1}) the choice of the model is crucial for the 
performance of $\widetilde{\gs}$ at a given parameter $\gs$. A good model should have 
a small dimension and be close to $\gs$. Unfortunately, since $\gs$ is unknown, choosing 
a good model from scratch is possible only under rather precise information on $\gs$. 
The solution provided by T-estimators is to deal with a large family $\overline{\SS}$ of 
models $\overline{S}$ and extend the construction of T-estimators to the union of all models 
contained in $\overline{\SS}$. This is precisely what has been done in
Birg\'e~\citeyearpar{MR2219712}, resulting in the following risk bound in the case of 
models of finite dimension,
\[
\E\cro{\gh^{2}(\gs,\widetilde\gs)}\le C_6\inf_{\overline{S}\in\overline{\SS}}\max\{\overline{D}_{\overline{S}},\Delta(\overline{S}),\inf_{\gt\in \overline{S}}\gh^2(\gs,\gt)\},
\]
where the weight function $\Delta$ satisfies $\sum_{\overline{S}\in\overline{\SS}}\exp[-\Delta(\overline{S})]\le1$. As compared to (\ref{Eq-Opt1}) we see that we get the same risk 
bound as the one corresponding to the best model, apart from the extra $\Delta(\overline{S})$ 
term which describes the complexity of the family $\overline{\SS}$. The larger this family, the 
larger the weights $\Delta(\overline{S})$. For simple families, one can choose 
$\Delta(\overline{S})\le C_7\overline{D}_{\overline{S}}$ and only loose a constant factor 
as compared to (\ref{Eq-Opt1}). Otherwise there is some additional loss which is sometimes unavoidable. The advantage of model selection is that it allows to handle many models 
simultaneously with the hope that one of them will be quite suitable for the estimation of 
the unknown parameter $\gs$. For a detailed discussion about model selection, we refer 
the reader to Barron, Birg\'e and Massart~\citeyearpar{MR1679028} or Birg\'e and Massart~\citeyearpar{MR2288064}.

\subsubsection{A history of dimensions\label{I1b}} 
The notion of metric dimension $\widetilde{D}_{\overline{S}}$ actually applies to subsets 
of any metric space, not only to $(\scr S,\gh)$, and it is actually the right notion that is 
needed to control the performance of T-estimators. It is one possible way of measuring 
the massiveness of a model $\overline{S}$ but definitely not the only one. Others have 
been developed earlier, the simplest one being the ordinary dimension of a Euclidean 
space, but one can also mention Kolmogorov's entropy --- see Kolmogorov and Tikhomirov~\citeyearpar{MR0124720} --- among other possible notions. The fact that there is often 
some close relationship between the minimax risk over $\overline{S}$ and some notion of 
dimension of $\overline{S}$ has been known for a long time, the simplest example being 
the estimation of the mean of a Gaussian vector with identity covariance matrix when this 
mean is assumed to belong to a $D$-dimensional linear space $\overline{S}$. Similar 
results hold for parametric statistical estimation problems which are regular enough. Upper 
bounds for the minimax risk based on some earlier (more restrictive) version of metric 
dimension were developed by Le~Cam~\citeyearpar{MR0334381,MR0395005} 
and generalized by Birg\'e~\citeyearpar{MR722129} together with the connection to lower 
bounds previously developed by Ibragimov and Has'minskii ~\citeyearpar{Ibrag-Hasm}.
The performance of the MLE on a parameter set $\overline{S}$ may also be deduced from 
some suitable notion of dimension, namely entropy with bracketing --- see van de Geer~\citeyearpar{MR1324688} and Birg\'e and Massart~\citeyearpar{MR1240719} --- and the 
concentration of the posterior distribution in Bayesian frameworks as well --- see Ghosal, 
Gosh and van der Vaart~\citeyearpar{MR1790007} ---. 

The superiority of the notion of metric dimension is due to the fact that it is a weaker notion 
than entropy. For instance, the metric dimension of a Euclidean space is roughly equal to 
its ordinary dimension while its entropy is infinite. The entropy of a compact set automatically 
controls its metric dimension while the reciprocal is not true. Nevertheless it is not possible to 
characterize the minimax risk over a model $\overline{S}$ by its metric dimension and we do 
not know of any notion ${\mathcal D}_{\overline{S}}$ such that
\begin{equation}
c\varphi({\mathcal D}_{\overline{S}})\le R_M(\overline{S})\le 
C\varphi({\mathcal D}_{\overline{S}})\quad\mbox{with }0<c<C,
\label{Eq-Opt2}
\end{equation}
for some suitable function $\varphi$, at least under very mild assumptions on $\overline{S}$. 

\subsection{From T- to $\rho$-estimators\label{I4}}   
The construction of estimators from tests between balls centered on the points of some 
finite set, which is due to Le~Cam~\citeyearpar{MR0334381}, has been extended to 
countable sets and developed at length in the form of T-estimators by Birg\'e~
\citeyearpar{MR2219712}. Then, in an attempt to build a procedure for selecting estimators, 
Baraud~\citeyearpar{MR2834722} designed a new method which amounts to replacing 
tests between balls centered at points $\gt$ and $\gu$ in $\scr{S}$ by tests that tend to 
decide which of the two distances $\gh(\gt,\gs)$ or $\gh(\gu,\gs)$ is smaller, where $\gs$ 
denotes the true parameter. When applied to the discretized models $S_\eta$ used for the 
construction of T-estimators, Baraud's estimators can be viewed as a particular version of 
T-estimators, but this alternative construction allows to relax some of the assumptions 
needed for the use of T-estimators. The procedure has been taken back later by 
Sart~\citeyearpar{Sart2014,refId0} in a 
context of dependent data. A modification of Baraud's construction, following an idea of 
Sart, finally led to the construction of $\rho$-estimators that we shall present here. 
It is an attempt not only to answer the above mentioned question of Oleg Lepski but also to 
solve this search for a ``universal" estimator, at least in the case of independent observations.

In order to give a brief account of our new procedure, let us consider the problem of density 
estimation for i.i.d.\ observations $X_1,\cdots,X_n$ with values in a measurable space $(\X,\A)$, 
in which case $\gs=(s,\ldots,s)$ where $s$ denotes the common density of the $X_i$ with 
respect to some dominating measure $\mu$ so that $\gP_{\gs}=(s\cdot\mu)^{\otimes n}$ and, 
for $\gt,\gu\in\scr{S}$,
\begin{equation}
\frac{1}{n}\gh^2(\gt,\gu)=h^2(t,u)={1\over 2}\int_{\X}\pa{\sqrt{t}-\sqrt{u}}^{2}d\mu=
1-\int_{\X}\sqrt{tu}\,d\mu=1-\rho(t,u),
\label{Eq-h-rho}
\end{equation}
where $\rho(t,u)$ is called the {\em Hellinger affinity} between $t$ and $u$. Let us start with 
a model $\overline{S}$ and two distinct densities $t_{0}$ and $t_{1}$ in $\overline{S}$ (that 
may be different from the true one $s$). The difference $\rho(s,t_{1})-\rho(s,t_{0})=h^2(s,t_0)-
h^2(s,t_1)$ tells us which of the points $t_{0}$ or $t_{1}$ is closer to $s$ with respect to the 
Hellinger distance. If we have at hand a good estimator $T_n(t_0, t_1)$ of $\rho(s,t_{1})-
\rho(s,t_{0})$, it can be used not only to decide which of $t_0$ and $t_1$ is closer to $s$ 
but also, considering $\sup_{t\in \overline{S}}T_n(t_0, t)$ as an estimator of
\[
T(t_0)=\sup_{t\in \overline{S}}\cro{h^2(s,t_0)-h^2(s,t)}=h^2(s,t_0)-\inf_{t\in \overline{S}}h^2(s,t),
\]
to see whether $t_0$ is likely to be almost a closest point to $s$ in $\overline{S}$. Indeed, 
the smaller the quantity $T(t_0)$, the better $t_0$ as an approximation of $s$ in $\overline{S}$. 
It seems therefore natural to try to minimize $\sup_{t\in \overline{S}}T_n(t_0, t)$ with respect 
to $t_0\in\overline{S}$ in order to derive a good estimator of $s$ within $\overline{S}$ provided that 
$T_n(t_0, t_1)$ is close enough to $\rho(s,t_{1})-\rho(s,t_{0})$ for all $t_0,t_1$. These are, 
roughly speaking, the ideas behind the construction of what we shall call a {\it $\rho$-estimator} 
since it is based on a suitable estimation of the Hellinger affinities between the true density 
and the points in the model. 

While the study of T-estimators mainly relies on combinatorial arguments, the study of $\rho$-estimators involves empirical processes techniques for which a lot of results are known. 

\subsection{What's new here?\label{I5}}
Our paper, although initially motivated by Oleg Lepski's question and the will of finding a 
generic treatment of fixed-design regression under very weak assumptions (in particular 
no boundedness restrictions and no moment conditions), also results in both an improvement 
over T-estimators and a path in the direction of solving the problem summarized by 
(\ref{Eq-Opt2}).

It happens, as already shown in Birg\'e~\citeyearpar{MR722129}, that in many situations, 
a lower bound on the minimax risk $R_M(\overline{S})$ over $\overline{S}$, as defined in 
Section~\ref{I2}, of the form 
\[
R_M(\overline{S})\ge C_4\eta^2\quad\mbox{for some }\eta^2\ge 
C_8\widetilde{D}_{\overline{S}}(\eta)
\]
actually holds, providing a reciprocal to (\ref{Eq-Opt1}); unfortunately this is not always the case. 
There are situations for which $\widetilde{D}_{\overline{S}}(\eta)=+\infty$ for all $\eta>0$ and 
this  typically happens when the diameter of $\overline{S}$ is $\sqrt{n}$, in particular when 
$\overline{S}$ is the translation model described by (\ref{Eq-trans}) with parameter space 
$\Theta=\R$. The use of a T-estimator therefore requires that $\theta$ belong to some known 
interval $[a, a+M]$ and its risk bound would unfortunately deteriorate as $M$ becomes larger. 
This difficulty can be fixed via the use of a preliminary estimator, a quantile estimator for 
instance, allowing to locate the parameter $\theta$ approximately but this solution does 
not extend to the regression framework. There are also cases with i.i.d.\ random variables 
where the quantity $\eta^{-2}R_M(\overline{S})$ with $\eta^2= 
C_2\widetilde{D}_{\overline{S}}(\eta)$ tends to zero when the number of observations 
tends to infinity which means that the risk bound (\ref{Eq-Opt1}) derived from the metric 
dimension has not the right order of magnitude. We shall even exhibit  in Section~\ref{E2} an example of a statistical model for which the metric dimension is infinite, hence the construction of a T-estimator is impossible, while a $\rho$-estimator reaches the optimal rate of convergence, namely $1/\sqrt{n}$, with respect to the Hellinger loss.

While $\rho$-estimators retain all the nice properties of T-estimators, in particular their 
robustness, their risk is bounded via new notions of dimensions which improve the one 
of metric dimension as shown by Corollary~\ref{C-main00} below. These dimensions can actually be suitably controlled for many non-compact models which is an essential property  for the statistical problems we want to solve. 

An additional attractive feature of $\rho$-estimators in density estimation lies in the fact that 
when $n$ is large enough, they recover the usual MLE at least when the model is parametric, 
regular enough and contains the true density to estimate. Some simulations developed by Sart  
show that this occurs even for moderate values of $n$. Another connection with the MLE lies 
in the fact that the risk bounds obtained for the MLE under bracketing entropy assumptions 
are still valid (up to possible numerical constants) for bounding the risks of $\rho$-estimators. 

In the regression framework, our procedure  improves upon the classical least squares from 
numerous aspects. First of all, we can deal with errors bearing no finite moments of any order 
such as the Cauchy distribution while the least squares approach cannot. Besides, we can 
handle various types of errors possibly leading to faster rates of estimation of the parameters 
than the ones reached by the least squares. Even in the case of 
the simple linear regression, our method may estimate at a much faster rate than the least 
squares, when the errors are uniformly distributed on $[-1,1]$ for instance. Finally, our procedure guarantees 
robustness properties for the resulting estimator that the use of least squares does not. 

Our procedure also substantially improves upon T-estimation. A first drawback of T-estimation lies in the fact that it requires that the supremum norm of the regression function be known. When the design is random, T-estimation also requires that its distribution be known in order to achieve the properties of robustness and optimality described above. These two assumptions are unfortunately rather restrictive. In contrast, although initially conceived to handle complicated situations of regression with fixed 
design, our procedure also allows to deal with various random design problems and therefore handles the whole regression framework in much greater generality. More precisely, $\rho$-estimation does not require any knowledge about a possible bound on the regression function and about the distribution of the design, at least when the errors are modelled as symmetric. These two properties illustrate the superiority of $\rho$-estimation over T-estimation and we are not aware of any statistical procedure that leads to a rate-optimal estimator (up to a possible 
logarithmic factor)  when the distribution of the errors is only assumed to belong to a large 
family of possible ones, including the Gaussian, Cauchy, uniform, etc.

\subsection{Connection with statistical learning theory}
The core of the proof of our main theorem relies on the control of the supremum of an empirical process, indexed by a bounded class of functions, over some vicinity of a specific element of this class. The same type of control is also needed to deal with empirical risk minimization, as explained in great details in the very nice paper by Koltchinskii~\citeyearpar{MR2329442}, and similar tools are used to handle both problems. Talagrand's concentration inequality allows to reduce the control of the supremum of the empirical process to that of  its expectation and universal entropy is then used to bound this expectation. Our use of VC-classes for bounding the universal entropy is somewhat analogous to that of Koltchinskii~\citeyearpar{MR2329442}. As a natural consequence of this parallelism, the notion of dimension that we introduce to control the risk of $\rho$-estimators is quite similar to the notion of local Rademacher complexity used in Koltchinskii~\citeyearpar{MR2329442} to control the performance of empirical risk minimization. This problem and, more specifically, that of binary classification was also considered and treated with similar tools (Talagrand's theorem, universal entropy and VC-classes) in Massart and N\'ed\'elec~\citeyearpar{MR2291502}. 
 
\subsection{Organization of the paper}
We present in Section~\ref{A} three statistical settings to which our procedure can be applied 
and the basic ideas underlying our approach in Section~\ref{B}. The construction of the 
estimator and the main results about its performance on a single model can be found in 
Section~\ref{C}. In Section~\ref{S-CMLE} we show that, in favourable cases, the MLE is a 
particular case of $\rho$-estimator. We also show that the assumptions which are used to 
analyze the performance of the MLE in favourable situations can also be used to derive 
similar risk bounds for $\rho$-estimators. In Section~\ref{E}, we illustrate the performance of 
$\rho$-estimators in the regression setting (with either fixed or random design) and provide 
an example for which their risk remains under control in a situation where the metric 
dimension of the model can be made arbitrary large and even infinite. In Section~\ref{MS}, 
we consider the problem of model selection. We establish there an oracle-type inequality 
and provide an application in view of estimating a regression function when the distribution 
of the errors belongs to a large class of densities including Laplace, Gaussian and uniform 
among others. We provide an annex on VC-subgraph classes in Section~\ref{S} since 
models $\overline{S}$ of these types play a special role in our results. Finally, Section~\ref{P} 
is devoted to the proofs.

\section{The statistical setting and examples\label{A}}

\subsection{Main notations and conventions\label{A7}}
In the sequel we shall use the following notations and conventions.  We set $\log_{+}x=\max\{\log x,0\}$
for $x>0$ and $\log_{+}0=0$. For $x,y\in\R$, $x\wedge y$ and $x\vee y$ denote $\min\{x,y\}$ 
and $\max\{x,y\}$ respectively, $\delta_x$ denotes the Dirac measure at point $x$ and $|A|$ 
the cardinality of the set $A$. Except if otherwise specified (in Section~\ref{B} below), we 
shall use the conventions $\sup\varnothing=0$, $0/0=1$, $0\times (+\infty)=0$ and 
$x/0=+\infty$ for all $x>0$. The word {\it countable} always means {\it finite or countable}. 
Throughout the paper, $C,C',\ldots$ denote positive {\it numerical positive} constants that 
may vary from line to line. The notations $C(\cdot), C'(\cdot),\ldots$ mean that $C,C',\ldots$ 
are  positive functions depending on the argument specified in the parenthesis (when the 
number of arguments is too large, the dependency is specified in the text).  We shall also 
often use the fact that
\begin{equation}
(x+y)^2\le(1+\alpha)x^2+\left(1+\alpha^{-1}\right)y^2=(1+\alpha)\left(x^2+\alpha^{-1}y^2\right)
\quad\mbox{for all }\alpha>0.
\label{Eq-square}
\end{equation}
Our definitions and results will actually involve a number of numerical constants. In order to avoid complicated formulas, we shall give specific names to the numerical constants that will be systematically used in the sequel.
\begin{equation}\left\{
\begin{array}{lll}  \displaystyle{c_0={1\over8}\pa{1-{1\over\sqrt{2}}}=\frac{1}{8\left(2+\sqrt{2}\right)};\quad\; c_{1}=2\left(7+4\sqrt{2}\right);\quad\; c'_1=2(c_1-1);} \\ 
\displaystyle{c_2=1+{1\over \sqrt{2}}=\frac{\sqrt{2}+1}{\sqrt{2}}}={1\over 16c_{0}};\quad\;\kappa=357;\quad\; c_{3}={8\kappa c_{2}};\quad\; c_4=2.5c_3.
\end{array}\right.
\label{Eq-cons1}
\end{equation}

\subsection{The general statistical setting\label{A0}}
Let $\gX=(X_{1},\ldots,X_{n})$ be a vector of independent random variables with values in a 
product of measured spaces $(\prod_{i=1}^{n}\X_{i},\bigotimes_{i=1}^{n}\A_{i},\bigotimes_{i=1}^{n}
\mu_{i})$. We assume that for each $i$, $X_{i}$ admits a density $s_{i}$ with respect to $\mu_{i}$ 
and our aim is to estimate $\gs=(s_{1},\ldots,s_{n})$ from the observation of $\gX$. To avoid 
trivialities, we shall always assume that $n\ge3$. We shall emphasize the dependence of the 
distribution of $\gX$ with respect to the unknown parameter $\gs$ by writing $\P_{\gs}[\gX\in A]$ 
for a measurable set $A$ and $\E[g(\gX)]$ for an integrable function $g$. 

On the measured space $(\X_{i},\A_{i},\mu_{i})$ we consider the set $\LL'_i$ of all measurable
real-valued functions $u$ such that $\int_{\X_{i}}|u|\,d\mu_{i}<+\infty$ and the subset $\LL_i$ of
$\LL'_i$ of all probability densities with respect to $\mu_{i}$, that is non-negative measurable
functions $u$ on $\pa{\X_{i},\A_{i}}$ such that $\int_{\X_{i}}u\,d\mu_{i}=1$. We equip $\LL_i$ with the
Hellinger pseudometric $h$ given, according to (\ref{Eq-Hel9}), by 
\[
h^2(u,u')={1\over 2}\int_{\X_{i}}\pa{\sqrt{u}-\sqrt{u'}}^2d\mu_{i}\quad\mbox{for all }u,u'\in\LL_i.
\]
Note that $h$ is only a pseudometric (symmetric and satisfying the triangular inequality)
since $h(u,u')=0$ if $u\ne u'$ but $u=u'$ $\mu$-a.e., although $h$ is a genuine distance on
the corresponding probability space since $h(u\cdot\mu,u'\cdot\mu)$ implies that $u\cdot\mu=
u'\cdot\mu$. In particular, $s_{i}$ may be any element of $\LL_{i}$ such that the distribution of
$X_{i}$ can be written $s_{i}\cdot\mu_{i}$.

We define $\LL_0$ as the product space $\prod_{i=1}^{n}\LL_i$, call the elements of 
$\LL_0$ densities and equip it with the pseudometric $\gh$ given, by analogy with
(\ref{Eq-h-rho}) and following Le~Cam~\citeyearpar{MR0395005}, by
\[
\gh^{2}(\gt,\gt')={1\over 2}\int\pa{\sqrt{\gt}-\sqrt{\gt'}}^2d\gmu=\sum_{i=1}^{n}h^2(t_i,t'_i)\le n
\quad\mbox{for }\gt,\gt'\in\LL_{0}.
\]
For simplicity we shall still call $h$ and $\gh$ {\em distances}, although they are only
pseudometrics on $\LL_{i}$, $1\le i\le n$ and $\LL_{0}$ respectively, call $(\LL_{0},\gh)$ a
pseudometric space and introduce the following definition.
%
\begin{defi}\label{def-ident}
A subset $S$ of $\LL_{0}$ is said to be identifiable if, when $\gu,\gu'\in S$ are such that
$\gu\ne\gu'$, then $\gh(\gu,\gu')>0$ or, equivalently, if $\gh$ is a genuine distance on $S$.
\end{defi}
For $\gu\in\LL_0$ and more generally for $\gu\in\prod_{i=1}^{n}\LL_i'$, we shall set
\[
\gu(\gX)=\sum_{i=1}^nu_i(X_i)\qquad\mbox{and}\qquad\int\gu\,d\gmu=
\sum_{i=1}^n\int_{\X_i}u_i\,d\mu_i.
\]
Hereafter, we shall deal with estimators with values in $\LL_{0}$ and measure their performances by
the risk induced by the loss function $\gh^2$. For simplicity we shall also call $\gh$ the Hellinger
distance and define the Hellinger affinity $\grho$ between two elements $\gt$ and $\gt'$ of $\LL_0$
as
\[
\grho(\gt,\gt')=\int\sqrt{\gt\gt'}\,d\gmu=
\sum_{i=1}^{n}\int_{\X_{i}}\sqrt{t_it'_i}\,d\mu_{i}=n-\gh^{2}(\gt,\gt')\ge0.
\]
For $\gt\in\LL_{0}$ and $y>0$, we shall denote by $\B(\gt, y)$ the closed ball of center $\gt$ and 
radius $y$ in the pseudometric space $(\LL_0,\gh)$ and, given some subset $S$ of $\LL_{0}$, by
$\B^{S}(\gs,y)=\B(\gs,y)\cap S$ the closed Hellinger ball in $S$ centered at $\gs$ with radius 
$y$. Finally, $\gh(\gt,S)=\inf_{\gt'\in S}\gh(\gt,\gt')$.

This general setting will allow us to deal in particular with the three following specific frameworks. 

\subsection{The density framework\label{A2}}
In this framework, we assume that the random variables $X_{1},\ldots,X_{n}$ are i.i.d.\ with 
values in a measured space $(\X,\A,\mu)$ and common density $s$ with respect to $\mu$, 
which leads to $\X_{i}=\X$, $\mu_{i}=\mu$ for all $i$ and $\gs=(s,\ldots,s)$. In this particular 
context, it will be convenient to identify a density $t$ on $(\X,\A,\mu)$ with the element 
$\gt=(t,t,\ldots,t)$ of $\LL_{0}$, which we shall do in the sequel.
Given two densities $t,t'$ on $(\X,\A,\mu)$, we have the relations
\[
\gh^{2}(\gt,\gt')=nh^{2}(t,t')\qquad\mbox{and}\qquad\grho(\gt,\gt')=n\rho(t,t').
\]
The risk of an estimator $\widetilde{\gs}=(\widetilde{s},\ldots,\widetilde{s})$ of $\gs$ is therefore 
$\E\left[\gh^2(\gs,\widetilde{\gs})\right]=n\E\left[h^2(s,\widetilde{s})\right]$.

\subsection{The homoscedastic regression framework with fixed design\label{A3}}
In this framework, we assume that the $X_{i}$ are real-valued random variables satisfying equations of the form
\[
X_{i}=f_{i}+\lambda\eps_{i}\quad\mbox{for }i=1,\ldots,n,\quad\lambda>0,
\]
where the vector $\gf=(f_1,\ldots,f_n)$ belongs to $\R^n$, the $\eps_{i}$ are real-valued 
i.i.d.\ random variables with density $p$ with respect to the {\em Lebesgue measure} 
$\mu$ on $(\R,{\mathcal B}(\R))$, ${\mathcal B}(A)$ denoting the Borel $\sigma$-algebra on 
the topological space $A$. It follows that the density of $X_i$ is $s_i(x)=\lambda^{-1}
p\left(\lambda^{-1}(x-f_i)\right)$ so that estimating $s_i$ amounts to estimating $\lambda$, $p$ 
and $f_i$. Our aim is therefore to estimate $\gf$, $p$ and $\lambda$ from the observation of 
$X_{1},\ldots,X_{n}$. To deal with this framework, it will be convenient to introduce the 
following notations: for $f$ and $x$ in $\R$, $\lambda$ in $\R_+\setminus\{0\}$, $\gf$ and 
$\gx=(x_1,\ldots,x_n)$ in $\R^n$, we set
\begin{equation}
p_{f,\lambda}(x)={1\over \lambda}p\pa{x-f \over \lambda}\quad\mbox{for all }x\in\R;\qquad p_f=p_{f,1};
\label{Eq-trsc1}
\end{equation}
\begin{equation}
\gp_{\gf,\lambda}(\gx)=\pa{{1\over \lambda}p\pa{x_1-f_{1} \over \lambda},\ldots,
{1\over \lambda}p\pa{x_n-f_{n} \over \lambda}}\;\mbox{ for all }\gx\in\R^n\quad\mbox{and}\quad\gp_{\gf}=\gp_{\gf,1}.
\label{Eq-trsc2}
\end{equation}
It follows that $p_f(x)=p(x-f)$, $p_{0,\lambda}(x)=\lambda^{-1}p(x/\lambda)$, etc.
In this framework we take, for $i=1,\ldots,n$, $\X_{i}=\R$, $\A_{i}={\mathcal B}(\R)$, 
$\mu_{i}=\mu$ and $s_i=p_{f_i,\lambda}$ so that the density of $\gX$ is 
$\bigotimes_{i=1}^np_{f_i,\lambda}$ and therefore entirely determined by $\gp_{\gf,\lambda}$.

\subsection{The homoscedastic regression framework with random design\label{A4}}
Let $(W,Y)$ be a pair of random variables with values in $(\mathscr{W}\times
\R,\mathcal{W}\otimes\B(\R))$ linked by the relation 
\begin{equation}\label{Model-Reg-Rand}
Y=f(W)+\eps
\end{equation}
where $f$ is unknown in a set $\FF$ of measurable functions from $(\mathscr{W},\mathcal{W})$ 
into $(\R,\B(\R))$ and $\eps$ is an unobservable random variable, independent of $W$ and 
admitting a {\em known} (or {\em approximately known}) density $p$ with respect to the 
{\em Lebesgue measure} $\mu$. In contrast, the distribution $\nu$ of $W$ is possibly unknown.  

Our aim is to estimate $f$, or equivalently the conditional distribution of $Y$ given $W$, from 
the observation of $\gX=(X_1,\ldots,X_n)$ where the $X_i$ are i.i.d.\ with the same distribution 
on $(\mathscr{W}\times\R)$ as that of the pair $(W,Y)$ so that $\gX$ can be identified to 
$(\gW,\gY)$ with $\gW=(W_1,\ldots,W_n)$ and $\gY=(Y_1,\ldots,Y_n)$. Since the density of 
$(W,Y)$ with respect to the dominating measure $\nu\otimes\mu$ is $s(w,y)=p(y-f(w))=p_{f}(w,y)$, 
$p_{f}(W,\cdot)$ is  the conditional density of $Y$ given $W$ with respect to $\mu$. It is therefore 
natural to look for estimators of $s$ of the form $\widehat s=p_{\widehat f}$ where 
$\widehat f=\widehat f(\gW,\gY)$ is an estimator of $f$ which also provides an estimator 
$p_{\widehat{f}}(w,\cdot)$ of the conditional density of $\gY$ when $\gW=w$. At this stage, it is 
important to emphasize the fact that the construction of $\widehat{s}$ {\it should not involve $\nu$} 
in order to provide genuine estimators of $f$ and $p_f$. As we shall see in Section~\ref{E4}, 
the $\rho$-estimator of $s$ derived from our general method does satisfy this requirement. 

To evaluate the performance of $\widehat f(\gW,\gY)$, we use the risk 
$\E\left[\gh^2(\gs,{\widehat \gs})\right]$ of $\widehat{\gs}(\gW,\gY)$ or, equivalently, the risk of 
$\gp_{\widehat\gf}$ which writes
\[
\E\left[\gh^2(\mathbf{p_{f}},\mathbf{p_{\widehat f}})\right]=
n\E \left[h^{2}(p_{f},p_{\widehat f})\right]=
n\E\left[\int_{\mathscr{W}}h^{2}\left(p_{f}(w,\cdot),p_{\widehat{f}}(w,\cdot)\right)d\nu(w)\right].
\]

\section{Basic ideas underlying our approach\label{B}}

\subsection{The density framework\label{B1}}
The aim of this section is to present the basic  ideas and formulas underlying our approach. For the sake of simplicity, we shall restrict this introduction to the density framework described in Section~\ref{A2} where the observations $X_{1},\ldots,X_{n}$ are i.i.d.\ with an unknown density $s$ with respect to $\mu$.

Given two candidate densities $t,t'$ for $s$, one should  prefer $t'$ to $t$ if it is closer to $s$, that is,  if $h^{2}(s,t')$ is smaller than $h^{2}(s,t)$ or equivalently if $\rho(s,t')-\rho(s,t)>0$. Deciding whether $t'$ is preferable to $t$ amounts thus to estimating the difference $\rho(s,t')-\rho(s,t)$ in a suitable way. To do so, we start by an approximation of the affinity $\rho$. For two densities $t$ and $t'$, we set
\begin{equation}\label{def-rhoQ}
r=\frac{t+t'}{2}\qquad\mbox{and}\qquad
\varrho(s,t,t')={1\over 2}\cro{\rho(t,r)+\int_{\X}\sqrt{{t\over r}}s\,d\mu}< +\infty,
\end{equation}
using the special convention that $t/r=0$ when $t=t'=r=0$. It was proved in Proposition~1 of Baraud~\citeyearpar{MR2834722} that 
\begin{equation}\label{rhoM-approx}
0\le\varrho(s,t,t')-\rho(s,t)\le\cro{h^{2}(s,t)+h^{2}(s,t')}/\sqrt{2}.
\end{equation}
The important point about~\eref{rhoM-approx} lies in the fact that the constant $\sqrt{2}$ is larger than 1. This makes it possible to use the sign of the difference 
\[
T(s,t,t')=\varrho(s,t',t)-\varrho(s,t,t')
\]
as an alternative benchmark to find which of $t$ and $t'$ is closer to $s$ (up to a multiplicative constant). It actually follows from~\eref{rhoM-approx}, as shown in Corollary~1 in Baraud~\citeyearpar{MR2834722}, that
%
\begin{equation}
T(s,t,t')\le\left(1+{1\over \sqrt{2}}\right)\!h^{2}(s,t)-\left(1-{1\over \sqrt{2}}\right)\!h^{2}(s,t')
=c_2h^{2}(s,t)-8c_0h^{2}(s,t')
\label{eq-UbT}
\end{equation}
%
and
%
\begin{equation}
T(s,t,t')\ge\left(1-{1\over \sqrt{2}}\right)\!h^{2}(s,t)-\left(1+{1\over \sqrt{2}}\right)\!h^{2}(s,t')
=8c_0h^{2}(s,t)-c_2h^{2}(s,t').
\label{eq-LbT}
\end{equation}
%
Given some subset $S$ of $\LL_0$, (\ref{eq-UbT}) and the fact that $T(s,t,t)=0$ also imply that, 
for $t\in S$
\[
0\le\sup_{t'\in S}T(s,t,t')\le c_2h^{2}(s,t)-8c_0h^{2}(s,S)\quad\mbox{and}\quad
0\le\inf_{t\in S}\sup_{t'\in S}T(s,t,t')\le\sqrt{2}h^{2}(s,S).
\]
If $u\in S$ is such that 
\[
h^2(s,u)>(8c_0)^{-1}\left(c_2+\sqrt{2}\right)h^{2}(s,S)=\left(5+4\sqrt{2}\right)h^{2}(s,S),
\]
it follows from (\ref{eq-LbT}) that
\[
\sup_{t'\in S}T(s,u,t')\ge8c_0h^{2}(s,u)-c_2h^{2}(s,S)
>\left[\left(c_2+\sqrt{2}\right)-c_2\right]h^{2}(s,S)=\sqrt{2}h^{2}(s,S),
\]
and $u$ cannot be a minimizer of $t\mapsto\sup_{t'\in S}T(s,t,t')$. Hence any minimizer 
$\overline{s}$ of this function does satisfy $h^2(s,\overline{s})\le\left(5+4\sqrt{2}\right)h^2(s,S)
<11h^2(s,S)$. Therefore, minimizing over $S$ the function $t\mapsto\sup_{t'\in S}T(s ,t,t')$ leads 
to some point $\overline{s}\in S$ which, up to a factor smaller than 11, is the closest to $s$, 
that is, the best approximation of $s$ in $S$. In particular, if $s\in S$, $\overline{s}=s$.

Unfortunately, $T(s ,t,t')$ depends on $\varrho(s ,t,t')$ which depends on the unknown $s$.
Our interest for the quantity $\varrho(s,t,t')$ rather than $\rho(s,t)$ lies in the fact that the former 
can be estimated by its empirical counterpart, namely 
\begin{equation}\label{def-rhohat}
\varrho(\gX,t,t')={1\over 2n}\sum_{i=1}^{n}\cro{\rho(t,r)+\sqrt{{t\over r}(X_{i})}}
\quad\mbox{with }r=\frac{t+t'}{2},
\end{equation}
which is an unbiased estimator of $\varrho(s,t,t')$. A natural way of deciding which of the densities $t$ or $t'$ is the closest to $s$ is therefore to replace the unknown $T(s,t,t')$ by an unbiased estimator, namely the statistic 
\[
T(\gX,t,t')=\varrho(\gX,t',t)-\varrho(\gX,t,t').
\]
Note that 
\begin{equation}\label{T-psi}
T(\gX,t,t')=\frac{1}{2}\cro{\rho(t',r)-\rho(t,r)}+{1\over \sqrt{2}\,n}\sum_{i=1}^{n}\psi\pa{\sqrt{t'\over t}(X_{i})},
\end{equation}
where $\psi$ is the Lipschitz, increasing function from $[0,+\infty]$ to $[-1,1]$ 
(with Lipschitz constant not larger than 1.143) given by
\begin{equation}\label{def-psi}
\psi(u)= \sqrt{{1\over 1+u^{-2}}}-{\sqrt{1\over 1+u^{2}}}={u-1\over \sqrt{1+u^{2}}}\quad\mbox{for }u\in [0,+\infty)\ \ \mbox{and}\ \ \psi(+\infty)=1.
\end{equation}
Here we use the convention that $t'(X_i)/t(X_i)=1$ when $t(X_i)=t'(X_i)=0$ as indicated in Section~\ref{A7}. This convention is indeed consistent with the one we started from on the ratio $t/r$ since when $t(X_{i})=t'(X_{i})=r(X_{i})=0$ for some $i$,
\[
\cro{\sqrt{{t'\over r}(X_{i})}-\sqrt{{t\over r}(X_{i})}}=0-0=0\quad\mbox{and}\quad
\psi\pa{\sqrt{t'\over t}(X_{i})}=\psi\pa{0\over 0}=\psi(1)=0.
\]

Replacing the ``ideal" statistic $T(s,t,t')$ by its empirical counterpart $T(\gX,t,t')$ leads to an estimation error given by the process $Z(\gX,.,.)$ defined on $\LL_{0}^{2}$  by
\begin{eqnarray}
Z(\gX,t,t')&=&\sqrt{2}\left[T(\gX,t,t')-T(s ,t,t')\right]\nonumber\\
&=&\sqrt{2}\left(\rule{0mm}{3.7mm}
\cro{\varrho(\gX,t',t)-\varrho(s,t',t)}-\cro{\varrho(\gX,t,t')-\varrho(s,t,t')}\right)\nonumber\\ 
&=& 
{1\over n}\sum_{i=1}^{n}\cro{\rule{0mm}{7.5mm}\psi\pa{\sqrt{t'\over t}(X_{i})}-
{\mathbb{E}_{s}}\cro{\psi\pa{\sqrt{t'\over t}(X_{i})}}}.\label{def-Z00}
\end{eqnarray}

\subsection{The general framework\label{B2}}
We may similarly apply the previous reasoning to the more general context of independent but not
necessarily i.i.d.\ variables $X_{i}$, $1\le i\le n$. To do so, we shall extend the previous notations
to elements of $\LL_0$ and, in view of the application to the regression setting, we shall not
renormalize the sums by $1/n$. This leads to the following notations to be used throughout this
paper: for $\gt,\gt'\in\LL_{0}$ and $\gr=(\gt+\gt')/2$, we set
\begin{eqnarray}
\gvrho(\gX,\gt,\gt')&=&{1\over 2}\cro{\grho(\gt,\gr)+\sqrt{{\gt\over \gr}}(\gX)}\label{def-rho}\;\;=\;\;
{1\over 2}\sum_{i=1}^{n}\cro{\rho(t_{i},r_{i})+\sqrt{{t_{i}\over r_{i}}}(X_{i})};\nonumber\\
\gT(\gX,\gt,\gt')&=&\gvrho(\gX,\gt',\gt)-\gvrho(\gX,\gt,\gt')\label{def-T}\\&=&
\frac{1}{2}\sum_{i=1}^{n}[\rho(t'_i,r_i)-\rho(t_i,r_i)]+{1\over \sqrt{2}}\sum_{i=1}^{n}
\psi\pa{\sqrt{t'_i\over t_i}(X_{i})};\nonumber\\
\gZ(\gX,\gt,\gt')&=&\psi\pa{\sqrt{\gt'\over \gt}(\gX)}-\E\cro{\psi\pa{\sqrt{\gt'\over \gt}(\gX)}}
\label{def-Z}\\&=&\sum_{i=1}^{n}\cro{\rule{0mm}{9mm}\psi\pa{\sqrt{t'_{i}\over 
t_{i}}(X_{i})}-{\mathbb{E}_{s_{i}}}\cro{\psi\pa{\sqrt{t'_{i}\over t_{i}}(X_{i})}}}\nonumber.
\end{eqnarray}
With these notations, $\gvrho(\gX,\gt,\gt')$ is in fact the analogue of $n\varrho(\gX,t,t')$ 
given by (\ref{def-rhohat}) (and actually equal to it in the density framework), and so on.

\section{Estimation on a model\label{C}}

\subsection{Models\label{A1}}
As already mentioned, our construction of estimators will be based on ``models".  A 
model $\overline{S}\subset\LL_0$ should be viewed as an approximation set for the true unknown parameter $\gs$, which is used to build an estimator. It does not necessarily 
contain $\gs$, although we shall occasionally assume so. Typical models are either the parametric models that are used in Statistics or more general subsets of $\LL_0$ with 
well-known approximation properties that are derived from Approximation Theory in 
order to get a control on the approximation error $\gh(\gs,\overline{S})$. For measurability reasons to be explained later, we shall adopt the following definition for a model.
%
\begin{defi}\label{D-model}
A model $\overline{S}$ is a nonempty {\em separable} subset of the pseudometric space 
$(\LL_0,\gh)$ which means that one can find a countable subset $S$ of $\overline{S}$ such that
$\gh(\gt,S)=0$ for all $\gt\in\overline{S}$.
\end{defi}
In typical situations, $\LL_0$ itself is separable so that any nonempty subset of $\LL_0$ can be
used as a model. In the density framework described in Section~\ref{A2}, we have identified
a density $t$ on $(\X,\A,\mu)$ with the element $\gt=(t,t,\ldots,t)$ of $\LL_{0}$. Similarly, we shall
identify a separable set $\overline S$ of densities on $\X$ to the subset $\{\gt=(t,\ldots,t),\ t\in \overline S\}
\subset \LL_{0}$ and for simplicity denote both sets the same way.

\subsection{Construction of the estimators\label{C1}}
In order to avoid measurability issues, the construction of our estimator will be performed over
countable subsets $S$ of $\LL_{0}$ only. Besides, this corresponds to the practical point
of view since numerical optimization will always be done over a finite set. Replacing the original
model $\overline{S}$ by a countable and dense subset $S$ does not increase the approximation
error since then $\gh(\gs,S)=\gh(\gs,\overline{S})$. If instead we replace $\overline{S}$ by
$S\subset\LL_0$ which only satisfies $\sup_{\gu\in\overline S}\gh(\gu,S)=\sup_{\gu\in\overline{S}}
\inf_{\gt\in S}\gh(\gu,\gt)\le\eta$, this replacement may involve an additional error $\left|\gh(\gs,S)
- \gh(\gs,\overline{S})\right|$ which is not larger than $\eta$. We postpone the discussion about
what should be suitable choices of $S$ for a given model $\overline S$ to Section~\ref{C3}.

We shall also always assume $S$ to be {\em identifiable} in order that $(S,\gh)$ be a genuine
metric space. Note that this identifiability condition is not restrictive at all: if $S'$ is countable but not
identifiable, one can withdraw from it the redundant points in order to get an identifiable subset
$S\subset S'$ with $\gh(\gs,S)=\gh(\gs,S')$ for all $\gs$ in $\LL_{0}$. If $\overline{S}$ is identifiable
and $S$ is a subset of $\overline{S}$, there is nothing to do. Otherwise, we proceed as indicated
above. In any case, we shall always assume in the sequel and sometimes without further notice, that the
countable sets $S$ that we shall use in our construction are identifiable so that $(S,\gh)$ is a metric space.

Given $S$, we noticed in Section~\ref{B1} that, in the density framework, an almost best approximation of the density $s$ in $S$ can be obtained by minimizing over $S$ the function $t\mapsto\sup_{t'\in S}T(s ,t,t')$. If we assume that $T(\gX,t,t')$ provides a good approximation of $T(s ,t,t')$, it looks natural to minimize $\sup_{t'\in S}T(\gX ,t,t')$ with respect to $t\in S$ in order to derive a good estimation $\widehat{s}(\gX)$ of $s$. This suggests the following construction in the general situation of independent random variables.
For each $\gt\in S$ we define
\[
\gup(S,\gt)=\sup_{\gt'\in S}\gT(\gX,\gt,\gt'),
\]
which is always non-negative since $\gT(\gX,\gt,\gt)=0$, and 
\begin{equation}\label{def-EE}
\EE(\gX,S)=\ac{\widetilde\gs\in S\,\left|\,\gup(S,\widetilde\gs)\le\inf_{\gt\in S}\gup(S,\gt)+
{\kappa\over 10}\right.}\quad\mbox{with $\kappa$ given by (\ref{Eq-cons1})}.
\end{equation}
Finally, we define our estimator  of $\gs$ as any element (chosen in a measurable way)
$\widehat \gs$ in the closure ${\rm Cl}\!\left(\EE(\gX,S)\strut\right)$ of $\EE(\gX,S)$ in
$\LL_0$ with respect to $\gh$, that is the set 
\[
\left\{\gt\in\LL_{0}\mbox{ such that }\gh\left(\gt,\EE(\gX,S)\strut\right)=0\right\}.
\]
We shall call such an estimator a {\it $\rho$-estimator}, the greek letter $\rho$ referring to the Hellinger
affinity. Although it depends on our choice of $S$ we shall, for simplicity, omit to make this dependence
of $\widehat{\gs}$ with respect to $S$ explicit in our notations. Though the risk bounds we shall establish
remain valid for any choice of $\widehat\gs$ in ${\rm Cl}\!\left(\EE(\gX,S)\strut\right)$, we recommend in
practice to choose $\widehat\gs$ as a minimizer of $\gup(S,\cdot)$ over $S$ whenever it exists.

\subsection{Main theorem\label{C2}}
The properties of our estimator follow from those of the empirical process $\gZ(\gX,.,.)$ defined by~\eref{def-Z}. Given an element $\overline \gs\in\LL_{0}$ and a positive number 
$y$, we set
\begin{equation}
\B^{S}(\gs,\overline \gs,y)=\ac{\gt\in S\,\left|\,\gh^{2}(\gs,\gt)+\gh^{2}(\gs,\overline \gs)\le 
y^{2}\right.}
\label{Eq-BS}
\end{equation}
and
\[
\w^{S}(\gs,\overline \gs,y)=\E\cro{\sup_{\gt\in \B^{S}(\gs,\overline \gs,y)}\ab{\gZ(\gX,\overline \gs,\gt)}},
\]
with $\w^{S}(\gs,\overline \gs,y)=0$ if $\B^{S}(\gs,\overline \gs,y)$ is empty, according to our convention. When $\gs$ belongs to $S$ and one takes $\overline \gs=\gs$, 
\[
\w^{S}(\gs,\gs,y)=\E\cro{\sup_{\gt\in \B^{S}(\gs,y)}\ab{\gZ(\gX,\gs,\gt)}}
\]
measures, in some sense, the massiveness of $S$ in a neighborhood of $\gs$. Since
$-1\le\psi\le1$, the process $|\gZ(\gX,.,.)|$ is bounded by $2n$ and the non-decreasing mapping 
$y\mapsto \w^{S}(\gs,\overline \gs,y)$ as well. This implies that the number
\[
D^{S}(\gs,\overline \gs)=\overline{y}^{2}\vee1\quad\mbox{ with }\quad \overline{y}=\sup\ac{y\ge0\,\left|\,\w^{S}(\gs,\overline \gs,y)>c_0y^{2}\right.}
\]
belongs to the interval $\left[1,2nc_0^{-1}\right]$. It follows from this definition of $D^{S}$ that
\begin{equation}\label{def-D}
\w^{S}(\gs,\overline \gs,y)\le c_0y^{2}\quad\mbox{for all }y>\sqrt{D^{S}(\gs,\overline \gs)}.
\end{equation}
\begin{thm}\label{main00}
Let $S$ be a countable and identifiable subset of $\LL_0$. The estimation procedure described in Section~\ref{C1} leads to the following bound which is valid for any $\rho$-estimator 
$\widehat \gs$ based on $S$, all $\gs$ in $\LL_{0}$ and all $\xi>0$ :
\begin{equation}\label{Mmain}
\P_{\gs}\left[\gh^{2}(\gs,\widehat \gs)\le \inf_{\overline \gs\in S}
\ac{c_1\gh^{2}(\gs,\overline \gs)-\gh^{2}(\gs,S)
+c_{2}D^{S}(\gs,\overline \gs)}+c_3(1.45+\xi)\right]\ge1-e^{-\xi}.
\end{equation} 
This implies in particular that
\begin{equation}\label{risk}
\E\cro{\gh^2(\gs,\widehat \gs)}\le \inf_{\overline \gs\in S}\ac{c_1\gh^2(\gs,\overline \gs)-\gh^{2}(\gs,S)
+c_{2}D^{S}(\gs,\overline \gs)}+c_4
\end{equation}
and, more generally,
\begin{equation}
\E\cro{\gh^{\ell}(\gs,\widehat \gs)}\le C(\ell)\left[\inf_{\overline \gs\in S}\ac{\gh^{\ell}(\gs,\overline \gs)+\left(D^{S}(\gs,\overline \gs)\right)^{\ell/2}}\right]\quad\mbox{for all }\ell\ge1.
\label{Eq-risk1}
\end{equation}
\end{thm}
The proof will be provided in Section~\ref{P1}. If we set
\begin{equation}
D^{S}=\sup_{(\gs,\overline \gs)\in\LL_{0}\times S}D^{S}(\gs,\overline\gs)\qquad\mbox{and}\qquad
\overline{D}^{S}=\sup_{(\gs,\overline\gs)\in\LL_{0}\times\LL_0}D^{S}(\gs,\overline\gs),
\label{Eq-DS}
\end{equation}
then (\ref{Mmain}) becomes
\begin{equation}
\P_{\gs}\left[\gh^{2}(\gs,\widehat \gs)\le(c_1-1)\gh^{2}(\gs,S)
+c_{2}D^{S}+c_3(1.45+\xi)\right]\ge1-e^{-\xi}.
\label{Eq-Mmain}
\end{equation} 
\noindent{\bf Remark:} In the sequel, we shall often content ourselves to provide our results in the form of exponential deviations similar to (\ref{Mmain}) and (\ref{Eq-Mmain}) like
\[
\P_{\gs}\left[\gh^{2}(\gs,\widehat \gs)\le\Gamma+c\xi\right]\ge1-e^{-\xi}\quad\mbox{for all }
\xi>0,
\]
where $\Gamma$ depends on various quantities involved in our assumptions. 
Such a deviation bound immediately leads by integration to $\E\left[\gh^{2}(\gs,\widehat \gs)\right]
\le\Gamma+c$ and also implies bounds similar to (\ref{Eq-risk1}) for the moments of 
$\gh(\gs,\widehat \gs)$ as well as risk bounds for more general loss functions of the form 
$\ell(\gh(\gs,\widehat \gs))$.\vspace{2mm}

In the forthcoming sections, we shall use this central Theorem to establish risk bounds for our 
estimator over more general models than just countable ones. Before turning to these bounds, 
let us note here that we can already deduce from~\eref{risk} (taking $\overline \gs=\gs$) that the estimator $\widehat\gs$ satisfies, since $D^{S}(\overline\gs,\overline\gs)\ge 1$,
\begin{equation}\label{Eq-enEsp}
\sup_{\gs\in S}\E\cro{\gh^2(\gs,\widehat \gs)}\le Cd(S)\;\;\mbox{ with }\;\;d(S)
=\sup_{\overline \gs\in S}D^{S}(\overline \gs,\overline \gs).
\end{equation}
The assumption that $\gs$ belongs to $S$ is quite restrictive, nevertheless we shall see in the next section that a bound of this type is not only true for the elements $\gs$ lying in $S$ but also for those which are close enough to $S$ with respect to the Kullback-Leibler divergence.

\subsection{Robustness properties with respect to the Kullback-Leibler divergence\label{C4}}
We recall that the Kullback-Leibler divergence (KL-divergence for short) between two probabilities 
$P $ and $Q$ on $\X$ is given by 
\[
K(P,Q)=\int\log\pa{dP\over dQ}dP\in[0,+\infty]\;\;\mbox{if }P\ll Q\quad\mbox{and}\quad
K(P,Q)=+\infty\;\;\mbox{otherwise}.
\]
For $\gt,\gt'\in\LL_{0}$, we shall set for simplicity
\[
\gK(\gt,\gt')=K\left(\bigotimes_{i=1}^{n}(t_{i}\cdot\mu_{i}),\bigotimes_{i=1}^{n}(t'_{i}\cdot
\mu_{i})\right)=\sum_{i=1}^{n}K(t_{i}\cdot\mu_{i},t_{i}'\cdot\mu_{i}).
\]
It is well-known that $2\gh^2(\gt,\gt')\le\gK(\gt,\gt')$ for all $\gt,\gt'\in\LL_0$.
%
\begin{thm}\label{T-KL}
Let $\overline S$ be a model and $S$ a countable subset of $\overline S$ satisfying
\begin{equation}\label{hypo-KL}
\inf_{\overline \gs\in S}\gK(\gs,\overline{\gs})=\inf_{\overline \gs\in \overline S}\gK(\gs,\overline{\gs})=\gK(\gs,\overline S)\quad\mbox{for all }\gs\in\LL_{0}.
\end{equation}
Then, any $\rho$-estimator $\widehat{\gs}$ based on $S$ satisfies, with $d(S)$ given by 
(\ref{Eq-enEsp}),
\begin{equation}\label{eq-Krobuste}
\mathbb{E}_{\gs}\cro{\gh^{2}(\gs,\widehat \gs)}\le C\cro{\gK(\gs,\overline S)+d(S)}
\quad\mbox{for all }\gs\in\LL_{0}
\end{equation}
and some universal constant $C$.
\end{thm}
\begin{proof}
It relies on Theorem~\ref{main00} and the following proposition (to be proved in Section~\ref{P22}) which is a variant of the lemma (Section~5.3) in Barron~\citeyearpar{MR1154352} and of independent interest since it applies to many other situations.
%
\begin{prop}\label{P-KLrob}
Let $\gs$ and $\overline{\gs}$ belong to $\LL_0$ and $T(\gX)$ be a random variable  such that
\[
\mathbb{P}_{\overline{\gs}}\cro{T(\gX)\ge z}\le ae^{-z}\quad\mbox{for all }z\ge0\quad\mbox{and some }a>0.
\]
Then, if $c=\log(1+a)+\gK(\gs,\overline\gs)$,
\begin{equation}\label{Eq-KLa}
\mathbb{E}_{\gs}\cro{T(\gX)}\le1+c+\log\left(1+c+\sqrt{2c}\right)<1+c+\sqrt{2c}.
\end{equation}
If $\mathbb{P}_{\overline{\gs}}\cro{T(\gX)\ge z}\le ae^{-bz}$ for all $z\ge z_0\ge0$ with $a,b>0$, then
\begin{equation}
\mathbb{E}_{\gs}\cro{T(\gX)}\le z_0+ b^{-1}\left(1+c'+\sqrt{2c'}\right)\quad\mbox{with}\quad
c'=\log\left(1+ae^{-bz_0}\right)+\gK(\gs,\overline\gs).
\label{Eq-KLb}
\end{equation}
\end{prop}
It follows from (\ref{Mmain}) with $\gs$ replaced by $\overline{\gs}\in S$ that
\begin{equation}
\P_{\overline\gs}\left[\gh^{2}(\overline\gs,\widehat \gs)>c_{2}D^{S}(\overline \gs,\overline \gs)
+1.45c_3+\xi\right]\le e^{-\xi/c_3}\quad\mbox{for all }\xi>0
\label{Eq-risk0}
\end{equation}
and we may therefore apply (\ref{Eq-KLb}) to $T(\gX)=\gh^{2}(\overline\gs,\widehat \gs)-c_{2}D^{S}(\overline \gs,\overline \gs)-1.45c_3$ with $a=1$, $b=1/c_3$ and $z_0=0$. Then $c'=\log2+\gK$ with 
$\gK=\gK(\gs,\overline\gs)$ and
\[
\mathbb{E}_{\gs}\cro{\gh^{2}(\overline\gs,\widehat \gs)}\le
c_{2}D^{S}(\overline{\gs},\overline\gs)+1.45c_3+c_{3}\left(1+\log 2+\gK+\sqrt{2(\log2+\gK)}\right).
\]
We finally derive from the triangular inequality and (\ref{Eq-square}) with $\alpha=1/50$ that
\[
\mathbb{E}_{\gs}\cro{\gh^{2}(\gs,\widehat \gs)}\le\frac{51}{50}\left[c_{2}D^{S}(\overline{\gs},\overline\gs)+c_{3}\left(2.45+\log 2+\gK+\sqrt{2(\log2+\gK)}\right)+50\gh^{2}(\gs,\overline{\gs})\right].
\]
Since $\overline{\gs}$ is arbitrary in $S$ and $\gh^2\le\gK/2$, it follows that
\[
\mathbb{E}_{\gs}\cro{\gh^{2}(\gs,\widehat \gs)}\le\frac{51}{50}\,\inf_{\overline{\gs}\in S}\left\{c_{2}D^{S}(\overline{\gs},\overline\gs)+c_{3}\left(2.45+\log 2+\gK+\sqrt{2(\log2+\gK)}\right)+25\gK\right\},
\]
hence by (\ref{Eq-enEsp}) and \eref{hypo-KL}, 
\[
\mathbb{E}_{\gs}\cro{\gh^{2}(\gs,\widehat \gs)}\le C\cro{\inf_{\overline \gs\in S}\gK(\gs,\overline \gs)+d(S)}=C\cro{\gK(\gs,\overline S)+d(S)}
\]
for some universal constant $C>0$.
\end{proof}
Inequality~\eref{eq-Krobuste} shows that~\eref{Eq-enEsp} is not only true when $\gs$ belongs to $S$ but also when it belongs to $\overline S$ provided that \eref{hypo-KL} holds and that this risk bound deteriorates by at most the additional term $\gK(\gs,\overline S)=\inf_{\overline \gs\in\overline{S}}
\gK(\gs,\overline \gs)$ when $\gs$ does not belong to $\overline S$. The estimator $\widehat \gs$ is therefore robust with respect to the KL-divergence. 

Similar results actually hold for any estimator $\widetilde{\gs}$ and any non-negative loss function $\ell$ such that an analogue of (\ref{Eq-risk0}) is satisfied, more precisely if
\[
\P_{\overline\gs}\left[\ell(\overline\gs,\widehat \gs)>C(\overline{\gs})+\xi\right]\le e^{-b\xi}\quad\mbox{for all }\xi>0\mbox{ and }\overline{\gs}\in S.
\]
This indeed implies by (\ref{Eq-KLb}) that
\[
\mathbb{E}_{\gs}\cro{\ell(\overline{\gs},\widehat \gs)}\le C(\overline{\gs})+(C'/b)
\left[1+\gK(\gs,\overline\gs)\right]
\]
and, if the loss function $\ell$ satisfies $\ell(\gs,\gt)\le A\left[\ell(\gs,\gu)+\ell(\gu,\gt)\right]$ 
for some constant $A$ and all $\gs,\gt,\gu$, then
\[
\mathbb{E}_{\gs}\cro{\ell(\gs,\widehat \gs)}\le A\left[C(\overline{\gs})+(C'/b)
\left[1+\gK(\gs,\overline\gs)\right]+\ell(\gs,\overline{\gs})\strut\right]\quad\mbox{for all }
\overline{\gs}\in S
\]
and finally
\[
\mathbb{E}_{\gs}\cro{\ell(\gs,\widehat \gs)}\le A\sup_{\overline{\gs}\in S}
C(\overline{\gs})+C_0\left[1+\gK(\gs,S)+\ell(\gs,S)\strut\right]
\quad\mbox{with}\quad\ell(\gs,S)=\inf_{\overline \gs\in S}\ell(\gs,\overline \gs).
\]
This means that, if one allows bias terms depending on KL-divergences, which is often the case in density estimation when one uses likelihood-based methods, one can always assume that the true parameter belongs to the model $S$ and then extend the result to all $\gs$ satisfying $\inf_{\overline{\gs}\in S}\gK(\gs,\overline{\gs})<+\infty$.

\subsection{Robustness properties with respect to the Hellinger distance}\label{C3}
Unfortunately, if $\gh^2(\gs,\overline \gs)\le\gK(\gs,\overline \gs)/2$, the reciprocal 
$\gh^2(\gs,\overline \gs)\ge c\gK(\gs,\overline \gs)/2$ for some positive $c$ is definitely not true in general and we cannot use the previous results to get robustness properties with respect to the Hellinger distance, which is actually a much stronger property. Hopefully, this robustness property is already included in our Theorem~\ref{main00}. We do need this robustness property in order to work with general models, not only countable ones. Since our models, as described in Section~\ref{A1} (for instance the classical sets that are used in Approximation Theory), are typically uncountable, we have to replace them by countable approximations. 

We shall therefore apply the following strategy: given a model $\overline{S}$ for $\gs$ replace it by a countable and identifiable approximating set $S$ to build our estimator. The natural question at this stage is then:
``given $\overline{S}$, how to choose $S$?". First, $S$ should approximate $\overline{S}$ within some (typically small) $\eta$, according to the following definition.
%
\begin{defi}\label{D-etanet}
Given a model $\overline{S}$ in the pseudometric space $(\LL_0,\gh)$ and $\eta\ge0$, we say 
that a {\em countable} subset $S[\eta]$ of $\LL_0$ (not necessarily included in 
$\overline{S}$) is an $\eta$-net for $\overline{S}$ if $\sup_{\gt\in\overline S}\gh(\gt,S[\eta])
\le \eta$. In particular a countable and dense subset of $\overline{S}$ is a 0-net.
\end{defi}
This immediately leads to the following corollary:
%
\begin{cor}\label{C-main0}
Let $S=S[\eta]$ be a countable and identifiable $\eta$-net for $\overline{S}$ and $\widehat \gs$ a $\rho$-estimator based on $S$. Then, for all $\xi>0$,
\begin{equation}\label{Mmain2}
\P_{\gs}\left[\gh^{2}(\gs,\widehat \gs)\le c'_1\gh^{2}\left(\gs,\overline{S}\right)+\left(c'_1\eta^2+c_2D^{S[\eta]}\right)+c_3(1.45+\xi)\right]\ge1-e^{-\xi},
\end{equation} 
hence
\begin{equation}\label{risk2}
\E\cro{\gh^2(\gs,\widehat \gs)}\le c'_1\gh^{2}\left(\gs,\overline{S}\right)+
\left(c'_1\eta^2+c_2D^{S[\eta]}\right)+c_4.
\end{equation} 
\end{cor}
\begin{proof}
The first bound follows from the inequality $\gh^{2}(\gs,S[\eta])\le2\gh^{2}(\gs,\overline{S})+2\eta^2$ applied to (\ref{Eq-Mmain}) and the second one by integration.
\end{proof}
We see that (\ref{risk2}) corresponds to a decomposition of the risk into the sum of three terms among which only one, namely $c'_1\eta^2+c_2D^{S[\eta]}$ depends on the chosen net $S[\eta]$. We shall therefore focus our attention on a similar quantity, introducing the two following new notions of dimension. 
%
\begin{defi}\label{def-dimG}
Given a model $\overline{S}$ in $\LL_{0}$, we define its dimension $D(\overline{S})$ and its uniform dimension $\overline D(\overline{S})$ by
\begin{equation}
D(\overline{S})=\inf_{S}\cro{2c_1\sup_{\gu\in\overline S}\gh^2(\gu,S)+c_{2}D^{S}}\ \mbox{and}\ \ \overline{D}(\overline{S})=\inf_{S}\cro{2c_1\sup_{\gu\in\overline S}\gh^2(\gu,S)+c_{2}\overline{D}^{S}},
\label{Eq-dimG}
\end{equation}
where $D^S$ and $\overline{D}^S$ have been defined in (\ref{Eq-DS}) and, in both cases, the infimum is taken over all countable and identifiable subsets $S$ of $\LL_{0}$.
\end{defi}
It follows from these definitions and the bounds $1\le D^{S}(\gs,\overline \gs)\le2nc_0^{-1}$ that 
\begin{equation}
1<c_2\le D(\overline{S})\le \overline D(\overline{S})\le2n\left(c_1+c_2c_0^{-1}\right)<144n.
\label{Eq-dims}
\end{equation}
If $\overline{S}\subset\overline{S'}$, then $D(\overline{S})\le D(\overline{S'})$ and 
$\overline D(\overline{S})\le \overline D(\overline{S'})$ which means that $D$ and 
$\overline D$ are  non-decreasing with respect to the inclusion. These two dimensions 
have actually different purposes: we shall use $D(\overline{S})$ to bound the risk of a 
$\rho$-estimator on a given model $\overline S$ while $\overline D(\overline{S})$ will 
be used for model selection purposes in Section~\ref{MS}.

Since the separability of $\overline{S}$ implies the existence of $\eta$-nets $S[\eta]$ 
for $\overline{S}$ whatever $\eta\ge0$, \eref{Eq-dimG} can be reformulated as 
\begin{equation}
D(\overline{S})=\inf_{\eta\ge0}\inf_{S[\eta]}\cro{2c_1\eta^{2}+c_{2}D^{S[\eta]}}
\qquad\mbox{and}\qquad\overline{D}(\overline{S})=
\inf_{\eta\ge0}\inf_{S[\eta]}\cro{2c_1\eta^{2}+c_{2}\overline D^{S[\eta]}},
\label{Eq-dim}
\end{equation}
where the infima now run over all possible identifiable $\eta$-nets $S[\eta]\subset\LL_{0}$ for $\overline S$.

The important property of these dimensions lies in the fact that they allow to replace the model $\overline{S}$ by a suitable $\eta$-net $S[\eta]$ ($\eta\ge0$) to which 
Theorem~\ref{main00} applies. In particular, choosing $S[\eta]$ such that 
\[
2c_1\eta^{2}+c_{2}D^{S[\eta]}\le D(\overline{S})+c_3/20,
\]
which is always possible in view of (\ref{Eq-dim}), we get from (\ref{Mmain2}) and an integration with respect to $\xi$ the following risk bounds.
%
\begin{cor}\label{C-main1}
Given a model $\overline{S}$, there exists a $\rho$-estimator $\widehat\gs$ such that for all $\gs\in\LL_{0}$
\begin{equation}
\P_{\gs}\left[\gh^{2}(\gs,\widehat \gs)\le c'_1\gh^{2}\left(\gs,\overline{S}\right)+D(\overline{S})
+c_3(1.5+\xi)\right]\ge1-e^{-\xi}\quad\mbox{for all }\xi>0
\label{Mmain3}
\end{equation} 
and
\begin{equation}
\E\cro{\gh^2(\gs,\widehat \gs)}\le c'_1\gh^{2}\left(\gs,\overline{S}\right)+D(\overline{S})+c_4.
\label{risk3}
\end{equation} 
\end{cor}
The replacement of $\overline{S}$ by a suitable pair $(\eta,S[\eta])$ leads to the risk bound 
(\ref{risk3}) depending on $\overline{S}$ only. This means that, given a model $\overline{S}$, 
the risk of  $\widehat\gs$ breaks down, up to numerical constants, into the bias term 
$\gh^{2}(\gs,\overline{S})$ which depends on the quality of the approximation of $\gs$ by 
the model $\overline{S}$ and the dimensional term $D(\overline{S})$ which measures in 
some sense the massiveness of $\overline{S}$. In particular, in the i.i.d.\ case, we get
\[
\E\cro{h^{2}(s,\widehat s)}\le c'_1h^{2}(s,\overline{S})+n^{-1}\left[D(\overline{S})+c_4\right],
\]
as expected.

\subsubsection{Models which are VC-subgraph classes\label{C3a}}
The first situation that we shall consider is about models $\overline{S}$ such that any 
countable and identifiable subset $S$ of $\overline{S}$ satisfies $D^S\le D'$ where $D'$ only depends on 
$\overline{S}$ but not on the choice of the subset $S$. In such a case it is natural to choose 
for $S$ a countable and dense subset of $\overline{S}$ which is a 0-net for $\overline{S}$ 
so that (\ref{risk2}) leads to 
\begin{equation}\label{risk-VC}
\E\cro{\gh^2(\gs,\widehat \gs)}\le c'_1\gh^{2}\left(\gs,\overline{S}\right)+c_2D'+c_{4}.
\end{equation} 
In order to deal with this situation, we first need to prove an auxiliary result and 
for this we shall consider an element  ${\bf t}=(t_{1},\ldots,t_{n})$ in $\LL_{0}$ as a real-valued function on $\overline \X=\bigcup_{i=1}^{n}\pa{\{i\}\times \X_{i}}$ defined by
\begin{equation}
{\bf t}(\overline x)=t_{i}(x)\quad\mbox{for all }\overline x=(i,x)\in\overline{\mathscr{X}}.
\label{Eq-Xbar}
\end{equation}
Replacing the $X_{i}$ by the random variables $\overline X_{i}=(i,X_{i})$ so that ${\bf t}(\overline X_{i})=t_i(X_i)$, we see that $\w^{S}({\bf s},\overline{\bf s},y)$ can be written as
\begin{equation}
\w^{S}({\bf s},\overline{\bf s},y)=\mathbb{E}_{\bf s}\left[
\sup_{{\bf f}\in\mathscr{F}^{S}({\bf s},\overline{\bf s},y)}\left|\,\sum_{i=1}^{n}
\left({\bf f}(\overline X_{i})-\mathbb{E}_{\bf s}\left[{\bf f}(\overline{X_{i}})\right]\right)\right|\right],
\label{Eq-VCwS}
\end{equation}
where the supremum runs among the class $\mathscr{F}^{S}({\bf s},\overline{\bf s},y)$ of real-valued functions ${\bf f}$ on $\overline{\mathscr{X}}$ given by 
\[
\mathscr{F}^{S}({\bf s},\overline{\bf s},y)=\left\{\left.\psi\left(\sqrt{{\bf t}/\overline{\bf s}}\right)
\,\right|\,{\bf t}\in\mathscr{B}^{S}({\bf s},\overline{\bf s},y)\right\}.
\]
For a set of real-valued functions $\FF$ on $\overline \X$ and a probability $Q$ on $\overline \X$, we denote by $N(\FF,Q,\eta)$ the $\eta$-covering number of $\FF$ with respect to $Q$, that is, the smallest number of closed balls (with respect to the distance in $\IL_{2}(Q)$) with centers in $\FF$ and radius $\eta$ needed to cover $\FF$. We finally introduce the following assumption for a function $\Ent$.
\begin{ass}\label{A-VC}
The function $\Ent$ defined on $[1/2,+\infty)$ is non-negative, non-decreasing and
\[
L=\sup_{x\ge 1/2}\left\{x\left[\Ent(x)\right]^{-1/2}\int_{x}^{+\infty}u^{-2}\sqrt{\Ent(u)}\,du\right\}<+\infty.
\]
\end{ass}
The next result will be proved in Section~\ref{P3}. 
\begin{prop}\label{Kolt2}
For all $y>0$, assume that there exists a function $\Ent_{\!y}$, possibly depending on $y$ and satisfying Assumption~\ref{A-VC} with $L=L_y$, such that for all $\gs,\overline \gs\in\ \LL_{0}$, 
\begin{equation}
\log N\pa{\Y,{1\over n}\sum_{i=1}^n\delta_{\overline X_i(\omega)},z}\le\Ent_{\!y}\pa{1\over z}
\quad\mbox{for all }\omega\in\Omega\mbox{ and } 0<z\le2.
\label{vap0}
\end{equation}
There exists a universal constant $C_0$ such that
\begin{equation}
\w^{S}(\gs,\overline \gs,y)\le C_0\cro{yL_y\sqrt{6H_y}+
L_y^{2}H_y}\quad\mbox{with}\quad H_y=\Ent_{\!y}\pa{{\sqrt{n\over24y^2}}\bigvee\frac{1}{2}},
\label{B10}
\end{equation}
hence
\begin{equation}
D^{S}(\gs,\overline\gs)\le \overline D^{S}\le \sup\left\{y^2\,\left|\,0<y^2<{2C_{0}\over c_0}
\pa{1+{3C_{0}\over c_0}}L_y^2H_y\right.\right\}\bigvee1.
\label{B21}
\end{equation}
\end{prop}
It happens that an inequality such as~\eref{vap0} is typically satisfied 
for VC-subgraph classes $\FF^S$. In order to avoid a long digression, we differ the relevant 
definitions, properties and proofs about VC-subgraph classes to Section~\ref{S}. At this stage, 
it is sufficient to recall that the index $\overline{V}$ of a VC-subgraph class is a positive integer. 
Our main result about models $\overline{S}$ which are VC-subgraph classes is as follows.
%
\begin{thm}\label{main1}
If $\overline{S}$, viewed as a set of real-valued functions on $\overline{\X}$ as defined by 
(\ref{Eq-Xbar}), is VC-subgraph with index $\overline V$, then for all countable, identifiable and dense subsets $S$ of $\overline S$,
\begin{equation}
\overline{D}(\overline{S})\le c_2\overline D^S\le C\overline V\left[1+\log_+\left(n/\overline V\right)\right]
\label{Eq-DSbar}
\end{equation}
for some universal constant $C$. Consequently, any $\rho$-estimator $\widehat{\gs}$ based 
on such an $S$ satisfies, whatever $\gs\in\LL_{0}$,
\begin{equation}\label{risk-VC01}
\P_{\gs}\cro{C'\gh^2(\gs,\widehat \gs)\le \gh^{2}\left(\gs,\overline{S}\right)+
\overline V\left[1+\log_+\left(n/\overline V\right)\right]+\xi}\ge 1-e^{-\xi}\quad\mbox{for all }\xi>0
\end{equation} 
and
\begin{equation}\label{risk-VC1}
C''\E\cro{\gh^2(\gs,\widehat \gs)}\le \gh^{2}\left(\gs,\overline{S}\right)+
\overline V\left[1+\log_+\left(n/\overline V\right)\right]
\end{equation} 
for some universal constants $C',\,C''>0$.
If $\X_{i}= \X$ for all $i$ and $\overline{S}$ is of the form $\{\gt=(t,\ldots,t), t\in\Theta\}$ 
for a set $\Theta$ of real valued functions on $\X$ which is VC-subgraph with index $\overline V$, 
the previous bound still holds.
\end{thm}
\begin{proof}
Let $S$ be any countable subset of $\overline{S}$. Then it is VC-subgraph with index not larger 
than $\overline V$. Since for all $\gs$ and $\overline \gs$ in $\LL_{0}$ and $y>0$, 
$\Y\subset\left\{\left.\psi\left(\sqrt{{\bf t}/\overline{\bf s}}\right)\,\right|\,{\bf t}\in S\right\}$, it follows 
from $(vii)$ of Proposition~\ref{perm} that $\Y$ is VC-subgraph with index not larger than 
$\overline{V}$. Then, by (\ref{Eq-h(x)}) below, there exists a universal constant $A$ such that, 
for all $\gs,\overline \gs\in\LL_{0}$, $y,z>0$ and any probability $Q$ on $\overline\X$, 
\begin{equation}\label{unifB}
\log N\pa{\Y,Q,z}\le 2\overline V \log_{+}(A/ z).
\end{equation}
Proposition~\ref{Kolt2} therefore applies with $\Ent_y(x)=2\overline V \log_+(Ax)$ and we may assume that $A\ge2e$ so that $\log_+(Au)=\log(Au)\ge1$ for $u\ge1/2$. An integration by parts then leads, for $x\ge1/2$, to
\[
\frac{\int_{x}^{+\infty}u^{-2}\sqrt{\log(Au)}\,du}{x^{-1}\sqrt{\log(Ax)}}=1+\frac{x}{\sqrt{\log(Ax)}}
\int_{x}^{+\infty}\frac{du}{2u^2\sqrt{\log(Au)}}<1+x\int_{x}^{+\infty}\frac{du}{2u^2}=\frac{3}{2},
\]
which shows that $L_y\le3/2$ and~\eref{Eq-DSbar} follows from  (\ref{B21}). Inequalities~\eref{risk-VC01} and~\eref{risk-VC1} derive from~\eref{Mmain2} and~\eref{risk2} respectively with $S[\eta]=S$ and $\eta=0$ since $\overline{V}\ge1$.
\end{proof}
%
\subsubsection{Models which are totally bounded\label{C3b}}
Of course, not all models are VC-subgraph classes but there exists another type of models for which we are able to bound $D^{S[\eta]}$ for suitable $\eta$-nets of $\overline{S}$. When $\overline{S}$ is totally bounded, one can take $S[\eta]$ finite for all $\eta>0$ and so are the subsets $\B^{S[\eta]}(\gs,\overline \gs,y)$ of $S[\eta]$ for all positive $y$ and $\eta$. Conversely, if, for all $y,\eta>0$, one can choose $S[\eta]$ so that the sets $\B^{S[\eta]}(\gs,\overline \gs,y)$ are finite, this is in particular true for $y=\sqrt{2n}$ and, since the distance $\gh$ is bounded by $\sqrt{n}$, $S[\eta]=\B^{S[\eta]}(\gs,\overline \gs,\sqrt{2n})$ is finite  for all $\eta>0$. This implies that $\overline{S}$ is totally bounded so that this approach based on the cardinality of $\B^{S[\eta]}(\gs,\overline \gs,y)$ is restricted to totally bounded models only. It nevertheless has the advantage to require the control of the supremum of the process $\left|\gZ(\gX,\overline \gs,.)\right|$ over a finite set which can be done via the following result to be proved in Section~\ref{P2}.
\begin{prop}\label{Kolt}
Let $\gs$ and $\overline{\gs}$ belong to $\LL_{0}$, $y>0$ and $S$ be a countable and identifiable subset of $\LL_{0}$. Assume that $\left|\B^{S}(\gs,y)\right|<+\infty$, then
\[
\w^{S}(\gs,\overline \gs,y)\le2\left[y\sqrt{3\log_{+}(N)}+\log_{+}(N)\right]\quad\mbox{with}\quad 
N=2\left|\B^{S}(\gs,y)\right|.
\]
\end{prop}
With such a result at hand, bounding $D^{S[\eta]}(\gs,\overline\gs)$ amounts to controlling $|S[\eta]\cap \B(\gs,y)|$ when $S[\eta]$ is minimal. Since $\gs$ is unknown, we need to bound the number of points of $S[\eta]$ lying in an arbitrary Hellinger ball of radius $y$. It is then natural to introduce the following entropy bounds.
\begin{defi}\label{D-metdim}
Given a totally bounded model $\overline{S}$ of $\LL_{0}$, $\eta>0$ and $S[\eta]$ an $\eta$-net for $\overline{S}$, we set
\[
\HH^{\overline{S}}(\eta,S[\eta],y)=\sup_{\gs\in\LL_0}\log\left|S[\eta]\cap \B(\gs,y)\strut\right|\ge0\quad\mbox{for }y\ge\eta.
\]
We shall say that $\overline{S}$ has an entropy dimension bounded by $V\ge0$ if, for all $\eta>0$, there exists some $\eta$-net $S[\eta]$ for $\overline{S}$ such that
\begin{equation}\label{def-entropy}
\HH^{\overline{S}}(\eta,S[\eta],y)\le V\log\pa{y/\eta}\quad\mbox{for all }y\ge2\eta.
\end{equation} 
Let $\widetilde D$ be a right-continuous function from $(0,+\infty)$ into $[1/2,+\infty]$ with 
$\widetilde D\left(\eta\right)=1/2$ for $\eta\ge\sqrt{n}$. We shall say that $\overline{S}$ has a metric 
dimension bounded by $\widetilde{D}(\cdot)$ if, for all $\eta>0$, there exists some $\eta$-net 
$S[\eta]$ for $\overline{S}$ such that
\begin{equation}\label{def-birge}
\HH^{\overline{S}}(\eta,S[\eta],y)\le(y/\eta)^{2}\widetilde D(\eta)\quad\mbox{for all }y\ge2\eta.
\end{equation}
\end{defi}

The definition of the metric dimension is due to  Birg\'e~\citeyearpar{MR2219712} 
(Definition~6 p.~293). Since the distance $\gh$ that we use here is bounded by $\sqrt{n}$, 
any singleton $\{\gt\}$ in $\LL_0$ is a $\sqrt{n}$-net for any subset of $\LL_0$ so that $\HH^{\overline{S}}(\eta,\{\gt\},y)=0$ 
for $y/2\ge\eta\ge\sqrt{n}$ and we can always set $\widetilde D\left(\eta\right)=1/2$ 
for $\eta\ge\sqrt{n}$. The logarithm being a slowly varying function, 
it is not difficult to see that the notion of metric dimension is more general than the 
entropy one in the sense that if $\overline{S}$ has an entropy dimension bounded 
by some $V$, then it also has a metric dimension bounded by $\widetilde D(\cdot)$ with
\begin{equation}\label{comp}
\widetilde D(\eta)\le(1/2)\vee[V(\log2)/4]\quad\mbox{for all }\eta>0.
\end{equation}
%
\begin{prop}\label{bound1}
Let $\overline{S}$ be a totally bounded nonempty subset of $\LL_{0}$ with metric dimension bounded by $\widetilde D(\cdot)$. Let $\overline \eta$ be defined by
\[
\overline \eta=\inf\ac{\eta>0\,\left|\,\eta^{-2}\widetilde D(\eta)\le8c_0^2/131\right.}.
\]
Then one can find an $\overline \eta$-net $S[\overline \eta]$ for $\overline{S}$ which satisfies 
$\overline{D}^{S[\overline{\eta}]}\le4\overline\eta^{2}$. Hence $\overline{D}(\overline{S})\le 2\left(c_1+2c_2\right)\overline\eta^{2}$ and any $\rho$-estimator $\widehat\gs$ based on $S[\overline{\eta}]$ satisfies
\[
\P_{\gs}\left[\gh^{2}(\gs,\widehat \gs)\le c'_1\gh^{2}\left(\gs,\overline{S}\right)+2(c_1+2c_2)
\overline\eta^{2}+c_3(1.45+\xi)\right]\ge1-e^{-\xi}\quad\mbox{for all }\xi>0.
\]
\end{prop}
\begin{proof} 
For $\eta>0$, let $S[\eta]$ be a minimal $\eta$-net for $\overline{S}$. Using~\eref{def-birge} and the fact that $\widetilde D(\eta)\ge 1/2$, we derive that 
\[
\log_{+}\pa{2\,|\B^{S[\eta]}(\gs,y)|}\le\log2+{y^{2}\widetilde D(\eta)\over\eta^{2}}\le\left(1+\frac{\log2}{2}
\right){y^{2}\widetilde D(\eta)\over\eta^{2}}\quad\mbox{for all }y\ge2\eta.
\]
If, moreover, $y\ge2\overline{\eta}$, using the fact that $\widetilde D(\eta)$ is right-continuous and the definition of $\overline \eta$, we see that
\[
\frac{y^{2}\widetilde D(\overline{\eta})}{\overline{\eta}^{2}}\le\frac{8c_0^2y^2}{131},
\]
so that we can apply Proposition~\ref{Kolt} with
\[
\log_{+}\pa{2\,|\B^{S[\overline{\eta}]}(\gs,y)|}\le[1+(\log2)/2]\left[8c_0^{2}y^{2}/131\right]=ay^{2}
\]
and get
\[
\w^{S[\overline{\eta}]}(\gs,\overline\gs,y)\le 2\pa{a+\sqrt{3a}}y^{2}<c_0y^2\quad\mbox{for all }
\gs,\,\overline{\gs}\in\LL_{0}\mbox{ and }y\ge2\overline{\eta}.
\]
Therefore $D^{S[\overline{\eta}]}(\gs,\overline\gs)\le4\overline\eta^{2}$ which leads 
to the bounds for $\overline{D}^{S[\overline{\eta}]}$ and $\overline{D}(\overline{S})$. 
The bound for $\gh^{2}(\gs,\widehat \gs)$ then follows from (\ref{Mmain2}).
\end{proof}
\noindent{\bf Remark:} Since $\widetilde{D}(\overline{\eta})\ge1/2$, $2(c_1+2c_2)\overline{\eta}^2 
\ge(131/8)(c_1+2c_2)c_0^{-2}$. It follows that the bounds provided by Proposition~\ref{bound1} 
are trivial if $n$ is not larger than this last quantity.

\subsubsection{Minimax risk on a model\label{C3c}}
Let us now focus on the specific case of the risk of $\rho$-estimators over a model 
$\overline{S}$ when $\gs$ is an arbitrary point in $\overline{S}$ or equivalently on the maximal 
risk of $\rho$-estimators over a model $\overline{S}$ in $\LL_{0}$ since it provides an upper 
bound for the minimax risk $R_M(\overline{S})$ over $\overline{S}$ defined by
\[
R_M(\overline{S})=\inf_{\widetilde \gs} \sup_{\gs\in\overline{S}}\E\cro{\gh^{2}(\gs,\widetilde \gs)},
\]
where the infimum runs among all possible estimators $\widetilde \gs$ of $\gs$. In particular, 
$R_M(\overline{S})\le\RR_{\rho}(\overline{S})$ where $\RR_{\rho}(\overline{S})=
\sup_{\gs\in\overline{S}}\E\cro{\gh^{2}(\gs,\widehat{\gs})}$ denotes the maximal risk of any $\rho$-estimator $\widehat{\gs}$ over $\overline{S}$. Restricting ourselves to $\rho$-estimators that satisfy (\ref{risk3}) and using \eref{Eq-dims}, we get
\[
R_M(\overline{S})\le\RR_{\rho}(\overline{S})\le D(\overline{S})+c_4\le\overline{D}(\overline{S})+c_4.
\]
It follows that it suffices to bound $D(\overline S)$ (or $ \overline{D}(\overline{S})$) from above in order to control the minimax risk over $\overline{S}$ which can be done by using the bounds of the previous sections and results in the next corollary.
%
\begin{cor}\label{C-main00}
If $\overline{S}$, viewed as a set of real-valued functions on $\overline{\X}$ as defined by (\ref{Eq-Xbar}) is  VC-subgraph with index $\overline V$, then 
\begin{equation}\label{risk-VC2}
R_M(\overline{S})\le\RR_{\rho}(\overline{S})\le c_4+
C\overline V\left[1+\log_+\left(n/\overline V\right)\right].
\end{equation} 
If $\overline{S}$ is a totally bounded nonempty subset of $\LL_{0}$ with metric dimension bounded 
by $\widetilde D(\cdot)$, then 
\[
R_M(\overline{S})\le\RR_{\rho}(\overline{S})\le c_4+2\left(c_1+2c_2\right)\overline\eta^{2}\quad
\mbox{with}\quad\overline \eta=\inf\ac{\eta>0\,\left|\,\eta^{-2}\widetilde D(\eta)\le8c_0^2/131\right.}.
\]
\end{cor}

The bound (\ref{risk-VC2}) for $\RR_{\rho}(\overline{S})$ that we derived from Theorem~\ref{main1} involves a logarithmic factor while one would rather expect a bound of the form $\RR_{\rho}(\overline{S})\le C\overline V $. If we compare this result to~\eref{def-entropy} (with $\eta=z\sqrt{n}$ in order to make $\gh$ and the $\IL_{2}(Q)$-distance comparable), we see that this phenomenon is due to the entropy bound~\eref{unifB} which is uniform with respect to $y$. An entropy bound of the form
\begin{equation}\label{unifB2}
\log N\pa{\Y,Q,z}\le C\overline V  \log_{+}\pa{{Ay/\sqrt{n}\over z}}\quad\mbox{for all }y
\mbox{ and }z>0
\end{equation}
would lead to the expected inequality $\RR_{\rho}(\overline{S})\le C\overline V $. Unfortunately, 
we do not know whether a bound such as~\eref{unifB2} is true or not but there exists at least 
one situation where this extra logarithmic factor can be removed: when $\overline{S}$ consists 
of piecewise constant functions. For the sake of simplicity we shall only consider the density 
framework described in Section~\ref{A2}. 

{\bf Histograms:}
Assume that we are in the density framework and have at hand some countable partition $\I$ of 
$\X$ such that $0<\mu(I)<+\infty$ for all $I\in\I$. We consider the set 
$\overline{S}=\overline{S}_{\I}$ of all densities on $(\X,\A,\mu)$ which are piecewise constant 
on each element $I$ of $\I$, which means that
\begin{equation}
\overline{S}_{\I}=\ac{\left.t=\sum_{I\in\I}\frac{t_{I}}{\mu(I)}\1_{I}\,\right|\,t_{I}\ge0
\quad\mbox{for all }I\in\I\quad\mbox{and}\quad\sum_{I\in\I}t_{I}=1}.
\label{Eq-Histo1}
\end{equation}
If we choose for $S$ a subset of $\overline{S}_{\I}$, the resulting $\rho$-estimator $\widehat\gs$ 
will therefore be an histogram-type estimator and the following result, to be proved in Section~\ref{P6}, holds. 
\begin{prop}\label{histo}
Let $S$ be a countable subset of $\overline{S}_{\I}$ and $s$ be a density with respect to $\mu$ such that $\J(s)=\left\{I\in\I\,\left|\,\int_{I}s\,d\mu>0\right.\right\}$ is finite. Then, for any $\overline{s}\in S$, $D^S(\gs, \overline{\gs})\le6c_0^{-2}|\J(s)|$.
\end{prop}
There are various potential applications of this result but let us focus here on the case of a finite measure $\mu$ and a finite partition $\I$ so that $|\J(s)|\le|\I|$ for all $\gs\in\LL_0$. It then follows from the previous proposition and (\ref{Eq-DS}) that $D^S\le6c_0^{-2}|\I|$ for all countable subsets $S$ of $\overline{S}_{\I}$, hence $D(\overline{S}_{\I})\le6c_2c_0^{-2}|\I|$. In this case $\overline{S}_{\I}$ is a subset of a linear space with dimension $ |\I|$ and is therefore VC-subgraph with index not larger than $|\I|+2$. Comparing our bound for $D(\overline{S}_{\I})$ with the one provided by Theorem~\ref{main1} for $\overline D(\overline{S}_{\I})$ we see that the extra logarithmic factor has disappeared. Nevertheless Proposition~\ref{histo} only provides an upper bound for $D(\overline{S}_{\I})$ and not for $\overline D(\overline{S}_{\I})$.

\section{Connection with the Maximum Likelihood Estimator}\label{S-CMLE}
Throughout this section, we consider the problem of density estimation from $n$ i.i.d.\ observations $X_1,\ldots,X_n$ as described in Section~\ref{A2}. Our aim is to show that $\rho$-estimation may 
recover the classical MLE in various situations.

\subsection{Regular parametric models}\label{S-CMLE1}
We consider here a parametric set of densities $\{t_{\theta},\,\theta\in\Theta'\}$ on the measured 
space $(\X,\A,\mu)$ indexed by some open subset $\Theta'$ of $\R^{d}$ and such that the 
mapping $\theta\mapsto P_\theta=t_\theta\cdot\mu$ is one-to-one. Our model is $\overline S=
\{t_{\theta},\,\theta\in\Theta\}$ for some $\Theta\subset\Theta'$ and we set $\norm{t}_{\infty}=
\sup_{x\in\X}\ab{t(x)}$ for any function $t$ on $\X$. There have been a number of different 
assumptions for the ``regularity" of a parametric set of densities. Here we mean a modern version 
of the notion, as inspired by the pioneering works of Le~Cam~\citeyearpar{MR0267676} and 
H\'ajek~\citeyearpar{MR0400513}. One may, for instance, use the definition given in Chapter~I, Section~7.1 of Ibragimov and Has{'}minski{\u\i}~\citeyearpar{MR620321}.
\begin{ass}\label{MLE}\mbox{}
\begin{enumerate}
	\item[($i$)] The parameter set $\Theta$ is a compact and convex subset of $\Theta'$ and 
	the true density $s$ is an element $t_{\vartheta}\in\overline{S}$ such that $\vartheta$ is an 
	interior point of $\Theta$.\vspace{1mm}
	\item[($ii$)] The parametric family $\{t_{\theta},\,\theta\in\Theta'\}$ is regular and the Fisher 
	Information matrix is invertible on $\Theta$.\vspace{1mm} 
	\item[($iii$)] There exists a constant $A_1$ such that
\[
\norm{\sqrt{t_{\theta}\over t_{\theta'}}-\sqrt{t_{\overline\theta}\over t_{\theta'}}}_{\infty}\le  
A_1\left|\overline\theta-\theta\right|\quad\mbox{for all }\theta,\,\overline\theta
\mbox{ and }\theta'\in\Theta.\vspace{1mm}
\]
	\item[($iv$)] With probability tending to one when $n$ goes to infinity, there exists a maximum likelihood estimator $\widetilde \theta_{n}$ which is consistent.
\end{enumerate}
\end{ass}
One can then prove (in Section~\ref{P14}):
\begin{thm}\label{thm-MLE}
Let $\overline S$ be a parametric model of densities satisfying Assumption~\ref{MLE} and $S$ an arbitrary countable and dense subset of $\overline S$. With probability tending to 1 as $n$ tends to infinity, $t_{\widetilde \theta_{n}}$ belongs to ${\rm Cl}\!\left(\EE(\gX,S)\strut\right)$ and is therefore a $\rho$-estimator.
\end{thm}
This result shows that when the model is regular enough and contains the true density, 
$\rho$-estimation allows to recover the MLE, at least when $n$ is large enough. The numerical study of Mathieu Sart~\citeyearpar{sart2016} on very simple statistical models $\overline S$ seems to indicate that our procedure allows to recover the MLE in almost all simulations even when the number of observations $n$ is small. Consequently, there seems to be some space for  improvement in Theorem~\ref{thm-MLE}. At least, as we shall see in the next section, Assumption~\ref{MLE} could be weakened. 

\subsection{A direct computation on a non-regular model}\label{S-CMLE2}
In this section, we give an example of a non-regular statistical model (in the usual statistical sense) on which we also recover the MLE with probability 1. This means that the connections between the MLE and $\rho$-estimators are not restricted to situations where the parameter is estimated at the usual parametric rate $n^{-1/2}$. 

Let us consider the problem of estimating $\theta$ from the observation of a sample 
$X_{1},\ldots,X_{n}$ of an unknown density $s$ belonging to the model $\overline S=
\{q_{\theta}=\1_{[-1/2+\theta,1/2+\theta]},\,\theta\in\R\}$. Elementary calculations show that
\begin{equation}
h^{2}(q_{\theta},q_{\theta'})=\ab{\theta-\theta'}\wedge1\quad\mbox{for all }\theta,\theta'\in\R,
\label{Eq-36}
\end{equation}
hence $S=\{q_{\theta},\ \theta\in\Q\}$ provides a countable and dense subset of $\overline S$.
\begin{prop}
Assume that $s\in \overline S$ and let $X_{(1)}<\ldots<X_{(n)}$ be the order statistics corresponding to our sample. The estimator $\widetilde \theta_{n}=\left(X_{(1)}+X_{(n)}\right)/2$ of $\theta$ maximizes the likelihood and $q_{\widetilde \theta_{n}}$ is a $\rho$-estimator of $s$.
\end{prop}
\begin{proof}
The fact that the likelihood $\theta\mapsto\prod_{i=1}^{n}\1_{[X_{i}-1/2,X_{i}+1/2]}(\theta)$ is maximal for $\theta=\widetilde \theta_{n}$ is easy to check. It remains to show that if $s\in \overline S$, $q_{\widetilde \theta_{n}}$ belongs to ${\rm Cl}\!\left(\EE(\gX,S)\strut\right)$ with probability 1. 

In the sequel, we only consider points $\theta$ and $\theta'$ that belong to $\Q$. Since the density $q_{0}$ is even, $\rho(q_{\theta},(q_{\theta'}+q_{\theta})/2)=\rho(q_{\theta'},(q_{\theta'}+q_{\theta})/2)$ and  therefore
\[
\gT(\gX,\gq_{\theta},\gq_{\theta'})={1\over \sqrt{2}}\sum_{i=1}^{n}\psi\pa{\sqrt{{q_{\theta'}\over q_{\theta}}}(X_{i})}\quad\mbox{for all }\theta,\theta'\in\Q.
\]
For all $\theta'$, $q_{\theta'}(\cdot)$ takes its values in $\{0,1\}$ and for all $i\in\{1,\ldots,n\}$, $q_{\theta'}(X_{i})=1$ if and only if $\theta'\in [X_{i}-1/2,X_{i}+1/2]$. It follows that 
$\theta'\in\Q$ and $q_{\theta'}(X_{i})=1$ for all $i$ if and only if $\theta'\in\widehat \Theta$ 
with
\[
\widehat \Theta=\cro{X_{(n)}-1/2,X_{(1)}+1/2}\cap \Q.
\]
This random subset of $\Q$ is non-void since, when $s\in \overline S$, the diameter of $\widehat \Theta$ is $\Delta(\gX)=1-\pa{X_{(n)}-X_{(1)}}>0$ $\P_{\gs}$-a.s. For all 
$\theta'\in \widehat \Theta$ and $i\in\{1,\ldots,n\}$
\[
\psi\pa{\sqrt{{q_{\theta'}\over q_{\theta}}}(X_{i})}=\psi\pa{\sqrt{{1\over q_{\theta}(X_{i})}}}= \1_{\{q_{\theta}=0\}}(X_{i}),
\]
which implies that  $\gT(\gX,\gq_{\theta},\gq_{\theta'})\ge 1/\sqrt{2}$ if $\theta\not\in\widehat \Theta$, hence
\[
\gup(S,\gq_{\theta})=\sup_{\theta'\in\Q}\gT(\gX,\gq_{\theta},\gq_{\theta'})\ge\sup_{\theta'\in\widehat \Theta}
\gT(\gX,\gq_{\theta},\gq_{\theta'})\ge {1\over \sqrt{2}}\quad\mbox{for }\theta\not\in \widehat \Theta.
\]
For $\theta\in \widehat \Theta$, $q_\theta(X_i)=1$ for all $i$ so that $q_{\theta'}(X_i)\le 
q_\theta(X_i)$ for all $i$ and $\gT(\gX,\gq_{\theta},\gq_{\theta'})\le0$ whatever $\theta'$. It follows that
\[
\gup(S,\gq_{\theta})=\sup_{\theta'\in \Q}\gT(\gX,\gq_{\theta},\gq_{\theta'})=0.
\]
Hence, $\theta\mapsto \gup(S,\gq_{\theta})$ is minimum for the elements $\theta\in\widehat\Theta$
 and $\{\gq_{\theta},\ \theta\in \widehat \Theta\}\subset \EE(\gX,S)$. Since $\widetilde \theta_{n}$ 
 belongs to the closure of $\widehat \Theta$ (with respect to the Euclidean distance) and since
for any sequence  $(\theta_j)_{j\ge1}$ converging towards $\widetilde \theta_{n}$, $\gq_{\theta_{j}}$ 
converges towards $\gq_{\widetilde \theta_{n}}$ with respect to the Hellinger distance by
(\ref{Eq-36}), $\gq_{\widetilde \theta_{n}}$ belongs to ${\rm Cl}\!\left(\EE(\gX,S)\strut\right)$ and is 
therefore a $\rho$-estimator.
\end{proof}

\subsection{Risk bounds under entropy with bracketing}\label{S-CMLE3}
Since, for some specific models of densities $\overline S$ that contain the true density $s$, 
the MLE is a $\rho$-estimator with probability close to 1, it is natural to wonder how to compare 
the performance of these two estimators on more general models $\overline S$, possibly not 
containing $s$. One way to do so is to compare their risk bounds. In the literature, the risk 
bounds which are established for the MLE  usually take the following form 
\begin{equation}\label{eq-MLE-RB}
C\E\cro{h^{2}(s,\tilde s)}\le K(s,\overline S)+\tau_{n}^{2}\vee n^{-1},
\end{equation}
where $C$ is a positive universal constant and $K(s,\overline S)=\inf_{t\in \overline S}K(s,t)$. As to the number $\tau_{n}^{2}$, which usually corresponds to the maximal risk over $\overline S$, it is obtained by solving an equation depending on the  bracketing entropy of $\overline S$. Such a result appears as Theorem~7.11 in Massart~\citeyearpar{MR2319879}. The aim of this section is to establish an analogue of (\ref{eq-MLE-RB}) with the same value of $\tau_{n}$ for our $\rho$-estimator. Our assumptions are similar to those used by Massart with a slight modification (replacing his assumption $(M)$ by $(i)$ below) which corresponds to the fact that we only use countable models.
\begin{ass}\label{A-MLE}
There exists a countable subset $S$ of $\overline S$ with the following properties.
\begin{itemize}
\item[$(i)$]  For all densities $s$ on $(\X,\A,\mu)$, $K(s,S)=K(s,\overline S)$.
\item[$(ii)$] For all $\sigma>0$ and $\overline s\in S$ there exists a non-increasing mapping 
$z\mapsto \HH_{[\ ]}^{S}(\overline s,\sigma,z)$ from $(0,+\infty)$ into $(0,+\infty)$ and 
a family $\I(\overline s,\sigma,z)$ of pairs of non-negative measurable functions on the 
measured space $(\X,\A,\mu)$ such that \[
\log 2\le \log \ab{\I(\overline s,\sigma,z)}\le \HH_{[\ ]}^{S}(\overline s,\sigma,z)\quad\mbox{for all }z>0.
\]
Moreover, for all $\gt\in \B^{S}(\overline \gs,\sigma\sqrt{n})$ one can find a pair $(t_{L},t_{U})\in \I(\overline s,\sigma,z)$ such that $t_{L}\le t\le t_{U}$ and 
\[
{1\over 2}\int\pa{\sqrt{t_{U}}-\sqrt{t_{L}}}^{2}d\mu\le z^{2}.
\]
\item[$(iii)$] There exists a non-decreasing function $\phi$ from $(0,+\infty)$ into $(0,+\infty)$  such that $x\mapsto \phi(x)/x$ is non-increasing on $(0,+\infty)$ and for which
\[
\sup_{\overline s\in S}\int_{0}^{\sigma}\sqrt{\HH_{[\ ]}^{S}(\overline s,\sigma,z)}\ dz\le \phi(\sigma).
\]
\end{itemize}
\end{ass}
From these assumptions, we can derive the following result to be proved in Section~\ref{P15}.
\begin{thm}\label{T-crochet}
Let $X_{1},\ldots,X_{n}$ be an $n$-sample with values in $(\X,\A,\mu)$ and density $s$ with respect to $\mu$. Let $\overline S$ be a model of densities satisfying Assumption~\ref{A-MLE} and
\[
\tau_{n}=\inf\left\{\sigma>0,\ \phi(\sigma)\le\sqrt{n}\sigma^{2}\right\}.
\]
Then there exist universal constants $C,C'>0$ such that 
\begin{equation}\label{borneD-S}
d(S)=\sup_{\overline \gs\in S}D^{S}(\overline \gs,\overline \gs)\le\left(Cn\tau_{n}^{2}\right)\vee 1\ \ \end{equation}
and for any $\rho$-estimator $\widehat s$ of $s$
\begin{equation}\label{eq-BR-rho-MLE}
C'\E\cro{h^{2}(s,\widehat s)}\le K(s,\overline S)+\tau_{n}^{2}\vee n^{-1}.
\end{equation}
\end{thm}
\subsection{Histogram estimators}\label{S-CMLE4}
Let us go back to the framework that we introduced at the end of Section~\ref{C3c} which means
that we consider the problem of estimation of a density $s$ with respect to $\mu$ using the model
of piecewise constant functions $\overline{S}_{\I}$ defined by (\ref{Eq-Histo1}) with a countable
partition $\I$ of $\X$ satisfying $0<\mu(I)<+\infty$ for all $I\in\I$. Note that this model is identifiable
so that $(\overline{S}_{\I},h)$ is a metric space. The model $\overline{S}_{\I}$ can then be identified
with the unit simplex $\scr{S}$ in $[0,1]^{|\I|}$ since by (\ref{Eq-Histo1}), for $t\in\overline{S}_{\I}$,
$\sum_{I\in\I}t_I=1$ with $t_I=\int_{I}t(x)\,d\mu(x)$. With this identification, the metric space 
$(\overline{S}_{\I},h)$ is topologically equivalent to the separable Euclidean simplex $\scr{S}$
so that $\overline{S}_{\I}$ is also separable for the distance $h$. We finally set $\overline{S}=
\{\gt=(t,t,\ldots,t),t\in\overline{S}_{\I}\}\subset\left(\overline{S}_{\I}\right)^n$.

Given $n$ i.i.d.\ observations $\Etc{X}{n}$ with values in $\X$ and density $t\in\overline{S}_{\I}$ 
with respect to $\mu$, we set $N_I=\sum_{i=1}^n\1_{I}(X_i)$ for $I\in\I$. The vector 
$(N_I)_{I\in\I}\in[0,n]^{|\I|}$ is a multinomial vector with parameter $(t_I)_{I\in\I}$, the MLE over 
$\scr{S}$ is then given by $\{\widehat{t}_I,\,I\in\I\}$ with $\widehat{t}_I=N_I/n$ and the 
corresponding density estimator $\widehat{t}=\sum_{I\in\I}\left(\widehat{t}_I/\mu(I)\right)\1_{I}$ 
of $t$ is the MLE on the model $\overline{S}_{\I}$. It is also the histogram estimator of 
the true density $s$ with respect to the partition $\I$ of $\X$.
%
\begin{prop}\label{P-discrete}
The histogram estimator $\widehat{t}=\sum_{I\in\I}[N_I/(n\mu(I))]\1_{I}$ is a $\rho$-estimator 
built on the model $\overline{S}$.
\end{prop}
\begin{proof}
For $t,u\in\overline{S}_{\I}$, $\rho(t,u)=\sum_{I\in\I}\sqrt{t_Iu_I}$ and $(du/dt)(x)=u_I/t_I$ for $x\in I$ 
with the convention $0/0=1$. The definition (\ref{def-T}) of the function $\gT$ therefore implies that
\begin{eqnarray*}
\gT(\gX,\gt,\gu)&=&\frac{n}{2\sqrt{2}}\sum_{I\in\I}\left[\sqrt{t_Iu_I+u_I^2}-\sqrt{t_Iu_I+
t_I^2}\right]+\frac{1}{\sqrt{2}}\sum_{I\in\I}\psi\left(\sqrt{\frac{u_I}{t_I}}\right)N_I
\\&=&\frac{n}{2\sqrt{2}}\sum_{I\in\I}\left[\sqrt{t_I+u_I}
\left(\sqrt{u_I}-\sqrt{t_I}\right)+2\widehat{t}_I\frac{\sqrt{u_I/t_I}-1}{\sqrt{1+(u_I/t_I)}}\right].
\end{eqnarray*}
It follows, setting $\J=\left\{I\in\I\,\left|\,\widehat{t}_I>0\right.\right\}=\left\{I\in\I\,\left|\,N_I>0\right.\right\}$, that
\[
\gT(\gX,\widehat{\gt},\gu)=\frac{n}{2\sqrt{2}}\left[\sum_{I\in\J}\left(3\widehat{t}_I+u_I\right)
\frac{\sqrt{u_I/\widehat{t}_I}-1}{\sqrt{1+(u_I/\widehat{t}_I)}}+\sum_{I\in\J^c}u_I\right].
\]
Setting $x_I=u_I/\widehat{t}_I$ for $I\in\J$ and $\sigma(u)=\sum_{I\in\J^c}u_I$, we finally get,
\[
\gT(\gX,\widehat{\gt},\gu)=\frac{n}{2\sqrt{2}}\left[\sum_{I\in\J}\widehat{t}_IG(x_I)
+\sigma(u)\right]\quad\mbox{with }\;G(x) = \frac{(\sqrt{x}-1) \left(3+x\right)}{\sqrt{1+x}}\quad\mbox{for all }x\ge0.
\]
Since, for all $x>0$,
\[
G'(x)=\frac{2x^2-x^{3/2}+3x+\sqrt{x}+3}{2\sqrt{x}(1+x)^{3/2}}\quad\mbox{ and }\quad G''(x)=\frac{\left(x^{5/2}-5x^{3/2}-9x-3\right)\sqrt{1+x}}{4x(1+x)^3\sqrt{x}},
\]
we can see that $G'(1)=\sqrt{2}$ and $(1-x)\left(G'(x)-\sqrt{2}\right)>0$ for all $x\ne1$. It follows that if 
$u'\ne u$ and $x'_I=u'_I/\widehat{t}_I$ for $I\in\J$ we get
\begin{equation}
\widehat{t}_I\left[G(x_I)-G(x'_I)\right]<\sqrt{2}\left(u_I-u'_I\right)
\quad\mbox{if either }\;x_I>x'_I\ge1\;\mbox{ or }\;x_I<x'_I\le1.
\label{Eq-G}
\end{equation}
We now consider two cases.\\
--- If $\sigma(u)>0$ there exists some $u_{I'}>0$ for $I'\in\J^c$ and some $I\in\J$ with $u_I<\widehat{t}_I$. 
It is therefore possible to decrease $u_{I'}$ to $u_{I'}-\varepsilon\ge0$ and increase $u_I$ to 
$u_I+\varepsilon\le\widehat{t}_I$ for $\varepsilon>0$ small enough which implies for 
$\gT$ an increase larger than $n2^{-3/2}[\sqrt{2}-1]\varepsilon>0$. It follows that 
\[
\gT(\gX,\widehat{\gt},\gu)<\sup_{\gt\in\overline{S}}\gT(\gX,\widehat{\gt},\gt)
\quad\mbox{for all }\gu=(u,u,\ldots, u)\mbox{ such that }\sigma(u)>0.
\]
--- If $\sigma(u)=0$ and $u\ne\widehat{t}$ one can find $J,J'\in\J$ with $x_J<1<x_{J'}$, hence, for 
$\varepsilon>0$ small enough, $x'_J=x_J+\varepsilon/\widehat{t}_J\le1$ and $x'_{J'}=x_{J'}-\varepsilon/\widehat{t}_{J'}\ge1$. For such an $u$, we define $u'$ by $u'_J=u_J+\varepsilon$, $u'_{J'}=u_{J'}-\varepsilon$ and $u'_I=u_I$ for all other $I\in\I$ so that $\sum_{I\in\I}u'_I=\sum_{I\in\I}u_I=1$ and $u'_{\I}\in\scr{S}$
as required. It then follows from (\ref{Eq-G}) that
\[
\gT(\gX,\widehat{\gt},\gu)-\gT(\gX,\widehat{\gt},\gu')<(n/2)\left[u_J-u'_J+u_{J'}-u'_{J'}\right]=0.
\]
It follows that, for all $u\ne\widehat{t}$, $\gT(\gX,\widehat{\gt},\gu)<\sup_{\gt\in\overline{S}}\gT(\gX,\widehat{\gt},\gt)$ and finally
\[
\sup_{\gt\in\overline{S}}\gT(\gX,\widehat{\gt},\gt)=\gT(\gX,\widehat{\gt},\widehat{\gt})=0.
\]
Since $|\J|\le n$ and the mapping
\[
y\mapsto\frac{\sqrt{u_I}-\sqrt{\widehat{t}_I+y}}{\sqrt{\widehat{t}_I+y+u_I}}
\left(3\widehat{t}_I+y+u_I\right)
\]
is continuous at 0 uniformly with respect to $u_I\in[0,1]$ for $\widehat{t}_I>0$, replacing $\overline{S}$ by a dense subset $S$ and $\widehat{\gt}$ by a close enough approximation 
$\widehat{\gt}_\varepsilon$ with $\varepsilon>0$ and $\sigma\left(\widehat{t}_\varepsilon\right)=0$ leads to 
\[
\gup\left(S,\widehat{\gt}_\varepsilon\right)=
\sup_{\gu\in S}\gT(\gX,\widehat{\gt}_\varepsilon,\gu)<\varepsilon.
\]
Since $\varepsilon$ is arbitrary, this proves that the MLE $\widehat{\gt}$ belongs to 
${\rm Cl}\!\left(\EE(\gX,S)\strut\right)$.
\end{proof}
This applies in particular to the example of Section~\ref{C3c}. It also applies to the case of 
the $X_i$ taking their values in a finite set $\X=\{a_1,\ldots,a_r\}$, $r>1$ (or even a countable set 
$\X=\{a_j,\,j\in\N\}$) and to the estimation of the density $s$ of the $X_i$ with respect to the 
counting measure $\mu$ on $\X$. Then the MLE $\widehat{s}$ over the set $\overline{S}$ 
of all densities on $\X$ is given by $\widehat{s}(a_j)=N_j/n$ with 
$N_j=\sum_{i=1}^n\1_{a_j}(X_i)$ for $1\le j\le r$ (or $j\in\N$) and it is a $\rho$-estimator 
with respect to the model $\overline{S}$. In particular, if the $X_i$ are i.i.d.\ Bernoulli variables 
with parameter $\theta$, the empirical mean is a $\rho$-estimator.

\section{Examples\label{E}}
\subsection{Homoscedastic regression with fixed design\label{E1}}
In this section, we consider the statistical framework described in Section~\ref{A3}. Our 
aim is therefore to estimate the function $\gf$ from the observation of the $X_{i}$. 

The choice of a model $\overline{S}_{q,F}$ corresponds here to those of a density $q$ 
(with respect to the Lebesgue measure $\mu$) to approximate $p$ and of a subset $F$ 
of $\R^n$ to approximate $\gf$. More precisely, given $q$ and $F$, we define the model 
$\overline{S}_{q,F}$ as the set of functions from $\R^n$ to $\R^n$ given by
\[
\overline{S}_{q,F}=\ac{\left.\gx\mapsto\gq_{\gg}(\gx)=\pa{q(x_1-g_{1}),\ldots,q(x_n-g_{n})}\,\right|\,\gg\in F},
\]
which is clearly identifiable.
\begin{ass}\label{hypoQ}
The density $q$ is unimodal. 
\end{ass}
\begin{thm}\label{thmVC}
Let Assumption~\ref{hypoQ} be satisfied. If $F$, viewed as a class of functions on 
$\{1,\ldots,n\}$, is VC-subgraph with index $\overline V$, then 
\begin{equation}\label{borne-DVC}
\overline D(\overline{S}_{q,F})\le C\overline V \left[1+\log_{+}\left(n/\overline V \right)\right]
\end{equation}
and the estimator $\widehat \gs=\gq_{\widehat \gf}$ built in Section~\ref{C1} and based on a countable and dense subset of $\overline{S}_{q,F}$ satisfies, for any 
density $p$ and vector $\gf\in\R^{n}$,
\begin{equation}
\P_{\gs}\cro{C\gh^{2}\pa{\gp_{\gf},\gq_{\widehat \gf}}\le\gh^{2}\pa{\gp_{\gf},\overline{S}_{q,F}}+\overline V\cro{1+\log_+\!\pa{{n/\overline V}}}+\xi}\ge 1-e^{-\xi}\quad\mbox{for all }\xi>0.
\label{Eq-risk2}
\end{equation}
\end{thm}
\begin{proof}
A vector $\gg\in F$ and the element $\gq_{\gg}\in \overline{S}_{q,F}$ can both be viewed as functions on $\overline \X=\{1,\ldots,n\}\times \R$ defined respectively, for $\overline x=(i,x)
\in\overline \X$, by $\gg(\overline x)=g_{i}$ and $\gq_{\gg}(\overline x)=q(x-g_{i})$. Under Assumption~\ref{hypoQ}, it follows from the properties $(iii), (i), (vi)$ of Proposition~\ref{perm}, that $\overline{S}_{q,F}$ is VC-subgraph with index not larger than $C'\overline V$. We conclude with Theorem~\ref{main1}.
\end{proof}
Since this proof relies on the fact that $\overline{S}_{q,F}$ is VC-subgraph, Assumption~\ref{hypoQ} can be replaced by ``$q$ is multimodal with no more than $k$ modes". Indeed, in this case the set 
$\overline{S}_{q,F}$ is still VC-subgraph but with index bounded by $C'(k)\overline V$ as noticed in the remark at the end of Section~\ref{S}. It follows that the constants $C$ appearing in~\eref{borne-DVC} and \eref{Eq-risk2} now depend on $k$.

There are various ways of applying the previous theorem according to the type of 
bound we would like to get. Let us first note that, by the triangular inequality,  
$\gh\pa{\gp_{\gf},\gq_{\gg}}\le \gh\pa{\gp_{\gf},\gp_{\gg}}+\gh\pa{\gp_{\gg},\gq_{\gg}}$ and, by translation invariance, $\gh^2\pa{\gp_{\gg},\gq_{\gg}}=nh^2(p,q)$ so that $\gh^{2}\pa{\gp_{\gf},\overline{S}_{q,F}}\le2\gh^{2}\pa{\gp_{\gf},\overline{S}_{p,F}}+2nh^2(p,q)$. Therefore (\ref{Eq-risk2}) implies that, for all $\xi>0$,
\[
\P_{\gs}\cro{C\gh^{2}\pa{\gp_{\gf},\gq_{\widehat \gf}}\le nh^2(p,q)+
\inf_{\gg\in F}\gh^2\pa{\gp_{\gf},\gp_{\gg}}+\overline V\pa{1+\log_+\pa{{n/\overline V}}}+\xi}\ge1-e^{-\xi}
\]
and, by the same argument,
\[
\P_{\gs}\cro{C\gh^{2}\pa{\gp_{\gf},\gq_{\widehat \gf}}\le nh^2(p,q)+
\inf_{\gg\in F}\gh^2\pa{\gq_{\gf},\gq_{\gg}}+\overline V\pa{1+\log_+\pa{n/\overline V}}+\xi}
\ge1-e^{-\xi}.
\]
Noticing that $\gh^{2}\pa{\gq_{\gf},\gq_{\widehat \gf}}\le2\gh^{2}\pa{\gq_{\gf},\gp_{\gf}}
+2\gh^{2}\pa{\gp_{\gf},\gq_{\widehat \gf}}$, we also derive similarly that
\begin{equation}
\P_{\gs}\cro{C\gh^{2}\pa{\gq_{\gf},\gq_{\widehat \gf}}\le nh^2(p,q)+
\inf_{\gg\in F}\gh^2\pa{\gq_{\gf},\gq_{\gg}}+\overline V\pa{1+\log_+\pa{n/\overline V}}+\xi}\ge1-e^{-\xi}.
\label{Eq-risk3}
\end{equation}
This last formula provides a risk bound for the estimation of $\gf$ by $\widehat{\gf}$ when the loss 
function takes the special form $\ell(\gf,\widehat{\gf})=\gh^{2}\pa{\gq_{\gf},\gq_{\widehat \gf}}$
for a known density $q$:
\[
C\E\cro{\ell\left({\gf},\widehat \gf\right)}\le nh^2(p,q)+
\inf_{\gg\in F}\ell\left(\gf,\gg\right)+\overline V\pa{1+\log_+\pa{n/\overline V}}.
\]
If $p$ is known so that we can set $q=p$ we find the usual ``bias plus variance" risk bound without the $nh^2(p,q)$ term, $\overline V\pa{1+\log_+\pa{n/\overline V}}$ playing here the
role of a variance term. This shows that the price to pay for not knowing the density 
$p$ of the errors and replacing it by $q$ is an additional bias term of order $nh^2(p,q)$.

To make this last risk bound more precise, we introduce the following definition.
\begin{defi}\label{D-order}
We shall say that a density $q$ is of order $\alpha\in (-1,1]$ if it satisfies
\begin{equation}\label{regul}
a_q\cro{\ab{u-v}^{1+\alpha}\wedge A_q^{-1}}\le h^{2}(q_{u},q_{v})\le A_q\cro{\ab{u-v}^{1+\alpha}\wedge A_q^{-1}}\ \ \mbox{for all}\ \ u,v\in\R
\end{equation}
and some constants $A_q\ge a_q>0$ depending on $q$.
\end{defi}
The reader can find in Ibragimov and Has{'}minski{\u\i}~\citeyearpar{MR620321} Chapter VI p.~281 
some sufficient conditions on the density $q$ to ensure that \eref{regul} holds. For illustration, we present 
here some examples borrowed from these authors. The density $q=\1_{[-1/2,1/2]}$ is of order 0. 
For $\alpha\in (-1,1)$, the density $q(x)=[2(1+\alpha)]^{-1}(1-|x|)^{\alpha}\1_{[-1,1](x)}$ is of order 
$\alpha$ and so is $q(x)=C(\alpha)\exp\cro{-|x|^{\alpha/2}}$ for $\alpha\in (0,1)$. For $\alpha>1$, 
this latter density is of order 1. If the translation model $\theta\mapsto q(\cdot-\theta)$ is regular 
(which, in this case, is equivalent to the fact that $\sqrt{q}$ is differentiable in quadratic mean), 
it is of order 1.

Let us now set, for $\alpha\in (-1,1]$, $\gg,\gg'\in \R^{n}$ and $G\subset \R^{n}$, 
\begin{equation}
d_{1+\alpha}(\gg,\gg')=\sum_{i=1}^{n}\pa{\ab{g_{i}-g'_{i}}^{1+\alpha}\wedge A_q^{-1}}
\qquad\mbox{and}\qquad d_{1+\alpha}(\gg,G)=\inf_{\gg'\in G}d_{1+\alpha}(\gg,\gg').
\label{eq-d1}
\end{equation}
Applying Theorem~\ref{thmVC} with the bound (\ref{Eq-risk3}) leads to the following result.
\begin{cor}\label{cor4}
Let $F$ be a subset of $\R^{n}$ which is VC-subgraph with index $\overline V$ and $q$ 
be a density on $\R$ of order $\alpha\in (-1,1]$ which satisfies Assumption~\ref{hypoQ}. 
Then the estimator $\widehat \gs=\gq_{\widehat \gf}$ satisfies, for all $\gf\in\R^{n}$ and 
all $\xi>0$,
\begin{equation}\label{borne42}
\P_{\gs}\cro{d_{1+\alpha}(\gf,\widehat \gf)\le C'(q)\!
\pa{d_{1+\alpha}(\gf,F)+nh^2(p,q)+\overline V \pa{1+\log_{+}\pa{{n/\overline V}}}+\xi}}\ge
1-e^{-\xi}.
\end{equation}
\end{cor}
To comment on this result, let us consider the simple example of the shift model for i.i.d.\ observations. 
Assume that the $X_{i}$ are i.i.d.\ with common density $p(\cdot-\theta)$ for some unknown parameter 
$\theta\in \R$ but a known density $p$ that we assume to be of order $\alpha\in (-1,1]$ and to satisfy 
Assumption~\ref{hypoQ}. In this case $f_{i}=\theta$ for all $i$ and it is natural to fix $q=p$ and consider 
as a model for $\gf$ the linear span $F$ of $(1,\ldots,1)$ in $\R^{n}$. The distance 
$d_{1+\alpha}(\gf,\gg)$ between $\gf=(\theta,\ldots,\theta)$ and an element 
$\gg=(\theta',\ldots,\theta')$ of $F$ becomes
\[
d_{1+\alpha}(\gf,\gg)=n\left(\ab{\theta-\theta'}^{1+\alpha}\wedge A_p^{-1}\right)
\]
and we can deduce from~\eref{borne42} that our estimator $\widehat \gf=(\widehat \theta,\ldots,\widehat \theta)$ satisfies, for $n$ large enough and with a probability close to 1, 
\[
|\theta-\widehat\theta|\le C(p,\alpha)[(\log n)/n]^{1/(1+\alpha)}.
\]
As soon as $\alpha\in (-1,1)$, the rate we get improves on the usual parametric one $1/\sqrt{n}$ 
achieved by the classical least-squares estimator (under suitable moment conditions on the 
$\eps_{i}$). Though faster, this rate can still be improved by a logarithmic factor, as, in fact, 
the maximum likelihood estimator can achieve the rate $n^{-1/(1+\alpha)}$ (we refer to Theorem~6.3 
p.~314 of the book by Ibragimov and Has{'}minski{\u\i}~\citeyearpar{MR620321}). We do not 
know whether this extra logarithmic factor is due to our techniques or if it is really necessary in order to get a robust estimator (which is not the case for the MLE) with unbounded models. If we were ready to make the additional assumption that the set $F$ is totally bounded, we could use a T-estimator or a 
$\rho$-estimator based on a finite model and use the arguments of Section~\ref{C3b} to get a risk which would not involve this extra $\log n$ factor. But for unbounded sets, we do not know how to avoid the use VC-subgraph classes and it is their introduction that leads to this $\log n$ factor.

\subsection{Simple linear regression\label{E6}}
As an illustration of the superiority of $\rho$-estimators over the least squares method in some 
regression frameworks, we consider the very simple situation of observations $Y_i=a+bx_i+
\varepsilon_i$, $1\le i\le n$, that is a simple linear regression, where the errors $\varepsilon_i$ 
are i.i.d.\ with a known unimodal density $p$ and satisfy $\mathbb{E}[\varepsilon_i]=0$ and 
$\Var(\varepsilon_i)=1$. We moreover assume that $n=2r-1$ with $r\ge 2$ is odd and 
$x_i=n^{-1}[2i-n-1]$ ($x_{1}=-1+1/n,\ldots,x_{r-1}=-2/n,x_{r}=0,x_{r+1}=2/n,\ldots,x_{n}=1-1/n$) 
so that $x_{i}\in(-1,1)$ for all $i=1,\ldots,n$, 
\[
\sum_{i=1}^nx_i=0\qquad\mbox{and}\qquad\sum_{i=1}^nx_i^2=\frac{n^2-1}{3n}.
\]
This corresponds to an affine regression fonction and to the model $F=\{\gf(a,b)\in\Bbb{R}^n,
\,(a,b)\in\R^2\}$ with $f_i(a,b)=a+bx_i$ for $1\le i\le n$. It is well known that, in this case, the 
least squares estimator $(\widetilde{a},\widetilde{b})$ of the parameter $(a,b)$ satisfies 
\[
\mathbb{E}\left[(\widetilde{a}-a)^2\right]={1\over n}\qquad\mbox{and}\qquad
\mathbb{E}\left[(\widetilde{b}-b)^2\right]=\frac{3n}{n^2-1}>{3\over n}.
\]
Let us now assume that the density $p$ satisfies
\begin{equation}
h^2\left(p(\cdot-\theta),p(\cdot-\theta')\strut\right)\ge c \left[|\theta-\theta'|^\gamma\wedge1\right]
\label{Eq-Hel-Euc}
\end{equation}
for some $\gamma\in (0,2)$ and $c>0$. The joint density of the observations $Y_i$ is 
$\prod_{i=1}^np(y_i-a-bx_i)=\prod_{i=1}^np_{a+bx_i}(y_i)$ and can be estimated by a $\rho$-estimator 
based on the model $F$, resulting in the estimated density $\prod_{i=1}^n
p_{\widehat{a}+\widehat{b}x_i}\left(y_i\right)$ and we know that, with large probability,
\[
\sum_{i=1}^nh^2\left(p_{\widehat{a}+\widehat{b}x_i},p_{a+bx_i}\right)\le B\log n,
\]
for some constant $B$ independent of $n$. Then (\ref{Eq-Hel-Euc}) implies, with 
$\alpha=\widehat{a}-a$, $\beta=\widehat{b}-b$, that
\[
\sum_{i=1}^n\left[\left|\widehat{a}-a+\left(\widehat{b}-b\right)x_i\right|^\gamma\wedge 1\right]=
\sum_{i=1}^n\left[|\alpha+\beta x_i|^\gamma\wedge1\right]\le B\log n/c.
\]
Note that if $\alpha\beta\ge 0$ then $|\alpha+\beta x_i|=|\alpha|+|\beta|x_{i}$ for all $i\ge r$ and if 
$\alpha\beta<0$ then $|\alpha+\beta x_{i}|=|\alpha|+|\beta|x_{2r-i}$ for all $i\le r$ so that 
\begin{eqnarray*}
B\log n/c&\ge& \sum_{i=1}^n\left[|\alpha+\beta x_i|^\gamma\wedge1\right]\;\;\ge\;\; 
\sum_{i=r}^{2r-1}\cro{\pa{|\alpha|+|\beta| x_i}^\gamma\wedge1}\\&\ge& 
\sum_{i=r}^{2r-1}\max\{|\alpha|^{\gamma}\wedge1,(|\beta|^{\gamma}\wedge 1) x_i^\gamma\}
\;\;\ge\;\;r\max\{\pa{|\alpha|^{\gamma}\wedge 1},(|\beta|^{\gamma}\wedge 1)I\},
\end{eqnarray*}
with 
\[
I={1\over r}\sum_{i=r+1}^{2r-1}x_{i}^{\gamma}\ge {n\over 2r}\int_{0}^{1-1/n}u^{\gamma}du= 
{(1-1/n)^{\gamma+1}\over (\gamma+1)(1+1/n)}\ge {(2/3)^{\gamma}\over 2(\gamma+1)}.
\]
It follows that 
\[
\max\{\pa{|\alpha|^{\gamma}\wedge 1},(|\beta|^{\gamma}\wedge1)\}\le 
C(\gamma)n^{-1}\log n,
\]
hence $\max\{|\alpha|,|\beta|\}<1$ for $n$ large enough and finally
\[
|\alpha|^2+|\beta|^2=\left|\widehat{a}-a\right|^2+\left|\widehat{b}-b\right|^2=
O_P\left((\log n/n)^{2/\gamma}\right),
\]
which improves on the estimation by least-squares, at least when $n$ is large, since $\gamma<2$.

\subsection{Homoscedastic regression with random design\label{E4}}
In this section, we consider the regression framework with random design described in Section~\ref{A4}. 
Least squares or penalized least squares are the classical estimators which are used in this context 
and many efforts have been made to analyze their performances  under suitable conditions on the 
moments of the errors and the distribution of the design (see Baraud~\citeyearpar{MR1918295}, 
Audibert and Catoni~\citeyearpar{MR2906886} and the references therein). Our point of view is 
different. We shall rather assume that the distribution of the design is completely unknown while the 
distribution of the errors is approximately known and symmetric, but possibly without moments. 
Furthermore, while the $\IL_{2}$-norm with respect to the law of the design is the usual loss function 
that is used for analyzing the performance of the least squares, we shall rather stick to Hellinger-type 
losses. More precisely, we evaluate the performance of an estimator $\widehat{f}$ of $f$ by the risk
\[
\E\left[\gh^2(\mathbf{p_{f}},\mathbf{p_{\widehat f}})\right]=
n\E\left[\int_{\mathscr{W}}h^{2}\left(p_{\widehat{f}}(w,\cdot),p_{f}(w,\cdot)\right)d\nu(w)\right],
\]
with $p_g(w,\cdot)=p(\cdot-g(w))$ for all $g\in\FF$. This actually corresponds to the use of the loss function $\ell(g,g')$ on $\FF$ with
\begin{equation}
\ell(g,g')=h^2(p_g,p_{g'})=
\int_{\mathscr{W}}h^{2}\left(p_g(w,\cdot),p_{g'}(w,\cdot)\right)d\nu(w).
\label{Eq-ell}
\end{equation}
When the density $p$ is of order $\alpha\in(-1,1]$, as given by Definition~\ref{D-order},  
one can relate $\ell$ to some power of a more classical $\IL_{1+\alpha}$-loss since then, according to (\ref{regul}),
\[
a_p\int_{\mathscr{W}}\cro{|g-g'|^{1+\alpha}\wedge A_p^{-1}}d\nu\le\ell(g,g')\le A_p
\int_{\mathscr{W}}\cro{|g-g'|^{1+\alpha}\wedge A_p^{-1}}d\nu\quad\mbox{for all }g,g'\in\FF.
\]
If, moreover, the $\IL_\infty$-norms of the elements of $\FF$ are uniformly bounded by some number $b>0$, then
\[
{1\over[A_p(2b)^{1+\alpha}]\vee1}|g-g'|^{1+\alpha}\le |g-g'|^{1+\alpha}\wedge A_p^{-1}\le 
|g-g'|^{1+\alpha}\quad \mbox{for all }g,g'\in\FF
\]
and $\ell(g,g')$ becomes of the same order as $\norm{g-g'}_{1+\alpha,\nu}^{1+\alpha}=\int_{\mathscr{W}}|g-g'|^{1+\alpha}d\nu$ since 
\[
{a_p\over[A_p(2b)^{1+\alpha}]\vee1}\norm{g-g'}_{1+\alpha,\nu}^{1+\alpha}\le\ell(g,g') 
\le A_p\norm{g-g'}_{1+\alpha,\nu}^{1+\alpha}.
\]
In particular, we recover the usual $\IL_{2}$-loss when $p$ is of order 1 which is the case for the Gaussian, Cauchy and Laplace distributions among others.

To estimate $f$ we proceed as follows: we choose a candidate density $q$ for $p$ which we assume to be symmetric and unimodal and consider a model $F\subset\FF$, which is 
VC-subgraph with index $\overline{V}(F)$ to approximate $f$. To $F$, we associate the model of densities (with respect to $\nu\otimes\mu$) given by 
\[
\overline S_F=\left\{{\bf q_{g}}=(q_g,\ldots,q_g)\,|\,g\in F\right\}
\quad\mbox{where}\quad q_g(w,y)=q(y-g(w))
\]
and estimate the density $s$ of $(W,Y)$ from the observation of $(W_{1},Y_{1}),\ldots,
(W_{n},Y_{n})$ building the corresponding $\rho$-estimator from a countable and dense subset $S$ of $\overline S_F$. We can apply this procedure without knowing $\nu$ since, under the assumptions that $q$ is symmetric and $\mu$ is the Lebesgue measure, for all $g,g'$ and $w\in\mathscr{W}$,
\[
\int_{\R}\sqrt{q_{g}(w,y)r(w,y)}\,d\mu(y)=\int_{\R}\sqrt{q_{g'}(w,y)r(w,y)}\,d\mu(y)
\quad\mbox{with}\quad r={q_{g}+q_{g'}\over 2},
\]
so that by integration with respect to $\nu$, $\rho(q_{g},r)=\rho(q_{g'},r)$.
Therefore $\gT((\gW,\gY),p_{g},p_{g'})$ simply becomes
\[
\gT\left(\strut(\gW,\gY),p_{g},p_{g'}\right)={1\over \sqrt{2}}\sum_{i=1}^{n}\psi\pa{\sqrt{{p_{g'}(W_{i},Y_{i})\over p_{g}(W_{i},Y_{i})}}}.
\]
As in the proof of Theorem~\ref{thmVC}, it follows from Proposition~\ref{perm} that, under Assumption~\ref{hypo-homo1}, $\{q_g, g\in F\}$ is VC-subgraph with index not larger than $C'\overline V(F)$. Applying Theorem~\ref{main1} then leads to the following result.
%
\begin{thm}\label{T-RD}
If $q$ is unimodal and symmetric and $F$ is a model for $f$ which is VC-subgraph of index 
$\overline{V}(F)$ there exists a $\rho$-estimator $\widehat s=q_{\widehat f}$ of $s=p_{f}$ such 
that for all $\xi>0$, with probability at least $1-e^{-\xi}$,
\begin{eqnarray*}
Ch^{2}(p_{f},q_{\widehat f})&\le& \inf_{g\in F}h^{2}\pa{p_{f},q_{g}}+{\overline V(F)\over n}
\left[1+\log_+\left(\frac{n}{\overline V(F)}\right)\right]+{\xi\over n}\\&\le& 2h^{2}(p,q)+
2\inf_{g\in F}h^{2}\pa{p_{f},p_{g}}+
{\overline V(F)\over n}\left[1+\log_+\left(\frac{n}{\overline V(F)}\right)\right]+{\xi\over n}.
\end{eqnarray*}
If, in particular, $p=q$, then for all $\xi>0$,
\[
\P_{\gs}\left[C\ell(f,\widehat f)\le\inf_{g\in F}\ell(f,g)+{\overline V(F)\over n}
\left[1+\log_+\left(\frac{n}{\overline V(F)}\right)\right]+{\xi\over n}\right]\ge1-e^{-\xi}.
\]
If, moreover, (\ref{regul}) holds and $\max\left\{\sup_{g\in F}\|g\|_\infty,\|f\|_\infty\right\}\le b<+\infty$, then
\[
\P_{\gs}\left[C'\norm{f-\widehat f}_{1+\alpha,\nu}^{1+\alpha}\le
\inf_{g\in F}\norm{f-g}_{1+\alpha,\nu}^{1+\alpha}+{\overline V(F)\over n}
\left[1+\log_+\left(\frac{n}{\overline V(F)}\right)\right]+{\xi\over n}\right]\ge1-e^{-\xi},
\]
for all $\xi>0$ and some constant $C'$ depending only on $A_p,a_p,b$ and $\alpha$.
\end{thm}
%
\subsection{Further examples for regression problems\label{E3}}
Let us recall that we want to estimate the unknown function $f$ on $\mathscr{W}$ in both the random 
design framework where $X=(W,Y)$ and $Y=f(W)+\varepsilon$ and the fixed design framework which corresponds, 
with analogous notations, to $X=f(w)+\varepsilon$ with $w\in\mathscr{W}=\{1,\ldots,n\}$.

As we noticed in the previous sections, when dealing with both regression frameworks, when the 
model $F$ for the regression function $f$ is VC-subgraph, the performance of the estimator $\widehat{f}$ depends on the VC-index $\overline{V}(F)$. A common practice to design regression models is to 
choose for $F$ a $D$-dimensional linear space of functions which, according to Section~\ref{S}, is 
VC-subgraph with index bounded by $D+2$. This includes the celebrated ``linear model" in the fixed 
design framework when $F$ is the linear span of $D$ linearly independent vectors $\gg^1,\ldots,
\gg^D$ in $\R^n$ or, equivalently, of $D$ functions $g^1,\ldots,g^D$ on $\{1,\ldots,n\}$.

Let us, for a moment, focus on this situation of $F$ being a $D$-dimensional linear space. Classical 
least squares estimators in the fixed design case lead to risk bounds of order $D/n$ when the errors 
are Gaussian or, more generally, have a few moments, but fail miserably when they are Cauchy 
while, as we have seen in Section~\ref{E1}, $\rho$-estimators provide the same rate of convergence, 
apart from an extra $\log(n/D)$ factor, with the loss function
\[
d_{2}(\gg,\gg')=\frac{1}{n}\sum_{i=1}^{n}\pa{\ab{g_{i}-g'_{i}}^{2}\wedge A_q^{-1}}.
\]
When the errors have a uniform distribution, we derive a bound of order $(D/n)\log(n/D)$ for the loss
\[
d_{1}(\gg,\gg')=\frac{1}{n}\sum_{i=1}^{n}\pa{\ab{g_{i}-g'_{i}}\wedge A_q^{-1}}
\]
while errors with unbounded densities of the form $q(x)=[2(1-\beta)]^{-1}(1-|x|)^{-\beta}\1_{[-1,1](x)}$ with 
$0<\beta<1$ lead to the same bound with the loss
\[
d_{1-\beta}(\gg,\gg')=\frac{1}{n}\sum_{i=1}^{n}\pa{\ab{g_{i}-g'_{i}}^{1-\beta}\wedge A_q^{-1}}.
\]
Since subsets of VC-subgraph classes are also VC-subgraph one can restrict $F$ to be a bounded subset of such a linear space and get similar results for the random design situation, according to Theorem~\ref{T-RD}, with loss functions of the form 
$\|f-\widehat f\|_{2,\nu}^{2}$, $\|f-\widehat f\|_{1,\nu}$ and 
$\|f-\widehat f\|_{1-\beta,\nu}^{1-\beta}$ respectively.

An alternative way of building models that still satisfy the assumptions which are needed to apply our results is as follows. We start from a $D$-dimensional linear space $G$ of functions on $\mathscr{W}$ and consider some monotone function $\Psi$. Finally we take for $F$ the set $\{\Psi\circ g, g\in G\}$. It follows from
(ii) of Proposition~\ref{perm} that $F$ is still VC-subgraph with index not larger than $D+2$ and the previous results still holds. We may replace ``monotone" by ``unimodal" and get similar results according to (vi) of the same proposition. This allows, given $D$ independent functions $g^1,\ldots,g^D$, to use for instance models $F$ of the following forms:
\[
\left\{\exp\left[\sum_{j=1}^D\beta_jg^j\right]\!,\,\beta_j\in\R\mbox{ for }1\le j\le D\right\}\quad
\mbox{or}\quad\left\{\left|\sum_{j=1}^D\beta_jg^j\right|,\,\beta_j\in\R\mbox{ for }1\le j\le D\right\},
\]
among many other possibilities.

\subsection{A parametric bound over a set with infinite metric dimension\label{E2}}
In this section, we want to show that, unlike T-estimators, the construction and performance of which heavily depend on the metric dimension of the model that is used, our estimator can, in some cases, achieve a parametric rate which is not connected to its metric dimension. The following illustration given for the density framework described in Section~\ref{A2} is borrowed from Birg\'e~\citeyearpar{MR722129} (Section~6).

Let $\Lambda$ be any nonvoid subset of $\N$ and $\Theta=\Lambda^{\N\setminus\{0\}}$,
$\X=\bigcup_{j\ge 1}\Lambda^{j}$ be respectively the sets of infinite and finite sequences
with entries in $\Lambda$. Note that the set $\X$ is countable and that we may introduce on
$\X$ the family of probabilities $\{P_{\theta},\,\theta\in\Theta\}$ given by  
\[
P_{\theta}=\sum_{j\ge 1}2^{-j}\delta_{(\theta_{1},\,\ldots\, ,\theta_{j})}\quad\mbox{for all }
\theta=(\theta_{1},\ldots,\theta_k,\ldots)\in \Theta.
\]
In the sequel, we denote by $s_{\theta}$ the density of $P_{\theta}$ with respect to the counting measure 
on $\X$, that is $s_{\theta}(x)=P_{\theta}(\{x\})$ for all $x\in \X$, and set $\overline S=\{s_{\theta},\,
\theta\in\Theta\}$, which is identifiable. Our aim is to estimate $s_\theta$ from the observation of a sample $X_1,\ldots,X_n$.

It will be convenient to define the following operators : $\ell(x)$ is the length of an element $x\in\X$, 
that is $\ell(x)=j$ if $x\in\Lambda^j$, $\pi_j$ is the operator from $\Theta$ to $\Lambda^j$ such that 
$\pi_j(\theta)=(\theta_{1},\ldots,\theta_{j})$ and $\pi_{-1}$ is an operator from $\X$ to $\Theta$ such 
that if $x\in\Lambda^j$, $\pi_j\circ \pi_{-1}(x)=x$ or, equivalently, $\pi_{\ell(x)}\circ \pi_{-1}(x)=x$ 
for all $x\in\X$. It follows that $\pi_{-1}(\X)$ is a countable subset of $\Theta$ and 
$S=\left\{s_\theta,\,\theta\in\pi_{-1}(\X)\right\}$ is a countable subset of $\overline{S}$.

Let $J$ be the mapping from $\Theta^{2}$ to $\N\cup\{+\infty\}$ defined by
\[
J(\theta,\theta')=\sup\{j\in\N\,|\,\theta_k=\theta'_k\quad\mbox{for }1\le k\le j\}\quad\mbox{with }
\sup\N=+\infty,\;\;\sup\varnothing=0.
\]
Since $s_{\theta}(x)=2^{-\ell(x)}$ if $\pi_{\ell(x)}(\theta)=x$ and $s_{\theta}(x)=0$ otherwise, the Hellinger distance between two densities $s_{\theta}$ and $s_{\theta'}$ with $\theta,\theta'\in \Theta$ is given by 
\begin{equation}\label{f-h}
h^{2}(s_{\theta},s_{\theta'})=1-\rho(s_{\theta},s_{\theta'})=1-\sum_{j=1}^{J(\theta,\theta')}2^{-j}=2^{-J(\theta,\theta')}
\quad\mbox{with}\quad\sum_{\varnothing}=0;\quad2^{-\infty}=0.
\end{equation}
For any $x\in\X$ with $\ell(x)=j$ one can find a subset $\Theta_j$ of $\Theta$ with $|\Theta_j|=|\Lambda|$
and such that $\pi_j(\theta)=x$ for all $\theta\in\Theta_j$ but all $\pi_{j+1}(\theta)$ are different. It suffices 
for that to let $\theta_{j+1}$ go across all elements of $\Lambda$ to build the elements of $\Theta_j$. 
As a consequence, $h^2(\theta,\theta')=2^{-j}$ for all $\theta\ne\theta'\in\Theta_j$ and $\Theta_j$ is 
included in a closed ball of radius $2^{-j/2}$. This shows that the metric dimension of $\overline{S}$ 
can be made arbitrarily large or even infinite by playing with the cardinality of $\Lambda$. 

Though $\overline{S}$ can be massive, the parameter $\theta$ is not difficult to estimate.
Let us first observe that, if $\ell(X_i)=j$, then $X_i=\pi_j(\theta)$ $P_\theta$-a.s. Therefore, if $k=\ell(X_{i_{0}})=\sup_{1\le i\le n}\ell(X_i)$ and $\widehat\Theta=\pi_k^{-1}(X_{i_{0}})$, for all $\widehat \theta\in \widehat \Theta$,
\begin{equation}\label{ex01}
s_{\widehat \theta}(X_{i})=2^{-\ell(X_{i})}=s_{\theta}(X_{i}),\quad\mbox{for }i\in\{1,\ldots,n\},
\quad P_\theta\mbox{-a.s.}
\end{equation}
while for all $\theta'\not\in\widehat \Theta$ $s_{\theta'}(X_{i})\le s_\theta(X_{i})$ for 
$1\le i\le n$ and $0=s_{\theta'}(X_{i_0})<s_\theta(X_{i_0})$. It follows that the likelihood reaches its maximum over the elements $\widehat \theta\in\widehat \Theta$ $P_\theta$-a.s.
It is proven in Birg\'e~\citeyearpar{MR722129} that these maximum likelihood estimators satisfy
\begin{equation}\label{riex}
\mathbb{E}_{\theta}\cro{{h^{2}(s_{\theta},s_{\widehat \theta})}}\le Cn^{-1},\quad\mbox{for all }\theta\in \Theta
\end{equation}
and some numerical constant $C>0$. In particular, the minimax risk over $\overline S$ converges to 
zero with parametric rate.

Let us now consider the set of $\rho$-estimators of $s_\theta$ build from the model $S$, that is the set
${\rm Cl}\!\left(\EE(\gX,S)\strut\right)$. The following result then holds.
%
\begin{prop}\label{P-cexem}
For all $\theta\in \Theta$ and any choice of a maximum likelihood estimator $\widehat \theta$ in 
$\widehat \Theta$, $s_{\widehat \theta}$ is a $\rho$-estimator of $s_{\theta}$, $P_{\theta}$-a.s.
\end{prop}
\begin{proof}
Let us observe that if $\theta\in\Theta$, $x=\pi_j(\theta)$ and $\theta'=\pi_{-1}(x)$, it follows from 
(\ref{f-h}) that $h^2(s_\theta,s_{\theta'})\le2^{-j}$ so that $S$ is a dense subset of $\overline{S}$. To show 
that $\{s_{\widehat\theta},\,\widehat \theta\in\widehat\Theta\}\subset{\rm Cl}\!\left(\EE(\gX,S)\strut\right)$ 
it is therefore enough to prove that the elements of $\widehat S=
\{s_{\widehat\theta},\,\widehat\theta\in\widehat \Theta\}\cap S$ minimize $\gup(S,\cdot)$ over $S$.
First note that for $\theta,\theta'\in \Theta$ with $\theta\neq \theta'$, 
\begin{eqnarray*}
\rho\pa{s_{\theta},{s_{\theta}+s_{\theta'}\over 2}}&=&\sum_{x\in\X}\sqrt{s_{\theta}(x){s_{\theta}(x)+
s_{\theta'}(x)\over2}}\;\;=\;\;\sum_{j=1}^{J(\theta,\theta')}2^{-j}+\sum_{j>J(\theta,\theta')}{2^{-j}\over\sqrt{2}}
\\&=&1-\pa{1-{1\over\sqrt{2}}}2^{-J(\theta,\theta')}\;\;=\;\;\rho\pa{s_{\theta'},{s_{\theta}+s_{\theta'}\over2}}.
\end{eqnarray*}
Let us now fix some element $s_{\widehat \theta}\in \widehat S$. Because of \eref{ex01},  
$\gT(\gX,\gs_{\widehat \theta},\gs_{\widehat \theta'})=0$ for $s_{\widehat \theta'}\in \widehat S$ 
and, for $s_{\theta'}\in S\setminus \widehat S$,
\[
\gT(\gX,\gs_{\widehat \theta},\gs_{\theta'})= \sum_{i=1}^{n}{\sqrt{s_{\theta'}(X_{i})}-
\sqrt{s_{\widehat \theta}(X_{i})}\over \sqrt{2\left(s_{\theta'}(X_{i})+s_{\widehat \theta}(X_{i})\right)}}\le{\sqrt{s_{\theta'}(X_{i_{0}})}-\sqrt{s_{\widehat \theta}(X_{i_{0}})}\over\sqrt{2\left(s_{\theta'}(X_{i_{0}})+s_{\widehat \theta}(X_{i_{0}})\right)}}=-\frac{1}{\sqrt{2}}.
\]
Consequently $\gup(S,\gs_{\widehat \theta})=0$ and therefore $\gs_{\widehat \theta}$ minimizes 
$\gup(S,\cdot)$ over $S$.
\end{proof}
We shall prove in Section~\ref{P4}, the following result.
\begin{prop}\label{P-dimpara}
For all $s\in \overline S$ and $\overline s\in S$, 
\[
D^{S}(\gs,\overline \gs)\le4c_0^{-2}\cro{2\sqrt{2}+n^{2}h^{2}(s,\overline s)}^{2}.
\]
\end{prop}
Applying~\eref{risk} we get, since $\gh^{2}(\cdot,\cdot)=nh^{2}(\cdot,\cdot)$, 
\[
\E\cro{h^2(s,\widehat s)}\le c_1h^2(s,\overline s)+\frac{4c_2c_0^{-2}\pa{2\sqrt{2}+
n^2h^2(s,\overline s)}^2+2.45c_3}{n}\quad\mbox{for all }s\in \overline S\mbox{ and }\overline s\in S.
\]
Choosing $\overline s$ arbitrarily close to $s\in \overline S$ shows that any $\rho$-estimator 
$\widehat s$ (and therefore any  maximum likelihood estimator) satisfies
\[
\E\cro{h^{2}(s,\widehat s)}\le\left(32c_{2}c_0^{-2}+2.45c_{3}\right)n^{-1}\quad\mbox{for all } s\in\overline S
\]
and thus achieves a parametric rate of convergence independently of the metric dimension of $\overline{S}$.

\section{Model selection\label{MS}}
Let us now assume that, in place of a single model as in the previous sections, we have at disposal 
a countable collection $\overline{\SS}$ of such models $\overline{S}$ for the parameter $s$. We may 
therefore associate to each $\overline{S}\in\overline{\SS}$ a $\rho$-estimator $\widehat{s}(\overline{S})$ 
with quadratic risk $\E\cro{\gh^2\!\pa{\gs,\widehat\gs(\overline{S})}}$ and our aim is to select from the 
data $\gX$ a model $\widehat S\in\overline{\SS}$ or, equivalently, an estimator $\widehat{s}(\widehat{S})$ among the family of candidates $\{\widehat s(\overline{S}),\,\overline{S}\in\overline{\SS}\}$, in such a 
way that its risk is as close as possible to the minimal risk over the family, namely 
$\inf_{\overline{S}\in\overline{\SS}}\E\cro{\gh^2\!\pa{\gs,\widehat\gs(\overline{S})}}$. 

\subsection{Estimation procedure and main result\label{MS1}}
Let $\overline{\SS}$ be a countable family of models in $\LL_{0}$, endowed with a mapping 
$\Delta$ from $\overline{\SS}$ into $\R_{+}$ satisfying
\begin{equation}
\sum_{\overline{S}\in\overline{\SS}}\exp\left[-\Delta\left(\overline{S}\right)\right]\le1.
\label{Eq-subprob}
\end{equation}
To each $\overline{S}\in\overline{\SS}$ we attach some identifiable subset $S$ of $\LL_0$ which is
either a countable $\eta$-net for $\overline{S}$ or a dense subset of $\overline{S}$, as we did for a 
single model in Section~\ref{C3}, the connection between $\overline{S}$ and $S$ being only 
emphasized by the notations. This results in a new collection $\SS$ of subsets $S$ of 
$\LL_0$, each $S\in\SS$ corresponding to a model $\overline{S}\in\overline{\SS}$ and 
the set $\S=\bigcup_{\overline S\in\overline\SS}S=\bigcup_{S\in\SS}S$ is a 
countable subset of $\LL_0$. Let $\pen$ be some positive function on $\S$,  
\begin{equation}
\overline\gup(\S,\gt)=\sup_{\gt'\in\S}\ac{\gT(\gX,\gt,\gt')-\pen(\gt')}+\pen(\gt)
\quad\mbox{for all }\gt\in\S
\label{Eq-gup}
\end{equation}
and
\[
\EE(\gX,\S)=\left\{\widetilde\gs\in\S\,\left|\,\overline\gup(\S,\widetilde\gs)\le \inf_{\gt\in \S}\overline\gup(\S,\gt)
+\frac{\kappa}{10}\right.\right\}\quad\mbox{with $\kappa$ given by (\ref{Eq-cons1})}.
\]
As in Section~\ref{C1}, we define our estimator $\widehat \gs$ of $\gs$ as any element of
${\rm Cl}\!\left(\EE(\gX,\S)\strut\right)$. When $\overline \SS$ reduces to a single element 
$\overline S$ and the function    $\pen$ is constant on $S$, the estimator $\widehat\gs$ 
coincides with the one we defined in Section~\ref{C1}.

 It follows from (\ref{Eq-gup}) that the estimator only depends
on the differences $\pen(\gt')-\pen(\gt)$ rather than on the function $\pen$ itself. This means
that, in our computations, we may always replace the actual penalty function $\pen$ that
has been used to build the $\rho$-estimator by another 
one, $\pen'$, with $\pen'(\gt)=\pen(\gt)+G$ for some $G$ independent of $\gt\in\S$.
%
\begin{thm}\label{main0}
If the penalty function $\pen$ satisfies
\begin{equation}\label{def-pen6}
\pen(\gt)\ge\pen_1(\gt)=\inf_{\{S\in\SS\,|\,S\ni\gt\}}\ac{(1/8)\overline{D}^{S}+\kappa\Delta(\overline{S})}
\quad\mbox{for all }\gt\in \S,
\end{equation}
where $\overline D^{S}$ is defined by~\eref{Eq-DS}, any element $\widehat \gs$ in 
${\rm Cl}\!\left(\EE(\gX,\S)\strut\right)$ satisfies, for all $\gs\in\LL_{0}$ and $\xi>0$,
\begin{equation}
\P_{\gs}\left[\gh^{2}\pa{\gs,\widehat \gs}\le\inf_{\overline \gs\in\S}\left\{c_1\gh^{2}(\gs,\overline\gs)+8c_{2}\pen(\overline \gs)\right\}+c_3(1.45+\xi)\right]\ge1-e^{-\xi}.
\label{Eq-bou5}
\end{equation}
In particular, if $\pen(\gt)=\pen_1(\gt)$ for all $\gt\in \S$ and the sets $S\in\SS$ are chosen in order to satisfy
\[
2c_{1}\sup_{\gu\in\overline S}\gh^2(\gu,S)+c_{2}\overline D^{S}\le \overline D(\overline S)+{c_{3}\over 20},
\]
then for all $\gs\in\LL_{0}$ and $\xi>0$,
\begin{equation}\label{enInt}
\P_{\gs}\cro{\gh^{2}\pa{\gs,\widehat \gs}\le \inf_{\overline{S}\in\overline{\SS}}\ac{2c_1\gh^{2}(\gs,\overline{S})+\overline D(\overline S)+c_3\left(\Delta(\overline{S})+1.5+\xi\right)}}\ge 1-e^{-\xi}.
\end{equation}
\end{thm}
When the models $\overline S$ are VC-subgraph with respective indices 
$\overline V(\overline S)$ we have seen in Theorem~\ref{main1} that 
\[
\overline D^{S}\le C'\overline V(\overline S)\left[1+\log_+\left(n/\overline V(\overline S)\right)\right]
\]
for all choices of a countable and dense subset $S$ of $\overline S$. For such choices of $S$ and a penalty equal to $\pen_1$, we derive from 
\eref{Eq-bou5} that the estimator $\widehat \gs$ satisfies, for all $\xi>0$,
\begin{equation}\label{Risk-VC}
\P_{\gs}\cro{C\gh^{2}(\gs,\widehat\gs)\le \inf_{\overline{S}\in\overline{\SS}}
\ac{\rule{0mm}{4mm}\gh^{2}(\gs,\overline{S})+\overline V(\overline S)
\left[1+\log_+\left(n/\overline V(\overline S)\right)\right]+\Delta(\overline S)+\xi}}\ge 1-e^{-\xi}.
\end{equation}
As we have seen in Proposition~\ref{histo},  we are not always able to bound the uniform dimension 
$\overline D(\overline{S})$ of a model $\overline{S}$ from above but sometimes only its dimension $D(\overline{S})$. In this case, model selection is still possible under the following alternative assumption. 
\begin{ass}\label{cond-union}
Let $\breve{D}$ be a mapping from $\overline{\SS}$ into $[1,+\infty)$ such that for all 
$\overline{S},\overline{S}'\in\overline{\SS}$, 
\[
D^{S\cup S'}\le \breve{D}(\overline{S})+\breve{D}(\overline{S}').
\]
\end{ass}
\begin{thm}\label{main-histo}
Let Assumption~\ref{cond-union} hold and the penalty function $\pen$ satisfy
\begin{equation}\label{def-pen0bis}
\pen(\gt)\ge\inf_{\{\overline{S}\in\overline{\SS}\,|\,S\ni\gt\}}\ac{(1/8)\breve{D}(\overline{S})+\kappa\Delta(\overline{S})}\quad\mbox{for all }\gt\in \S.
\end{equation}
The estimator $\widehat\gs$ then satisfies for all $\gs\in\LL_0$ and $\xi>0$,
\begin{equation}\label{eq-histo1}
\P_{\gs}\left[\gh^{2}\pa{\gs,\widehat \gs}\le \inf_{\overline \gs\in \S}\ac{c_{1}\gh^{2}\pa{\gs,\overline \gs}+16c_{2}\pen(\overline \gs)}+c_{3}(1.45+\xi)\right]\ge1-e^{-\xi}.
\end{equation}
If, moreover, the sets $S\in\SS$ are chosen to satisfy,
\begin{equation}\label{Eq-fine}
2c_{1}\sup_{\gu\in\overline S}\gh^2(\gu,S)+c_{2}D^{S}\le D(\overline S)+{c_{3}\over 80a}\;\;\mbox{ for all }\overline{S}\in\overline{\SS}\mbox{ and some }a\ge 1/2,
\end{equation}
Assumption~\ref{cond-union} holds with $\breve{D}({\overline{S})\le aD^S}$ and equality holds in~\eref{def-pen0bis}, then for all $\gs\in\LL_0$ and $\xi>0$,
\[
\P_{\gs}\cro{\gh^{2}\pa{\gs,\widehat \gs}\le \inf_{\overline S\in \overline\SS}\ac{2c_{1}
\gh^{2}\pa{\gs,\overline S}+2aD(\overline S)+c_{3}\left(2\Delta(\overline S)+1.5+\xi\right)}}\ge 1-e^{-\xi}.
\]
\end{thm}
We shall now turn to examples in the next sections. Throughout these sections, we shall assume that 
$\widehat \gs$ is built with a choice of the penalty function equal to $\pen_1$ as defined in \eref{def-pen6}. 
Finally, given a countable set $T$, we shall say that $\pi$ is a positive sub-probability on $T$, if  
$\pi(t)>0$ for all $t$ in $T$ and $\sum_{t\in T}\pi(t)\le 1$. Given such a $\pi$, we shall set $\Delta_{\pi}(t)=
-\log\left(\strut\pi(t)\right)$. It follows that choosing a function $\Delta$ which satisfies (\ref{Eq-subprob}) 
amounts to finding a subprobabilty $\pi$ on $\overline{\SS}$ and setting $\Delta=\Delta_\pi$ .

\subsection{Homoscedastic regression with unknown scaling\label{MS2}}
We consider here the regression setting described in Section~\ref{A3} where $\mu$ is the Lebesgue measure and both $p$ and $\gf$ are unknown. Throughout this section, we shall consider a family $\F$ of subsets $F\subset \R^{n}$ to approximate $\gf$ and a family $\QQ$ of densities $q$ together with a scaling parameter $\lambda>0$ to approximate 
$p$ by densities of the form $q_{0,\lambda}$. We recall from (\ref{Eq-trsc1}) that $q_{0,\lambda}(x)=\lambda^{-1}q(x/\lambda)$ for $\lambda>0$ and $x\in\R$ and make the following assumptions.
\begin{ass}\label{hypo-homo1}
The family $\F$ is a countable family of VC-subgraph classes $F$ with respective VC-indices $\overline V(F)$ and $\F$ is endowed with a positive sub-probability $\pi$.
\end{ass}
\begin{ass}\label{hypo-homo2}
The family $\QQ$ is countable and endowed with a positive sub-probability 
$\gamma$. For each $q\in \QQ$, Assumption~\ref{hypoQ} is satisfied and there exists a non-decreasing function $\w_q$ from $[1,2]$ into $\R_{+}$ such that $\w_q(1)=0$ and
\[
h^{2}\pa{q,q_{0,\lambda}}\le \w_{q}(\lambda)\quad\mbox{for all }\lambda\in [1,2].
\]
\end{ass}
Given a density $q$ on $\R$, a vector $\gg\in\R^{n}$ and $\lambda>0$, we define the density $\gq_{\gg,\lambda}$ with respect to the Lebesgue measure on $\R^n$ according to (\ref{Eq-trsc2}), that is
\[
\gq_{\gg,\lambda}(x_1,\ldots,x_n)=\pa{\rule{0mm}{3.5mm}
q_{0,\lambda}(x_1-g_{1}),\ldots,q_{0,\lambda}(x_n-g_{n})}.
\]
We consider the family of models $\overline{\SS}$ defined as follows. For $i\in\N$, $j\in\Z$, 
$k\in\{0,\ldots,2^{i}-1\}$, $F\in\F$ and $q\in\QQ$, let
\[
\overline{S}^{i,j,k}_{q,F}=\ac{\gq_{\gg,\lambda}\,\left|\,\gg\in F, \lambda=
\lambda_{i,j,k}\right.}\quad\mbox{with}\quad\lambda_{i,j,k}=2^{j}(1+k2^{-i})
\]
and define the family $\overline{\SS}$ as 
\begin{equation}\label{def-fam}
\overline{\SS}=\ac{\left.\overline{S}_{q,F}^{i,j,k}\,\right|\,q\in\QQ, F\in\F, (i,j,k)\in\N\times \Z\times \{0,\ldots,2^{i}-1\}}.
\end{equation}
We endow $\overline{\SS}$ with the weights $\Delta$ given by 
\begin{equation}\label{def-poids}
\Delta\pa{\overline{S}_{q,F}^{i,j,k}}=\Delta_{\gamma}(q)+\Delta_{\pi}(F)+|j|+i+2+i\log 2
\end{equation}
and we check that 
\begin{eqnarray*}
\sum_{\overline{S}\in\overline{\SS}}e^{-\Delta(\overline{S})}&=&\sum_{q\in\QQ}\gamma(q)\sum_{F\in\F}\pi(F)\sum_{j\in\Z}e^{-(|j|+1)}\sum_{i\in\N}e^{-(i+1)}\sum_{k=0}^{2^{i}-1}e^{-i\log 2}<1.
\end{eqnarray*}
Then, the following holds.
\begin{thm}\label{thm-homo}
Let $\F$ be a family of models satisfying Assumption~\ref{hypo-homo1} and $\QQ$ a family of 
densities satisfying Assumption~\ref{hypo-homo2}. For the collection $\overline{\SS}$ and weight 
function $\Delta$ defined by~\eref{def-fam} and~\eref{def-poids} respectively, any $\rho$-estimator 
$\widehat\gs$ satisfies for all densities $p$ and parameters $\gf\in\R^{n}$,
\begin{eqnarray*}
C\E\cro{\gh^{2}(\gs,\widehat \gs)}&\le&\inf_{q\in\QQ,\,\lambda>0}\left[nh^{2}(p,q_{0,\lambda})+
\inf_{i\ge 0}\cro{n\w_{q}\pa{1+2^{-i}}+i}+\Delta_{\gamma}(q)+\ab{\log\lambda}\right]\\
&&\mbox{}+\inf_{F\in\F}\cro{\inf_{\gg\in F}\gh^{2}(\gp_{\gf},\gp_{\gg})+\overline V(F)
\cro{1+\log_{+}\pa{n/\overline V(F)}}+\Delta_{\pi}(F)}.
\end{eqnarray*}
\end{thm}
\begin{proof}
Applying 
Theorem~\ref{main0} to the family $\overline{\SS}$ and using the fact that $\overline D$ is bounded from below by $c_2$, we derive that 
whatever the choices of $F\in\F$, $\gg\in F$, $i\in\N$, $j\in\Z$ and $k\in\{0,\ldots,2^{i}-1\}$, $\E\cro{\gh^{2}(\gs,\widehat\gs)}
=\E\cro{\gh^{2}( \gp_{\gf},\widehat\gs)}\le CA(q,F,\gg,i,j,k)$ where $C$ is a universal constant and
\[
A(q,F,\gg,i,j,k)=\gh^{2}\pa{\gp_{\gf},\gq_{\gg,\lambda_{i,j,k}}}+\overline D\left(\overline{S}^{i,j,k}_{q,F}\right)
+\Delta_{\gamma}(q)+\Delta_{\pi}(F)+|j|+i.
\]
By the triangular inequality, for $\lambda>0$ and $\lambda'=\lambda_{i,j,k}$,
\[
\gh\pa{\gp_{\gf},\gq_{\gg,\lambda'}}\le\gh(\gp_{\gf},\gp_{\gg})+\gh(\gp_{\gg},\gq_{\gg,\lambda})+
\gh(\gq_{\gg,\lambda},\gq_{\gg,\lambda'}).
\]
Since, for densities with respect to the Lebesgue measure, the Hellinger distance is translation 
and scale invariant, 
\[
\gh^2(\gp_{\gg},\gq_{\gg,\lambda})= nh^{2}(p,q_{0,\lambda})\quad\mbox{and}\quad
\gh(\gq_{\gg,\lambda},\gq_{\gg,\lambda'})=\gh(\gq_{\gg,\lambda/\lambda'},\gq_{\gg})
=\sqrt{n}h\left(q_{0,\lambda/\lambda'},q\right).
\]
To bound $\gh\pa{\gp_{\gf},\gq_{\gg,\lambda_{i,j,k}}}$, it remains to bound 
$h\left(q_{0,\lambda/\lambda'},q\right)$ when $\lambda'=\lambda_{i,j,k}$. Let $j\in\Z$ be such that 
$2^j\le\lambda<2^{j+1}$. Then $|j|\le|\log\lambda|/(\log2)+1$ and for all $i\in \N$ one can find 
$k\in\{0,\ldots,2^{i}-1\}$ such that
\[
\lambda_{i,j,k}=2^{j}\left[1+k2^{-i}\right]\le\lambda< \lambda_{i,j,k+1}=2^{j}\left[1+(k+1)2^{-i}\right],
\]
hence $1\le\lambda/\lambda'\le1+2^{-i}\le2$. It then follows from Assumption~\ref{hypo-homo2} that
\[
h^2(q,q_{0,\lambda/\lambda'})\le\w_q(\lambda/\lambda')\le\w_{q}\pa{1+2^{-i}}
\]
since $\w_{q}$ is non-decreasing. Putting all these bounds together for these choices of $i,j,k$ we 
derive that, for some universal constant $C'$ and all $i\in\N$,
\begin{eqnarray*}
C'A(q,F,\gg,i,j,k)&\le&\gh^2(\gp_{\gf},\gp_{\gg})+nh^{2}(p,q_{0,\lambda})+n\w_{q}\pa{1+2^{-i}}\\&&
\mbox{}+\overline D\left(\overline{S}^{i,j,k}_{q,F}\right)+\Delta_\gamma(q)+\Delta_\pi(F)+\log|\lambda|+i.
\end{eqnarray*}
It remains to bound $\overline D\left(\overline{S}^{i,j,k}_{q,F}\right)$. Since $q$ satisfies 
Assumption~\ref{hypoQ}, so does the density $q_{0,\lambda'}$ and it follows from 
Theorem~\ref{thmVC} that, under Assumptions~\ref{hypo-homo1} and~\ref{hypo-homo2}, 
\[
\overline D\left(\overline{S}^{i,j,k}_{q,F}\right)\le C''\overline V(F)\cro{1+\log_{+}\pa{n/\overline V(F)}},
\]
which concludes the proof. 
\end{proof}
Let us now comment on this result. To fix up the ideas, let us take for $q$ and $\lambda$ the values 
that provide the best approximation of $p$ by $q_{0,\lambda}$ among all choices in $\QQ\times(0,+\infty)$, 
even though this choice might not be the optimal one in view of minimizing our risk bound. The quantity 
$nh^{2}(p,q_{0,\lambda})$ therefore corresponds to the usual bias term resulting from the approximation 
of $p$ by the family of densities $q'_{0,\lambda'}$ for $q'\in\QQ$ and $\lambda'>0$. The quantity
\[
\inf_{F\in\F}\cro{\inf_{\gg\in F}\gh^{2}(\gp_{\gf},\gp_{\gg})+\overline V(F)
\cro{1+\log_{+}\pa{n/\overline V(F)}}+\Delta_{\pi}(F)}
\]
is the bound that we would get, if $p$ were known, for estimating $\gf$ by model selection among the 
family $\bigcup_{F\in\F}F$. Finally, the quantity $\Delta_{\gamma}(q)+\inf_{i\ge 0}\cro{n\w_{q}
\pa{1+2^{-i}}+i}+\ab{\log\lambda}$ comes from our estimation of $q$ and $\lambda$ by model selection. 

This regression model includes in particular the case of a known form of the errors corresponding to $p=\overline{p}_{0,\tau}$ where $\overline{p}$ is known and $\tau$ 
unknown, in which case it is natural to take $\QQ=\{\overline{p}\}$. The risk bound then becomes after a proper rescaling:
\begin{eqnarray}
C\E\cro{\gh^{2}(\gs,\widehat \gs)}
&\le&\inf_{F\in\F}\cro{\inf_{\gg\in F}\gh^{2}(\gp_{\gf},\gp_{\gg})+\overline V(F)
\cro{1+\log_{+}\pa{n/\overline V(F)}}+\Delta_{\pi}(F)}\nonumber\\
&&\mbox{}+\inf_{\lambda>0}\left[nh^{2}(\overline{p}_{0,\tau},\overline{p}_{0,\lambda})
+\ab{\log\lambda}\right]+\inf_{i\ge 0}\cro{n\w_{\overline{p}}\pa{1+2^{-i}}+i}\nonumber\\
&=&\inf_{F\in\F}\cro{\inf_{\gg\in F}\gh^{2}(\gp_{\gf},\gp_{\gg})+\overline V(F)
\cro{1+\log_{+}\pa{n/\overline V(F)}}+\Delta_{\pi}(F)}\nonumber\\
&&\mbox{}+\inf_{\sigma>0}\left[nh^{2}(\overline{p},\overline{p}_{0,\sigma})
+\ab{\log\tau\sigma}\right]+\inf_{i\ge 0}\cro{n\w_{\overline{p}}\pa{1+2^{-i}}+i}.
\label{Eq-Gamma9}
\end{eqnarray}
\subsection{An example\label{MS3}}
Let us consider the family of densities $\overline \QQ=\{p^\beta\,|\,\beta\ge0\}$ indexed by the 
parameter $\beta\in[0,+\infty)$ and given by
\begin{equation}\label{def-OQQ}
p^{\beta}(x)=\Lambda(\beta)e^{-|x|^{1/\beta}},\;\;\Lambda(\beta)=\left[\int_{\R}e^{-|x|^{1/\beta}}dx\right]^{-1}
\;\;\mbox{for }\beta>0\quad\mbox{and}\quad p^0={1\over 2}\1_{[-1,1]}.
\end{equation}
It follows from symmetry and a change of variables that
\begin{equation}
\left[\Lambda(\beta)\right]^{-1}=2\int_0^\infty e^{-x^{1/\beta}}dx=2\beta\Gamma(\beta)\quad\mbox{for }\beta>0.
\label{Eq-Gamma1}
\end{equation}
The family $\overline\QQ$ contains the Laplace and Gaussian distributions as well as the uniform 
distribution on $[-1,1]$ which corresponds to the limit of the densities $p^{\beta} $ when $\beta$ 
tends to 0. We recall from Section~\ref{E1} that these densities are of order $2/\beta$ for $\beta>2$ 
and of order 1 for $\beta<2$. In particular, for such densities inequality~\eref{regul} is satisfied with 
$\overline\alpha=(2/\beta)\wedge1$ in place of $\alpha$. We shall consider $\overline\QQ$ as a 
model for our unknown density $p$. In order to apply Theorem~\ref{thm-homo}, which only holds 
for a countable family $\QQ$, we have to discretize $\overline\QQ$. To do so, we need the following 
approximation result the proof of which is postponed to Section~\ref{P-prop10}.
\begin{prop}\label{approx-q}
For all $\beta>\beta'>0$, 
\begin{equation}
h^{2}\pa{p^{\beta},p^{\beta'}}\le\left\{\begin{array}{ll}(13/6)\left[(\beta/\beta')-1\right]^2&\;\mbox{if }
\:0<\beta'<\beta\le1,\\(7/4)(\beta-\beta')^2&\;\mbox{if }\:1<\beta'<\beta\le3,\\
\left[1.3(\beta-\beta')(\log\beta)\right]^2&\;\mbox{if }\:3<\beta'<\beta.\end{array}\right.
\label{Eq-gamma6}
\end{equation}
Moreover, 
\begin{equation}
h^{2}\pa{p^{\beta},p^0}\le\beta/2\quad\mbox{for }\,0<\beta\le1
\label{Eq-gamma7}
\end{equation}
and $\w_{p^\beta}(\lambda)\le(3/5)(\lambda-1)$ for all $\lambda\in [1,2]$ and $\beta\ge0$.
\end{prop}
We are now in a position to prove the following result.
\begin{cor}\label{C-QQ}
There exists a countable subset $\QQ$ of $\overline\QQ$ and a positive sub-probability 
$\gamma$ on $\QQ$ with the following properties: for all $p^{\beta} \in\overline\QQ$ there exists 
$p^b\in\QQ$ such that 
\begin{equation}
b\le\beta,\quad h^{2}\left(p^{\beta},p^b\right)\le n^{-1}
\label{eq-control}\end{equation}
and, for all $p^{b}\in\QQ$ and a suitable positive constant $c$ (independent of $n$),
\begin{equation}
\gamma\left(p^b\right)=\left\{\begin{array}{ll}c\left(\sqrt{n}\log n\right)^{-1}&\;\mbox{if }\:0\le b \le3,
\\cn^{-1}(b-3)^{-2}&\;\mbox{if }\:b>3.\end{array}\right.
\label{eq-control'}
\end{equation}
\end{cor}
\begin{proof}
We define $\QQ$ as the image by the application $\beta\mapsto p^\beta$ of a countable subset 
$B=B_1\cup B_2\cup B_3$ of $\R_+$. We first build $B_1=\{b_0<b_1<\ldots<b_m\}\subset[0,1]$ 
with $b_0=0$, $b_1=2/n$, $b_{i+1}=b_i\left(1+\sqrt{6/(13n)}\right)$ for $1\le i\le m-2$ and 
$b_m=1\le b_{m-1}\left(1+\sqrt{6/(13n)}\right)$, which defines the value of $m$. It follows from 
(\ref{Eq-gamma6}) and (\ref{Eq-gamma7}) that for any $\beta\le1$, there exists $b_i\le\beta$ 
with $h^{2}(p^{\beta},p^{b_i})\le n^{-1}$ and, since $b_{m-1}=2n^{-1}\left(1+\sqrt{6/(13n)}\right)^{m-2}
<1$, $m\le\kappa_1\sqrt{n}\log n$ for some constant $\kappa_1$. We then build $B_2=\{b_{m+1}
<\ldots<b_{m+l-1}\}\subset(1,3)$ in a similar way with $b_{i+1}=b_i+2/\sqrt{7n}$ for $m\le i\le m+l-2$ 
and $b_{m+l-1}<3\le b_{m+l-1}+2/\sqrt{7n}$. This implies that $l\le\kappa_2\sqrt{n}$ 
and, for $\beta\in(1,3)$, (\ref{eq-control}) holds with $b\in B_2$ by (\ref{Eq-gamma6}). Finally 
we build $B_3\subset[3,+\infty)$ as the infinite sequence $(b_{m+l+j})_{j\ge0}$ with 
\[
b_{m+l+j}=3+\frac{j}{1.3\sqrt{n}\alpha_j},\qquad\alpha_j=\log\left(3+\frac{j}{1.3\sqrt{n}}\right)
\quad\mbox{for }j\ge0.
\]
Since $\log3\le\alpha_j<\log(3+j)$, it follows that $b_{m+l+j}$ goes to infinity with $j$ and that
$\log(b_{m+l+j})<\alpha_j$. Therefore
\[
1.3\sqrt{n}\left(b_{m+l+j+1}-b_{m+l+j}\right)=\frac{j+1}{\alpha_{j+1}}-
\frac{j}{\alpha_j}<\frac{1}{\alpha_{j+1}}<\frac{1}{\log(b_{m+l+j+1})}.
\]
It follows from (\ref{Eq-gamma6}) that, for $\beta\ge3$, there exists $b\in B_3$ such that 
(\ref{eq-control}) holds. Since $|B_1\cup B_2|=m+l<\left(\kappa_1+\kappa_2\right)\sqrt{n}\log n$ and
\[
\frac{1}{n}\sum_{j\ge1}\left(b_{m+l+j}-3\right)^{-2}=(1.3)^2\sum_{j\ge1}\frac{\alpha_j^2}{j^2}
<(1.3)^2\sum_{j\ge1}\frac{\log^2(3+j)}{j^2}=\kappa_3<+\infty,
\]
we derive that 
\[
c^{-1}\gamma\left(\left\{p^b, b\in\QQ\right\}\right)=\sum_{j=0}^{m+l}\left(\sqrt{n}\log n\right)^{-1}+
\frac{1}{n}\sum_{j\ge1}\left(b_{m+l+j}-3\right)^{-2}<\left(1+\kappa_1+\kappa_2+\kappa_3\right).
\]
This implies that $\gamma$ is a sub-probability for a large enough value of $c$.
\end{proof}
We may now apply Theorem~\ref{thm-homo} to our example with the familly $\QQ$ and the 
sub-probability $\gamma$ provided by Corollary~\ref{C-QQ}, which leads to the following result. 
\begin{cor}
Let $\overline \QQ$ be the family of densities defined by~\eref{def-OQQ} and $\F$ be a family 
of models satisfying Assumption~\ref{hypo-homo1}. Let $\QQ$ and $\gamma$ be given by 
Corollary~\ref{C-QQ}. For $\overline \SS$ and $\Delta$ defined by~\eref{def-fam} and 
\eref{def-poids} respectively, the estimator $\widehat \gs$ satisfies, for all densities $p$ 
and vectors $\gf\in\R^{n}$,
\begin{eqnarray*}
C\E\cro{\gh^{2}(\gs,\widehat \gs)}&\le& \inf_{F\in\F}\cro{\inf_{\gg\in F}\gh^{2}(\gp_{\gf},\gp_{\gg})+
\overline V(F)\pa{1+\log_{+}\pa{n\over \overline V(F)}}+\Delta_{\pi}(F)}\\&&\mbox{}+
\inf_{\beta\ge0,\,\lambda>0}\left[nh^{2}(p,p^\beta_{0,\lambda})+\log_+(\beta)+\ab{\log\lambda}\right].
\end{eqnarray*}
In particular, if $p=p^\beta_{0,\tau}$ for some $p^\beta\in\overline\QQ$ with $\beta\ne2$ and $\tau>0$,
\begin{eqnarray*}
C(\beta)\E\cro{\gh^{2}(\gs,\widehat \gs)}&\le&\!\! \inf_{F\in\F}\cro{\inf_{\gg\in F}d_{1+[(2/\beta)\wedge1]}
\pa{\gf,\gg}+\overline V(F)\pa{1+\log_{+}\pa{n\over \overline V(F)}}+\Delta_{\pi}(F)}\\
&&\mbox{}+\inf_{\sigma>0}\left[nh^{2}(p^\beta,p^\beta_{0,\sigma})+
\ab{\log\tau\sigma}\right],
\end{eqnarray*}
where the distance $d_{1+\alpha}$ has been defined in (\ref{eq-d1}).
\end{cor}
\begin{proof}
Clearly, the family $\QQ$ satisfies Assumption~\ref{hypo-homo2}, the last requirement deriving 
from Proposition~\ref{approx-q} with $\w_{q}(\lambda)=(3/5)(\lambda-1)$ for all $\lambda\in [1,2]$ 
and $q\in \QQ$. Under Assumption~\ref{hypo-homo1} on the family $\F$, we may apply 
Theorem~\ref{thm-homo} and we get that, for $q\in \QQ$, $\lambda\in\R_{+}\setminus\{0\}$ and $\gf\in\R^{n}$, the risk of $\widehat \gs$ is bounded by $C(R_1+R_2)$ with
\[
R_1=\inf_{F\in\F}\cro{\inf_{\gg\in F}\gh^{2}(\gp_{\gf},\gp_{\gg})+\overline V(F)\pa{1+\log_{+}\pa{n\over \overline V(F)}}+\Delta_{\pi}(F)}+\inf_{i\ge 0}\cro{n2^{-i}+i}
\]
and
\[
R_2=\inf_{q\in\QQ,\,\lambda>0}\left[nh^{2}(p,q_{0,\lambda})+\Delta_{\gamma}(q)+
\ab{\log\lambda}\right].
\]
Let us first observe that, since $\overline V(F)\ge1$, $\overline V(F)\pa{1+\log_+\pa{n/\overline V(F)}}\ge1+\log n$, so that the term $\inf_{i\ge 0}\cro{n2^{-i}+i}\le2(1+\log n)$ in $R_1$ can be ignored at the price of the modification of the universal constant $C$. The Hellinger distance being unchanged by scale changes, $h^2(p,q_{0,\lambda})\le2h^2(p,p^\beta_{0,\lambda})+2h^2(p^\beta,q)$ for any $p^\beta\in\overline{\QQ}$ and $q\in\QQ$ so that, with $q$ chosen in order that $h^2(p^\beta,q)\le n^{-1}$,
\[
R_2\le\inf_{\beta\ge0,\,\lambda>0}\left[2nh^{2}(p,p^\beta_{0,\lambda})+2+
\Delta_{\gamma}(p^\beta)+\ab{\log\lambda}\right].
\]
In view of (\ref{eq-control}), $\Delta_{\gamma}(p^\beta)\le C'+\log n+2\log_+(\beta)$
and the first risk bound follows since we may again omit terms of order $\log n$. The second one then derives from the fact that $p^\beta$ is of order $(2/\beta)\wedge1$ for $\beta\ne2$ and (\ref{regul}), arguing as we did to get (\ref{Eq-Gamma9}).
\end{proof}

\subsection{Random design regression\label{MS4}}
We now turn back to the framework of Section~\ref{E4}. The same arguments with 
Theorem~\ref{main0} replacing Theorem~\ref{main00} lead to the following generalization of Theorem~\ref{T-RD}.
\begin{thm}\label{T-RD+}
If $q$ is unimodal and symmetric and $\F$ is a family of models for $f$ satisfying Assumption~\ref{hypo-homo1}, there exists a $\rho$-estimator $\widehat s=q_{\widehat f}$ of $s=p_{f}$ such that for all $\xi>0$, with probability at least $1-e^{-\xi}$,
\begin{eqnarray*}
\lefteqn{Ch^{2}(p_{f},q_{\widehat f})}\quad\\
&\le& \inf_{F\in \F}\ac{\inf_{g\in F}h^{2}\pa{p_{f},q_{g}}+{\overline V(F)\over n}\left[1+\log_+\left(\frac{n}{\overline V(F)}\right)\right]+{\Delta_{\pi}(F)\over n}}+{\xi\over n}\\
&\le& 2h^{2}(p,q)+\inf_{F\in \F}\ac{\inf_{g\in F}2\ell(f,g)+{\overline V(F)\over n}\left[1+\log_+\left(\frac{n}{\overline V(F)}\right)\right]+{\Delta_{\pi}(F)\over n}}+{\xi\over n}\\
\end{eqnarray*}
with $\ell$ given by (\ref{Eq-ell}).
In particular, if $p$ is known, unimodal and symmetric and $q=p$, 
\[
C\ell(f,\widehat f)\le \inf_{F\in \F}\ac{\inf_{g\in F}\ell(f,g)+{\overline V(F)\over n}\left[1+\log_+\left(\frac{n}{\overline V(F)}\right)\right]+{\Delta_{\pi}(F)\over n}}+{\xi\over n}
\]
with probability at least $1-e^{-\xi}$.\\
If, morever, (\ref{regul}) holds and $\max\left\{\sup_{g\in\FF}\|g\|_\infty,\|f\|_\infty\right\}\le b<+\infty$, then
\[
C'\norm{f-\widehat f}_{1+\alpha,\nu}^{1+\alpha}\le
\inf_{F\in \F}\ac{\inf_{g\in F}\norm{f-g}_{1+\alpha,\nu}^{1+\alpha}+{\overline V(F)\over n}
\left[1+\log_+\left(\frac{n}{\overline V(F)}\right)\right]+{\Delta_{\pi}(F)\over n}}+{\xi\over n}
\]
with probability at least $1-e^{-\xi}$ for some constant $C'$ depending only on $A_p,a_p,b$ and $\alpha$.
\end{thm}
We are not aware of any other procedure that leads to a comparable result. To illustrate this fact, let us consider the following example. We assume that $p$ is approximately known (approximately equal to $q$) and that the regression function $f$ takes the form
\[
f=\Psi(\zeta)\quad\mbox{with}\quad\zeta=\sum_{j=1}^M\beta_j\zeta_j,\;\;
\bm{\beta}=(\beta_1,\ldots,\beta_M)\in\R^M,
\]
where the $\zeta_j$ are $M$ given functions on $\mathscr{W}$, $\Psi$ is an unknown non-decreasing function on $\R$ and $M$ may be larger than $n$ but most of the coefficients 
$\beta_j$ are equal to zero, which means that we ignore which functions $\zeta_j$ are really influencial. We also choose a finite set $\{\Psi_k, 1\le k\le K\}$ of non-decreasing
functions to approximate $\Psi$.

Given $k\in\{1,\ldots,K\}$ and a non-void subset $m$ of ${\mathcal M}=\{1,\ldots,M\}$, we consider the model
\[
F_{k,m}=\left\{\Psi_k\left(\sum_{j\in m}\beta_j\zeta_j\right),\; \beta_j\in\R\quad\mbox{for all }j\in{\mathcal M}\right\}.
\]
This leads to the family $\F=\left\{\strut F_{k,m},1\le k\le K\mbox{ and }
m\subset{\mathcal M}\right\}$ of models for $f$ and we may set $\pi(F_{k,m})=\left[KM
(eM/|m|)^{|m|}\right]^{-1}$, so that
\[
\sum_{k=1}^K\sum_{m\in{\mathcal M}}\pi(F_{k,m})=\sum_{k=1}^K\sum_{l=1}^M
\sum_{\{m\in{\mathcal M}\,|\,|m|=l\}}\frac{1}{KM}\left(\frac{eM}{l}\right)^{-l}\le1,
\]
since $\binom{M}{l}\le(eM/l)^{l}$. It follows after some simplifications, that the quadratic risk of the corresponding $\rho$-estimator can be bounded in the following way (since $n\ge3$):
\begin{eqnarray*}
\lefteqn{C\E\left[h^{2}\left(p_{f},q_{\widehat f}\right)\right]}\qquad\\&\le&\ h^{2}(p,q)+\ \inf_{1\le k\le K}
\inf_{m\in\M}\cro{\inf_{g\in F_{k,m}}\ell\pa{\Psi(\zeta),g}+{|m|\over n}\log\left(\frac{nM}{|m|}\right)}+{\log(KM)\over n}.
\end{eqnarray*}
%
\section{VC-classes and subgraphs\label{S}}
We recall, following Dudley~\citeyearpar{MR876079} that
\begin{defi}\label{VC-Class}
Let $\CC$ be a non-empty class of subsets of a set $\Xi$. If $A\subset\Xi$ with 
$|A|=n$, then
\[
\Delta_n(\CC,A)=|\{A\cap B,\,B\in\CC\}|\qquad\mbox{and}\qquad
\Delta_n(\CC)=\max_{A\subset\Xi,\,|A|=n}\Delta_n(\CC,A).
\]
If $V=\sup\,\{n\in\N\,|\,\Delta_n(\CC)=2^n\}<+\infty$, then $\CC$ is a VC-class with VC-dimension $V$ and VC-index $\overline{V}=\inf\,\{n\in\N\,|\,\Delta_n(\CC)<2^n\}
=V+1$.

A class $\FF$ of functions from a set $\X$ with values in $(-\infty,+\infty]$ is 
VC-subgraph with dimension $V$ and index $\overline V$ if the class of subgraphs 
$\{(x,u)\in\X\times\R,\ f(x)>u\}$ as $f$ varies among $\FF$ is a VC-class of sets in
$\X\times\R$ with dimension $V$ and index $\overline V$.
\end{defi}
It immediately follows from this definition that any subset of a VC-subgraph class with index 
$\overline V$ is VC-subgraph with index not larger than $\overline V$, a property that we shall repeatedly use. Other known properties of VC-subgraph classes directly derive from the properties of VC-classes as described in van der Vaart and Wellner~\citeyearpar{MR1385671}, Lemma 2.6.17.

If $\FF$ is VC-subgraph with index $\overline V$ on a set $\X$ and $\norm{f}_{\infty}\le 1$ for all $f\in \FF$, it follows from Theorem~2.6.7 in van der Vaart and Wellner~\citeyearpar{MR1385671} that, for some numerical constant $K$ and all probability measures $Q$ on $\X$,
\begin{equation}\label{propVC}
N(\FF,Q,\epp)\le K\overline V  (16e)^{\overline V+1}
\epp^{-2\left(\overline V-1\right)}\quad\mbox{for }0<\epp<1.
\end{equation}
Noticing that $N(\FF,Q,1)=1$ since the closed ball of center 0 and radius 1 contains $\FF$, we derive, since $\overline{V}\ge1$, that there exists a universal constant $A$ such that
\begin{equation}
\log N(\FF,Q,\epp)\le 2\overline V  \log_+\pa{A/\epp}\quad\mbox{for all }\epp>0.
\label{Eq-h(x)}
\end{equation}
%
When the functions lying in $\FF$ are all non-negative, it is not difficult to see that we can restrict the class of subgraphs to that of  ``non-negative subgraphs" gathering the sets of the form $\{(x,u)\in\X\times \R,\ f(x)>u\ge 0\}$. If $S$ is a subset of a linear space of dimension $D$ then $S$ is VC-subgraph with index $\overline V\le D+2$ (see Lemma~2.6.15 in van der Vaart and Wellner~\citeyearpar{MR1385671}). We shall repeatedly use the following properties of VC-subgraph classes. 

\begin{prop}\label{perm}
Let $\FF$ be VC-subgraph with dimension $V$ on a set $\X$.
\begin{itemize}
\item[($i$)] For all functions $g$ on $\X$, $\FF+g=\{f+g,\ f\in\FF\}$ is VC-subgraph with dimension not larger than $V$.
\item[($ii$)] For all monotone function $\varphi$ on $\R$, $\varphi(\FF)=
\{\varphi\circ f,\ f\in \FF\}$ is VC-subgraph with dimension not larger than $V$. 
\item[($iii$)] The class $-\FF$ is VC-subgraph with dimension not larger than $V$.
\item[($iv$)] The class $\FF_{+}=\{f\vee 0,\ f\in\FF\}$ is VC-subgraph with dimension not larger than $V$.
\item[($v$)] If $\FF$ and $\GG$ are VC-subgraph with respective dimensions $V$ and $V'$, $\FF\vee\GG=\{f\vee g,\,f\in\FF,\,g\in\GG\}$ is VC-subgraph with dimension 
not larger than $4.701(V+V')$ and the same holds for $\FF\wedge\GG=\{f\wedge g,
\,f\in\FF,\,g\in\GG\}$.
\item[($vi$)] If $q$ is unimodal, the class $q(\FF)=\{q\circ f,\ f\in \FF\}$ is VC-subgraph with dimension not larger than $9.41V$.
\item[($vii$)] Let $\psi$ be given by (\ref{def-psi}), $g$ be some non-negative function 
on $\X$ and all functions in $\FF$ be non-negative. The class of functions
$\psi\left(\sqrt{\FF/g}\right)=\left\{\left.\psi(\sqrt{f/g})\,\right|\,f\in\FF\right\}$ is VC-subgraph with dimension not larger than $V$.
\end{itemize}
\end{prop}
\begin{proof}
For a proof of $(i)-(iv)$, we refer to Lemma~2.6.18 in van der Vaart and Wellner\citeyearpar{MR1385671} and for $(v)$ to the bound (1.2) from van der Vaart and Wellner~\citeyearpar{MR2797943} together with the relationship between VC-classes and VC-subgraph classes as explained in van der Vaart and Wellner\citeyearpar{MR1385671}, Section 2.6.5.

For $(vi)$ we argue as follows :  $q$ can be written as $\varphi_1\wedge\varphi_2$ where $\varphi_1$ is non-decreasing and $\varphi_2$ non-increasing so that $q\circ f=(\varphi_1\circ f)\wedge(\varphi_2\circ f)$. It follows that $q(\FF)\subset
\varphi_1(\FF)\wedge\varphi_2(\FF)$. The bound then follows from $(v)$.

Let us finally prove $(vii)$. It will be useful here and later on to introduce the function $\phi$ from $[0,+\infty]$ to $[-1,1]$ given by 
\begin{equation}
\phi(x)=\psi\left(\sqrt{x}\right),\quad\phi(0/0)=\phi(1)=0\quad\mbox{and}\quad
\phi(x/0)=\phi(+\infty)=1\;\mbox{ for all }x>0,
\label{Eq-phi}
\end{equation}
according to the conventions of Section~(\ref{A7}). Note that $\phi$ is continuous and increasing, hence one-to-one.

Let $(x_{1},u_{1}),\ldots,(x_{m},u_{m})$ be $m\ge 1$ points in $\X\times(-\infty,+\infty]$ shattered by the subgraphs of $\phi(\FF/g)=\psi\pa{\sqrt{\FF/g}}$. It suffices to prove that $m\le V$. First note that we necessarily have $u_i<1$ for all i since $\phi$ is bounded by 1. In particular for all $i$, $\phi^{-1}(u_i)<+\infty$. Besides, because of our convention, we also have $u_i\ge0$ for those $i$ such that $g(x_i)=0$,  since otherwise there would be no $f$ in $\FF$ such that $\phi(f/g)(x_i)\le u_i<0$. For all $I\subset \{1,\ldots,m\}$ there exists an element $f$ of $\FF$, depending on $I$, such that $i$ belongs to $I$ if and only if $\phi(f/g)(x_i)>u_i$. This is  equivalent to $f(x_i)>g(x_i)\phi^{-1}(u_i)$ if $g(x_i)>0$ and equivalent to $f(x_i)>0$ when $g(x_i)=0$ since $u_i\ge0$. In both cases, this is equivalent to $f(x_i)>g(x_i)\phi^{-1}(u_i)$. This means that the subgraphs of $\FF$ shatter the set $\{(x_i,g(x_i)\phi^{-1}(u_i)),i=1,...,m\}$, which is possible only when $m\le V$.
\end{proof}
\noindent{\bf Remark:} The proof of $(vi)$ extends recursively to multimodal functions with a given number $k$ of modes by noticing that a function with $k$ modes can be seen as the supremum of a unimodal function and a multimodal one with  $k-1$ modes. It follows that if $q$ is multimodal with $k$ modes, $q(\FF)$ is VC-subgraph with dimension not larger than $C(k)V$.
%

\section{Proofs\label{P}}

\subsection{Proofs of Theorem~\ref{main00},~\ref{main0} and~\ref{main-histo}\label{P1}}
All three theorems actually follow from the following (slightly) stronger result. 
\begin{thm}\label{main}
Let $\gs\in\LL_0$, $\overline\gs\in\S$, $G(\gs,\overline \gs)$ be an arbitrary function of $\gs$ and $\overline \gs$ and let the penalty function $\pen$ satisfy
\begin{equation}\label{def-pen}
\pen(\gt)+G(\gs,\overline \gs)\ge\pen_0(\gt,\gs,\overline \gs)=\inf_{\{S\in\SS\,|\,S\ni\gt\}}\ac{(1/8)D^{S}(\gs,\overline \gs)+\kappa\Delta(\overline{S})}\quad\mbox{for all }\gt\in\S.
\end{equation}
Then the estimator $\widehat \gs$ satisfies, for all $\xi>0$,
\begin{equation}
\P_{\gs}\left[\gh^{2}\pa{\gs,\widehat \gs}\le c_1\gh^{2}(\gs,\overline\gs)-\gh^{2}(\gs,\S)+
8c_{2}\left[\pen(\overline \gs)+G(\gs,\overline \gs)\right]+c_{3}(1.45+\xi)\right]\ge1-e^{-\xi}.
\label{Eq-Mboun}
\end{equation}
\end{thm}
%

\subsubsection{Proof of Theorem~\ref{main00}\label{P1a}}
Taking $\SS=\{\overline{S}\}$ hence $\S=S$, $\Delta(\overline{S})=0$, $\pen(\gt)=0$ and
$G(\gs,\overline \gs)=\pen_0(\gt,\gs,\overline \gs)=(1/8)D^{S}(\gs,\overline \gs)$ for all 
$\gt\in S$ we derive from (\ref{Eq-Mboun}) that
\[
\P_{\gs}\left[\gh^{2}\pa{\gs,\widehat \gs}\le c_1\gh^{2}(\gs,\overline\gs)-\gh^{2}(\gs,S)
+c_{2}D^{S}(\gs,\overline \gs)+c_3(1.45+\xi)\right]\ge 1-e^{-\xi}\quad\mbox{for all }\xi>0.
\]
Then (\ref{Mmain}) follows from the fact that this inequality is true for all choices of $\overline \gs\in S$. 
As to (\ref{risk}) and (\ref{Eq-risk1}), they follow by integration (see our remark following 
Theorem~\ref{main00}).

\subsubsection{Proof of Theorem~\ref{main0}\label{P1b}}
Inequality~\eref{Eq-bou5} is a straightforward consequence of Theorem~\ref{main} since, 
for all $(\gs,\overline \gs)\in \LL_{0}\times \LL_{0}$, $D^{S}(\gs,\overline\gs)\le \overline D^{S}$, 
therefore (\ref{def-pen6}) implies that (\ref{def-pen}) holds with $G(\gs,\overline \gs)=0$ and then \eref{Eq-bou5} follows from 
(\ref{Eq-Mboun}) .

\subsubsection{Proof of Theorem~\ref{main-histo}\label{P1c}}
Let us fix some $\overline \gs\in \S$. There exists $\overline{S'}\in\overline{\SS}$ such that
$\overline \gs\in S'$ and, by the definition of $D^{S}(\cdot,\cdot)$ and Assumption~\ref{cond-union},
\[
D^{S}(\gs,\overline\gs)\le D^{S\cup S'}(\gs,\overline \gs)\le D^{S\cup S'}\le \breve{D}(\overline{S})
+\breve{D}(\overline{S}')\quad\mbox{for all }\overline{S}\in\overline{\SS}.
\]
If we set $G(\gs,\overline\gs)=(1/8)\breve{D}(\overline{S}')$, the penalty function $\pen$ therefore
satisfies for all $\gt\in \S$
\begin{eqnarray*}
G(\gs,\overline\gs)+\pen(\gt)&=& (1/8)\breve{D}(\overline{S}')+\pen(\gt)\\&\ge& (1/8)\breve{D}(\overline{S}')
+\inf_{\{\overline{S}\in\overline{\SS}\,|\,S\ni\gt\}}\ac{(1/8)\breve{D}(\overline{S})+\kappa\Delta(\overline S)}\\
&\ge&\inf_{\{S\in\SS\,|\,S\ni\gt\}}\ac{(1/8)D^{S}(\gs,\overline \gs)+\kappa\Delta(\overline S)}
\;\;=\;\;\pen_0(\gt,\gs,\overline \gs)
\end{eqnarray*}
and \eref{def-pen} holds. It therefore follows from Theorem~\ref{main} that the estimator $\widehat \gs$  satisfies, for all  $\xi>0$ with probability at least $1-e^{-\xi}$,
\begin{eqnarray*}
\gh^{2}\pa{\gs,\widehat \gs}&\le&c_{1}\gh^{2}\pa{\gs,\overline \gs}+8c_{2}\cro{\pen(\overline \gs)+G(\gs,\overline\gs)}
+c_{3}(1.45+\xi)\\&=&c_{1}\gh^{2}\pa{\gs,\overline \gs}+8c_{2}\pen(\overline \gs)
+c_{2}\breve{D}(\overline{S}')+c_{3}(1.45+\xi).
\end{eqnarray*}
Since this holds for all $S'\ni\overline{\gs}$ and $\pen(\overline \gs)\ge\inf_{\{\overline{S}'\in\overline{\SS}\,|\,S'\ni\overline{\gs}\}}\{(1/8)\breve{D}(\overline{S}')\}$ by (\ref{def-pen0bis}),
\begin{equation}
\P_{\gs}\left[\gh^{2}\pa{\gs,\widehat \gs}\le c_{1}\gh^{2}\pa{\gs,\overline \gs}+16c_{2}\pen(\overline \gs)+c_{3}(1.45+\xi)\right]\ge 1-e^{-\xi}\quad\mbox{for all }\xi>0.
\label{Eq-histo1}
\end{equation}
Then \eref{eq-histo1} follows from the fact that $\overline \gs$ is arbitrary in $\S$.

Let us now fix some model $\overline S\in\overline \SS$, choose $\gs'\in \overline S$ such that $\gh^{2}(\gs,\gs')\le \gh^{2}(\gs,\overline S)+c_{3}/(160c_{1})$ and $\overline s\in S$ such that $\gh^{2}(\gs',\overline \gs)\le \gh^{2}(\gs',S)+c_{3}/(160c_{1})$. It follows that
\[
\gh^{2}(\gs,\overline \gs)\le2\gh^{2}(\gs,\overline S)+2\gh^{2}(\gs',S)+c_{3}/(40c_{1}).
\]
If equality holds in~\eref{def-pen0bis} for all $\gt\in\S$ and $\breve{D}({\overline{S})\le aD^S}$ with 
$2a\ge1$, it follows from \eref{Eq-histo1} that, for all $\xi>0$ with probability at least $1-e^{-\xi}$,
\begin{eqnarray*}
\gh^{2}\pa{\gs,\widehat \gs}&\le&2c_1\gh^{2}(\gs,\overline S)+2c_1\gh^{2}(\gs',S)+(c_{3}/40)
+16c_{2}\pen(\overline \gs)+c_{3}(1.45+\xi)\\
&\le& 2c_{1}\gh^{2}\pa{\gs,\overline S}+2c_{1}\gh^{2}(\gs',S)+2ac_{2}D^{S}+16c_2\kappa \Delta(\overline S)+c_{3}(1.475+\xi)\\
&\le& 2c_{1}\gh^{2}\pa{\gs,\overline S}+2a\pa{2c_1\gh^{2}(\gs',S)+c_{2}D^{S}}+c_{3}(2\Delta(\overline S)+1.475+\xi)\\
&\le& 2c_{1}\gh^{2}\pa{\gs,\overline S}+2aD(\overline S)+c_{3}(2\Delta(\overline S)+1.5+\xi),
\end{eqnarray*}
where the last inequality derives from~\eref{Eq-fine}. The conclusion follows since $\overline S$ is arbitrary in $\overline \SS$.
%

\subsection{Proof of Theorem~\ref{main}\label{P17}}
The proof will be divided into 2 steps.

{\bf Step 1.} Here we prove the following fondamental lemma. 
\begin{lem}\label{lem1}
If the function $\pen$ satisfies~\eref{def-pen}, then for all $\overline \gs\in\LL_{0}$ and $\xi>0$,
\begin{eqnarray*}
\lefteqn{\P_{\gs}\left[{1\over \sqrt{2}}\gZ(\gX,\overline \gs,\gt)\le4c_0\pa{\gh^{2}(\gs,\gt)+\gh^{2}(\gs,\overline \gs)}+\pen(\gt)+G(\gs,\overline \gs)+\kappa(1.4+\xi)\;\,\mbox{for all }\gt\in\S\right]}\hspace{125mm}\\&\ge&1-e^{-\xi}.
\end{eqnarray*}
\end{lem}
\begin{proof}
The proof relies on two propositions. The first one presents a version of Talagrand's result on the suprema of empirical processes that is proved in Massart~\citeyearpar{MR2319879}. 
An alternative solution would be to use Theorem~1.1 of Klein and Rio~\citeyearpar{MR2135312} instead of (\ref{eq-Massart}) below. This would lead to an analogue of (\ref{massart3}) with different values of the coefficients of $v^{2}$ and $x$ but not uniformly better. 
\begin{prop}\label{talagrand}
Let $T$ be some finite set, $U_{1},\ldots,U_{n}$ be independent centered random vectors with values 
in $\R^{T}$ and $Z=\sup_{t\in T}\ab{\sum_{i=1}^{n}U_{i,t}}$. If for some positive numbers $b$ and $v$, 
\[
\max_{i=1,\ldots,n}\ab{U_{i,t}}\le b\qquad\mbox{and}\qquad
\sum_{i=1}^{n}\E\left[U^2_{i,t}\right]\le v^{2}\ \quad\mbox{for all }t\in T,
\] 
then, for all positive $c$ and $x$,
\begin{equation}
\P\left[Z\le(1+c){\mathbb E}(Z)+(8b)^{-1}cv^2+2\left(1+8c^{-1}\right)bx\right]\ge1-e^{-x}.
\label{massart3}
\end{equation}
\end{prop}
%
\begin{proof}
The second displayed formula on page 170 of Massart~\citeyearpar{MR2319879} tells us that
\begin{equation}
\P\left[Z\le{\mathbb E}(Z)+2\sqrt{[2v^{2}+16b{\mathbb E}(Z)]x}+2bx\right]\ge1-e^{-x}.
\label{eq-Massart}
\end{equation}
To derive (\ref{massart3}) we use the fact that 
$2\sqrt{Ax}\le A/\!\left(16bc^{-1}\right)+\left(16bc^{-1}\right)\!x$ which results in
\[
2\sqrt{[2v^{2}+16b{\mathbb E}(Z)]x}\le 2c(16b)^{-1}v^2+c{\mathbb E}(Z)+16bc^{-1}x. 
\]
\end{proof}
Even though the result is stated for finite $T$, it can easily be extended to countable sets $T$ by monotone convergence.     

The second proposition we need is proved in Baraud~\citeyearpar{MR2834722} (more precisely, we refer to the proof of his Proposition~3 on page~386 with the difference that in this paper his function $\psi$ is equal to our function $\psi$ divided by $\sqrt{2}$ which involves an additional factor 2 for the control of the function $\psi^2$ as defined by (\ref{def-psi})).

\begin{prop}\label{variance}
Let $\gX=(X_{1},\ldots,X_{n})$ be a vector of independent random variables and $\gt,\overline \gs\in \LL_{0}$. Then
\[
\E\cro{\psi^{2}\pa{\sqrt{\gt\over \overline \gs}(\gX)}}=\sum_{i=1}^{n}\E\cro{\psi^{2}\pa{\sqrt{t_{i}\over \overline s_{i}}(X_{i})}}\le6\cro{\gh^{2}(\gs,\gt)+\gh^{2}(\gs,\overline \gs)}.
\]
\end{prop}
Let us now turn to the proof of Lemma~\ref{lem1}. 
We fix $\xi>0$, $S$ in $\SS$, $\tau=a/\!\left(32c_0^{2}\right)>0$ for some positive number $a$ to be chosen later and we set for all $j\in\N$,
\[
y_j^{2}=\left(\frac{5}{4}\right)^j\!\cro{D^{S}(\gs,\overline \gs)+\tau\left(\Delta(\overline S)+\xi+1.4\right)},
\;\quad x_j=\frac{y_j^{2}}{\tau}\ge\Delta(\overline S)+\xi+1.4\left(\frac{5}{4}\right)^j,
\]
\[
B^{S}_{j}(\gs,\overline \gs)=\ac{\gt\in S\mbox{ such that }
y_{j}^{2}<\gh^{2}(\gs,\gt)+\gh^{2}(\gs,\overline \gs)\le y^2_{j+1}}
\]
and
\[
Z_{j}^{S}(\gX,\overline \gs)={1\over \sqrt{2}}\,\sup_{\gt\in B^{S}_{j}(\gs,\overline \gs)}
\ab{\gZ(\gX,\overline \gs,\gt)}.
\]
For each $j\ge 0$, we may apply Proposition~\ref{talagrand} to the supremum 
$Z_{j}^{S}(\gX,\overline \gs)$ by taking  $T=B^{S}_{j}(\gs,\overline \gs)$ (which is countable 
as a subset of $S$) and 
\begin{equation}\label{def-U}
U_{i,\gt}=\frac{1}{\sqrt{2}}\left\{\psi\pa{\sqrt{t_{i}\over \overline s_{i}}(X_{i})}-
\E\cro{\psi\pa{\sqrt{t_{i}\over \overline s_{i}}(X_{i})}}\right\}\quad\mbox{for all }i=1,\ldots,n.
\end{equation}
For such a choice, the assumptions of the Proposition~\ref{talagrand} are met with $b=\sqrt{2}$ (since $\psi$ is bounded by $1$) and $v^{2}=3y_{j+1}^{2}$ (by Proposition~\ref{variance} and the definition of $B_{j}^{S}(\gs,\overline \gs)$). It therefore follows from \eref{massart3} that, with probability at least $1-e^{-x_{j}}$ and for all $\gt\in B^{S}_{j}(\gs,\overline \gs)$,
\begin{equation}
{1\over\sqrt{2}}\gZ(\gX,\overline \gs,\gt)\le Z_{j}^{S}(\gX,\overline \gs)
\le(1+c)\E\left[Z_{j}^{S}(\gX,\overline \gs)\right]+\frac{3}{8\sqrt{2}}cy_{j+1}^{2}+2\sqrt{2}\pa{1+8c^{-1}}x_{j}.
\label{Eq-pr1}
\end{equation}
Since $B^{S}_{j}(\gs,\overline \gs)\subset\B^{S}(\gs,\overline \gs,y_{j+1})$, it follows from  the definition of $D^{S}(\gs,\overline \gs)$ and the fact that $y_{j+1}^2>D^{S}(\gs,\overline \gs)$
that, 
\[
{\mathbb E}\left[Z_{j}^{S}(\gX,\overline \gs)\right]\le2^{-1/2}\w^{S}(\gs,\overline \gs,y_{j+1})
\le2^{-1/2}c_0y_{j+1}^2,
\]
and, since $x_j=4y_{j+1}^2/(5\tau)$, (\ref{Eq-pr1}) becomes
\[
{1\over\sqrt{2}}\gZ(\gX,\overline \gs,\gt)\le\frac{y_{j+1}^2}{\sqrt{2}}
\left[c_0(1+c)+\frac{3c}{8}+\frac{16\pa{1+8c^{-1}}}{5\tau}\right].
\]
Setting $c=16(2\tau)^{-1/2}$, we get with probability at least $1-e^{-x_{j}}$ and for all 
$\gt\in B^{S}_{j}(\gs,\overline \gs)$,
\begin{eqnarray*}
\lefteqn{{1\over\sqrt{2}}\gZ(\gX,\overline \gs,\gt)-4c_0\cro{ \gh^{2}(\gs,\gt)+\gh^{2}(\gs,\overline \gs)}}
\hspace{30mm}\\&\le&\frac{y_{j+1}^2}{\sqrt{2}}\left[c_0\left(1+\frac{16}{\sqrt{2\tau}}\right)
+\frac{6}{\sqrt{2\tau}}+\frac{16}{5\tau}+\frac{16}{5\sqrt{2\tau}}-\frac{16c_0\sqrt{2}}{5}\right].
\end{eqnarray*}
The bracketed factor writes
\begin{eqnarray*}
\lefteqn{c_0+\frac{64c_0^2}{\sqrt{a}}+\frac{24c_0}{\sqrt{a}}+\frac{512c_0^2}{5a}
+\frac{64c_0}{5\sqrt{a}}-\frac{16c_0\sqrt{2}}{5}}\hspace{30mm}\\&=&\frac{c_0}{5}\left[5-16\sqrt{2}+
\frac{184}{\sqrt{a}}+4\left(2-\sqrt{2}\right)\left(\frac{5}{\sqrt{a}}+\frac{8}{a}\right)\right],
\end{eqnarray*}
which is negative for $a=125.4$. With this choice of $a$, $\tau=62.7(4c_0)^{-2}$ and for all $j\in\N$,
\[
\P_{\gs}\left[{1\over\sqrt{2}}\gZ(\gX,\overline \gs,\gt)-4c_0\cro{\gh^{2}(\gs,\gt)+\gh^{2}(\gs,\overline \gs)}<0\;\,\mbox{for all }\gt\in B^{S}_{j}(\gs,\overline \gs)\right]\ge1-e^{-x_j}.
\]
Let us now define
\[
Z^{S}(\gX,\overline\gs)={1\over\sqrt{2}}\sup_{\gt\in\B^{S}(\gs,\overline\gs,y_0)}\ab{\gZ(\gX,\overline\gs,\gt)}
\]
and apply Proposition~\ref{talagrand} in a similar way to $Z^{S}(\gX,\overline \gs)$ with $x=x_0=y_0^2/\tau$ and $c=16\sqrt{3/(10\tau)}=64c_0/\sqrt{209}$. We then deduce analogously that, with probability at least $1-e^{-x_0}$ and for all $\gt\in \B^{S}(\gs,\overline \gs,y_0)$,
\begin{eqnarray*}
{1\over \sqrt{2}}\gZ(\gX,\overline \gs,\gt)&\le&Z^{S}(\gX,\overline \gs)\;\;\le\;\;\frac{y_0^2}{\sqrt{2}}
\left[c_0(1+c)+\frac{3c}{8}+\frac{4\pa{1+8c^{-1}}}{\tau}\right]
\\&\le&\frac{4c_0y_0^2}{\sqrt{2}}
\left[16c_0\left(\frac{1}{\sqrt{209}}+\frac{1}{62.7}\right)+\frac{1}{4}+\frac{6}{\sqrt{209}}+
\frac{2\sqrt{209}}{62.7}\right]\;\;<\;\;0.122\,y_0^2.
\end{eqnarray*}
Since $\{\B^{S}(\gs,\overline \gs,y_0), \{B^{S}_{j}(\gs,\overline \gs),\ j\ge 0\}\}$ provides a partition of $S$, by putting all these inequalities together we derive that for all $\gt\in S$, 
\[
{1\over \sqrt{2}}\gZ(\gX,\overline \gs,\gt)-4c_0\pa{\gh^{2}(\gs,\gt)+\gh^{2}(\gs,\overline \gs)}
<0.122\,y_0^2<(1/8)D^{S}(\gs,\overline \gs)+\kappa\left[\Delta(\overline S)+\xi+1.4\right],
\] 
except on a set of probability not larger than
\[
e^{-x_0}+\sum_{j\ge0}e^{-x_j}\le e^{-\xi-\Delta(\overline S)}\cro{2e^{-1.4}+\sum_{j\ge1}e^{-1.4\times(5/4)^j}}<e^{-\xi-\Delta(\overline S)}.
\]
The result finally extends to all $\gt\in \S$ by summing these bounds over $S\in\SS$ and using
(\ref{Eq-subprob}).
\end{proof}

{\bf Step 2.} 
Let us now set, for $\gs,\gt,\gt'\in\LL_{0}$,
\[
\gT(\gs,\gt,\gt')=\E\cro{\gT(\gX,\gt,\gt')}=\sum_{i=1}^nT(s_i,t_i,t_i').
\]
Applying inequality~\eref{eq-UbT} to each coordinate $s_{i},\:\overline s_{i}$ and $t_{i}$ of $\gs,\:\overline \gs$ and $\gt$ respectively and summing these inequalities over $i\in\{1,\ldots,n\}$ leads to
\begin{equation}
\gT(\gs,\overline \gs,\gt)\le c_2\gh^{2}(\gs,\overline \gs)-8c_0\gh^{2}(\gs,\gt)\quad\mbox{for all }\gs\in\LL_{0},\:\overline\gs\mbox{ and }\gt\in\S.
\label{eq-UbTb}
\end{equation}
Let us fix $\gs\in\LL_{0}$ and $\overline \gs\in \S$. Recalling that 
$\gZ(\gX,\overline \gs,\gt)/\sqrt{2}=\gT(\gX,\overline \gs,\gt)-\gT(\gs,\overline \gs,\gt)$, 
we deduce from Lemma~\ref{lem1} that, with probability at least $1-e^{-\xi}$ and for all $\gt\in\S$,
\begin{equation}\label{eq-lem1}
\gT(\gX,\overline \gs,\gt)-\gT(\gs,\overline \gs,\gt)\le4c_0\pa{\gh^{2}(\gs,\gt)+\gh^{2}(\gs,\overline \gs)}+\pen(\gt)+G(\gs,\overline \gs)+\kappa(1.4+\xi),
\end{equation}
which, together with~\eref{eq-UbTb}, leads to 
\begin{eqnarray}
\lefteqn{\gT(\gX,\overline \gs,\gt)-\pen(\gt)}\hspace{20mm}\nonumber\\&\le&(4c_0+c_2)\gh^{2}(\gs,\overline \gs)-4c_0\gh^{2}(\gs,\gt)+G(\gs,\overline \gs)+\kappa(1.4+\xi)\quad\mbox{for all }\gt\in\S.
\label{eq-bT}
\end{eqnarray}
Hence, with probability at least $1-e^{-\xi}$,
\begin{eqnarray*}
\overline \gup(\S,\overline\gs)-\pen(\overline\gs)&=&\sup_{\gt\in \S}\cro{\gT(\gX,\overline \gs,\gt)-\pen(\gt)}\\&\le&(4c_0+c_2)\gh^{2}(\gs,\overline \gs)-4c_0\gh^{2}(\gs,\S)+G(\gs,\overline \gs)+\kappa(1.4+\xi)
\end{eqnarray*}
and it follows from the definitions of $\widehat \gs$ and $\overline\gup(\S,\cdot)$ respectively that 
\begin{eqnarray}
\overline\gup(\S,\widehat\gs)&\le&\overline\gup(\S,\overline\gs)+(\kappa/10)\nonumber\\&\le&
(4c_0+c_2)\gh^{2}(\gs,\overline \gs)-4c_0\gh^{2}(\gs,\S)+\pen(\overline\gs)+
G(\gs,\overline \gs)+\kappa(1.5+\xi)
\label{eq-bgam}
\end{eqnarray}
and $\gT(\gX,\widehat{\gs},\overline\gs)+\pen(\widehat{\gs})\le\overline\gup(\S,\widehat{\gs})
+\pen(\overline\gs)$. Therefore, using~\eref{eq-bT} with $\gt=\widehat \gs$, the fact that $\gT(\gX,\overline \gs,\widehat \gs)=-\gT(\gX,\widehat \gs,\overline\gs)$ and \eref{eq-bgam}, we derive that, with probability at least $1-e^{-\xi}$,
\begin{eqnarray*}
4c_0\gh^{2}(\gs,\widehat\gs)&\le&(4c_0+c_2)\gh^{2}(\gs,\overline\gs)+\gT(\gX,\widehat\gs,\overline\gs)
+\pen(\widehat \gs)+G(\gs,\overline \gs)+\kappa(1.4+\xi)\\&\le&(4c_0+c_2)\gh^{2}(\gs,\overline \gs)
+\overline\gup(\S,\widehat\gs)+\pen(\overline \gs)+G(\gs,\overline \gs)+\kappa(1.4+\xi)\\&\le&
2(4c_0+c_2)\gh^{2}(\gs,\overline \gs)-4c_0\gh^{2}(\gs,\S)+2\pen(\overline \gs)+
2G(\gs,\overline \gs)+2\kappa(1.45+\xi),
\end{eqnarray*}
which leads to the result since $4c_0=(4c_2)^{-1}$ and $c_1=2+c_2/(2c_0)$ by 
(\ref{Eq-cons1}).

\subsection{Proof of Proposition~\ref{P-KLrob}\label{P22}}
It actually follows from the next one:
%
\begin{prop}\label{P-KLrob'}
If $\gs$ and $\overline{\gs}\in\LL_0$ and $T(\gX)$ is such that $\mathbb{P}_{\overline{\gs}}\cro{T(\gX)\ge z}\le ae^{-z}$ for all $z\ge0$ and some $a>0$, then
\begin{equation}
\E[T(\gX)]\le\left(1+\zeta^{-1}\right)\left[\log(1+a\zeta)+\gK\right]\quad\mbox{for all }\zeta>0\quad\mbox{and}\quad\gK=\gK(\gs,\overline\gs).
\label{Eq-KLc}
\end{equation}
In particular $\E[T(\gX)]\le1+\zeta_0$ where $\zeta_0$ is the largest solution of the equation $\zeta=\log(1+a\zeta)+\gK$ in $(-a^{-1},+\infty)$.
\end{prop}
\begin{proof}
We start with the following lemma which appears in a slightly different form in Barron~\citeyearpar{MR1154352}.
\begin{lem}\label{Info-Th}
Let $P$ and $Q$ be two probabilities on $(\X,\A)$ and $f$ a function from $(\X,\A)$ to $\R$ such that 
$\int_{\X}(f\wedge0)\,dP>-\infty$. Then
\[
\int_{\X}fdP\le\log\pa{\int_{\X}e^{f}dQ}+K(P,Q)\le\log\pa{1+\int_{0}^{+\infty}e^{\xi}Q[f>\xi]\,d\xi}+K(P,Q).
\]
\end{lem}
\begin{proof}
The classical variational formula for the Kullback-Leibler divergence asserts that
\[
K(P,Q)=\sup_{g\in{\mathcal G}} \int_{\X} g dP\quad\mbox{with}\quad{\mathcal G}=\ac{g:(\X,\A)\to \R\;
\mbox{ such that } \int_{\X} e^{g}dQ=1}.
\]
For $g=f-\log\pa{ \int_{\X} e^{f}dQ}$ which belongs to ${\mathcal G}$, we obtain that 
\[
K(P,Q)\ge \int_{\X} g dP=\int_{\X}fdP-\log\pa{ \int_{\X} e^{f}dQ}
\]
which leads to the first inequality. The second inequality derives from 
\[
\int_{\X}e^{f}dQ=\int_0^{+\infty}Q\left[e^{f}>t\right]dt=\int_{-\infty}^{+\infty}Q[f>\xi]e^\xi\,d\xi\le1+\int_{0}^{+\infty}e^{\xi}Q[f> \xi]\,d\xi.
\]
\end{proof}
To prove Proposition~\ref{P-KLrob'} we apply the lemma with $f=\lambda T$, $0<\lambda<1$, $P=\mathbb{P}_{\gs}$ and $Q=\mathbb{P}_{\overline{\gs}}$, getting
\[
\lambda\E[T(\gX)]\le\log\pa{1+\int_{0}^{+\infty}e^{\xi}\mathbb{P}_{\overline{\gs}}
[\lambda T(\gX)>\xi]\,d\xi}+\gK.
\]
Hence, setting $\zeta=\lambda/(1-\lambda)>0$ so that 
$\lambda=\zeta/(\zeta+1)$, we get
\begin{eqnarray*}
\frac{\zeta}{\zeta+1}\E[T(\gX)]&\le&\log\pa{1+\int_{0}^{+\infty}e^{\xi}\mathbb{P}_{\overline{\gs}}
\left[T(\gX)>\xi/\lambda\right]d\xi}+\gK\\&\le&
\log\pa{1+a\int_{0}^{+\infty}\exp\left[-\xi/\zeta\right]d\xi}+\gK\;\;=\;\;\log(1+a\zeta)+\gK,
\end{eqnarray*}
which proves (\ref{Eq-KLc}). The function $g(\zeta)=\zeta-\log(1+a\zeta)-\gK$ is strictly convex 
on $(-a^{-1},+\infty)$ with a minimum equal to $1-a^{-1}-\log a-\gK\le0$ when $\zeta=1-a^{-1}$ and $g(0)=-\gK\le0$ so that $\zeta_0\ge0$ (actually $>0$ except if $\gK=0$ and $a\le1$) and the bound 
$\E[T(\gX)]\le1+\zeta_0$ immediately follows when $\zeta_0>0$. If $\zeta_0=0$, then $\gK=0$ and $a\le1$ so that $\E[T(\gX)]\le a$.
\end{proof}
To prove Proposition~\ref{P-KLrob} we bound $\E[T(\gX)]\le1+\zeta_0$ in the following way, setting $f(x)=x-\log(1+x)$. We first observe that, since $\log(1+uv)\le\log(1+u)+\log(1+v)$ for all $u,v\ge0$,
\[
f(\zeta_0)=\zeta_0-\log(1+\zeta_0)\le c=\log(1+a)+\gK
\]
and, since $f$ is increasing on $[0,+\infty[$, $\zeta_0\le f^{-1}(c)$. Moreover,
\[
f\left(c+\log\left(1+c+\sqrt{2c}\right)\right)-c=\log\left(1+c+\sqrt{2c}\right)-
\log\left(1+c+\log\left(1+c+\sqrt{2c}\right)\right)
\]
has the same sign as
\begin{equation}
\left(1+c+\sqrt{2c}\right)-\left(1+c+\log\left(1+c+\sqrt{2c}\right)\right)=\sqrt{2c}-\log\left(1+c+\sqrt{2c}\right),
\label{Eq-P0}
\end{equation}
which has the sign of $\exp\left[\sqrt{2c}\right]-\left(1+c+\sqrt{2c}\right)>0$ since $c>0$ and $e^x>1+x+\left(x^2/2\right)$ for $x>0$. It follows that $c<f\left(c+\log\left(1+c+\sqrt{2c}\right)\right)$ and finally, using again the fact that the right-hand side of (\ref{Eq-P0}) is positive,
\[
1+\zeta_0\le1+f^{-1}(c)<1+c+\log\left(1+c+\sqrt{2c}\right)<1+c+\sqrt{2c},
\]
which completes the proof of \eref{Eq-KLa}. To get (\ref{Eq-KLb}) we apply (\ref{Eq-KLa}) to the random variable $T'(\gX)=b[T(\gX)-z_0]$.

\subsection{Proof of Proposition~\ref{Kolt2}\label{P3}}
It follows from the next lemma to be proved afterwards.
\begin{lem}\label{Koltb}
Let $X_{1},\ldots,X_{n}$ be independent random variables defined on a probability space $(\Omega,\A,\P)$ and with values in $\X$, $\FF$ a class of functions on $\X$ bounded by 1 and $\Ent$ a function satisfying Assumption~\ref{A-VC}. If 
\begin{equation}\label{Bvar}
\sup_{f\in\FF}\sum_{i=1}^{n}{\mathbb{E}}\cro{f^{2}(X_{i})}\le v^{2} 
\end{equation}
and 
\begin{equation}
\log N\pa{\FF,{1\over n}\sum_{i=1}^{n}\delta_{X_{i}(\omega)},z}\le\Ent\pa{1\over z}
\quad\mbox{for all }\omega\in\Omega\mbox{ and } 0<z\le2,
\label{E-entr}
\end{equation}
then there exists a universal constant $C_0$ such that, 
\[
{\mathbb{E}}\cro{\sup_{f\in \FF}\ab{\sum_{i=1}^{n}\pa{f(X_{i})-{\mathbb{E}}\cro{f(X_{i})}}}} 
\le C_0\cro{vL\sqrt{H}+L^{2}H}\quad\mbox{with}\quad H=\Ent\pa{{\sqrt{n}\over2v}\bigvee\frac{1}{2}}.
\]
\end{lem}
To prove (\ref{B10}) for a given value of $y$ we use (\ref{Eq-VCwS}) and apply this lemma to the 
family $\Y$, the elements of which are bounded by 1 and satisfy~\eref{Bvar} with $v^{2}=6y^{2}$ 
by Proposition~\ref{variance}, and to the function $\Ent_{\!y}$, so that $H=H_y$ and $L=L_y$. 
Since $2ab\le\alpha a^2+\alpha^{-1}b^2$ for all $a,b\in\R$ and $\alpha=c_0/C_0$, we derive that
\[
\w^{S}(\gs,\overline \gs,y)\le{C_{0}\over 2}\cro{\alpha y^{2}+\pa{2+6\alpha^{-1}}L_{y}^{2}H_y}
\le{c_0y^{2}\over2}+C_{0}\pa{1+{3C_0\over c_0}}L_{y}^{2}H_{y}.
\]
Then (\ref{B21}) follows from the definition of $D^{S}(\gs,\overline\gs)$.

Let us now turn to the proof of Lemma~\ref{Koltb}. 
The line of proof is the same as that of Theorem~3.1 of Gin\'e and Koltchinskii~\citeyearpar{MR2243881} with minor changes due to the fact that we consider non i.i.d.\ random variables $X_{i}$. Similar arguments were used in Massart and N\'ed\'elec~\citeyearpar{MR2291502} for classes $\FF$ of indicator functions. 

By a symmetrization argument, 
\[
{\mathbb{E}}\cro{\sup_{f\in \FF}\ab{\sum_{i=1}^{n}\pa{f(X_{i})-{\mathbb{E}}\cro{f(X_{i})}}}}\le 2E=2{\mathbb{E}}\cro{\sup_{f\in\FF}\ab{\sum_{i=1}^n\eps_{i}f(X_{i})}},
\]
where the $\eps_{i}$ are Rademacher random variables independent of the $X_{i}$. Arguing as in Gin\'e and Koltchinskii with $F=1$, we get
\[
E\le Cn^{1/2}{\mathbb{E}}\cro{\int_{0}^{2\widehat \sigma_{n}}\sqrt{\Ent(1/z)}\,dz}
\quad\mbox{with}\quad
\widehat \sigma_{n}^{2}=\sup_{f\in\FF}{1\over n}\sum_{i=1}^{n}f^{2}(X_{i})\le 1.
\]
The function $u\mapsto\Ent(1/u)$ being non-increasing, $u\mapsto \int_{0}^{u}\Ent(1/z)dz$ is concave and therefore 
\[
E\le Cn^{1/2}\int_{0}^{2{\mathbb{E}}[\widehat \sigma_{n}]}\sqrt{\Ent(1/z)}\,dz\le Cn^{1/2}\int_{0}^{2\sqrt{{\mathbb{E}}[\widehat \sigma_{n}^{2}]}}\sqrt{\Ent(1/z)}\,dz.
\]
Symmetrization and contraction arguments together with the fact that $|f|\le 1$ for all $f\in\FF$ lead to
\begin{eqnarray*}
{\mathbb{E}}[\widehat \sigma_{n}^{2}]\le B^{2}={v^{2}+8E\over n}\wedge1 \qquad\mbox{hence}\qquad
\frac{v}{\sqrt{n}}\wedge1\le B\le\frac{\left(v+\sqrt{8E}\right)}{\sqrt{n}}\wedge1.
\end{eqnarray*}
Using a change of variables, the definition of $L$, the monotonicity of $\Ent$ and the bounds for $B$, 
we obtain that
\begin{eqnarray*}
E&\le& Cn^{1/2}\int_{0}^{2B}\sqrt{\Ent(1/z)}\,dz\;\;=\;\;Cn^{1/2}\int_{1/(2B)}^{+\infty}{\sqrt{\Ent(u)}\over u^{2}}\,du\\
&\le&2CLn^{1/2}B\sqrt{\Ent\pa{1\over 2B}}\;\;\le\;\;2CL\left(v+\sqrt{8E}\right)\sqrt{H}.
\end{eqnarray*}
Solving this inequality with respect to $E$ leads to the conclusion.

\subsection{Proof of Proposition~\ref{Kolt}\label{P2}}
Since
\[
\w^{S}({\bf s},\overline{\bf s},y)=\sqrt{2}\,\E\cro{\sup_{t\in T}\ab{\sum_{i=1}^{n}\pa{U_{i,t}-\mathbb{E}\cro{U_{i,t}}}}}\quad\mbox{with}\quad T=\B^{S}(\gs,\overline \gs,y)
\]
and the $U_{i,t}$ defined by~\eref{def-U}, the result derives from the next proposition. In this case
$T\subset\B^{S}(\gs,y)$ so that $|T|\le  |\B^{S}(\gs,y)|$, $H=\log_{+}(2|\B^{S}(\gs,y)|)$, $b=\sqrt{2}$ and $v^{2}=3y^{2}$ (because of Proposition~\ref{variance}). 
%
\begin{prop}\label{P-Espsup}
Let $T$ be a finite set and $U_{1},\ldots,U_{n}$ independent random variables with values in 
$\R^{T}$ satisfying for all $t\in T$
\begin{equation}
\max_{i=1,\ldots,n}|U_{i,t}|\le b\;\;a.s.;\qquad\sum_{i=1}^{n}\mathbb{E}\cro{U_{i,t}^{2}}\le v^{2}
\qquad\mbox{and}\qquad \log_{+}(2|T|)\le H
\label{Eq-Espsup}
\end{equation}
for some positive numbers $b,v$ and $H$. Then,
\[
\mathbb{E}\cro{\sup_{t\in T}\ab{\sum_{i=1}^{n}\pa{U_{i,t}-\mathbb{E}\cro{U_{i,t}}}}}\le
bH+v\sqrt{2H}.
\]
\end{prop}
\begin{proof}
Since the $U_{i,t}$ are independent for $i=1,\ldots,n$ and satisfy (\ref{Eq-Espsup}) for all 
$t\in T$, classical computations of the Laplace transform of $S_{n,t}=\sum_{i=1}^{n}\left(U_{i,t}-{\mathbb{E}}\cro{U_{i,t}}\right)$ give, for $\lambda\in(0,1/b)$, 
\[
{\mathbb{E}}\cro{\exp\pa{\lambda |S_{n,t}|}}\le 2\exp\cro{{\lambda^{2}v^{2}\over 2(1-\lambda b)}}
\quad\mbox{for all }t\in T.
\]
For a proof of this inequality we refer to inequality~(2.21) in Massart~\citeyearpar{MR2319879}. Applying Jensen's inequality and then this bound leads to
\begin{eqnarray*}
\mathbb{E}\cro{\sup_{t\in T}|S_{n,t}|}&=&{1\over \lambda}\log\pa{\exp\pa{\mathbb{E}
\cro{\lambda\sup_{t\in T}|S_{n,t}|}}}\;\;\le\;\;{1\over\lambda}\log\mathbb{E}
\cro{\exp\pa{\lambda\sup_{t\in T}|S_{n,t}|}}\\&\le&{1\over\lambda}\log\pa{\sum_{t\in T}\mathbb{E}\cro{\exp\pa{\lambda |S_{n,t}|}}}\;\;\le\;\;\frac{H}{\lambda}+{\lambda v^{2}\over 2(1-\lambda b)}.
\end{eqnarray*}
Minimizing the right-hand side with respect to $\lambda\in(0,1/b)$ leads to 
$\lambda=\left(v+b\sqrt{2H}\right)^{-1}\!\sqrt{2H}$ and 
finally $\mathbb{E}\cro{\sup_{t\in T}|S_{n,t}|}\le bH+v\sqrt{2H}$.
\end{proof}

\subsection{Proof of Proposition~\ref{histo}\label{P6}}
Let us denote by $P_{s}$ the probability associated to $s$ on $(\X,\A,\mu)$. For $t=\sum_{I\in\I}\left[t_{I}/\mu(I)\right]\1_{I}$ with $\gt\in\B^{S}(\gs,\overline \gs,y)$, $P_s$ almost surely,
\[
\psi\pa{\sqrt{t\over \overline s}(X_{i})}=\sum_{I\in\J}\psi\pa{\sqrt{t_{I}\over \overline s_{I}}}\1_{I}(X_{i})\ \ \mbox{for all}\ i=1,\ldots,n
\]
and, by Cauchy-Schwarz Inequality,
\begin{eqnarray*}
S_{n}(t)&=&\ab{\sum_{i=1}^{n}\cro{\psi\pa{\sqrt{t\over\overline s}(X_{i})}-
\E\cro{\psi\pa{\sqrt{t\over \overline s}(X_{i})}}}}\\&=&\ab{\sum_{I\in\J}
\psi\pa{\sqrt{t_{I}\over\overline s_{I}}}\sum_{i=1}^{n}\cro{\strut\1_{I}(X_{i})-\E\cro{\1_{I}(X_{i})}}}\\
&\le&\cro{\sum_{I\in\J}\psi^{2}\pa{\sqrt{t_{I}\over\overline s_{I}}}P_{s}(I)}^{1/2}\cro{\sum_{I\in\J}\pa{\sum_{i=1}^{n}{{\1_{I}(X_{i})-\E\cro{\1_{I}(X_{i})}}\over\sqrt{P_{s}(I)}}}^{2}}^{1/2}.
\end{eqnarray*}
By Proposition~\ref{variance}, for all $\gt\in \B^{S}(\gs,\overline \gs,y)$,
\[
n\E\cro{\psi^{2}\pa{\sqrt{t\over \overline s}(X_{1})}}=n\sum_{I\in\I}\psi^{2}\pa{\sqrt{t_{I}\over \overline s_{I}}}P_{s}(I)\le 6y^{2},
\]
hence, $P_s$ almost surely,
\begin{eqnarray*}
\sup_{\gt\in \B^{S}(\gs,\overline \gs,y)}S_{n}(t)\le{y\sqrt{6}\over\sqrt{n}}\cro{\sum_{I\in\J}
\pa{\sum_{i=1}^{n}{{\1_{I}(X_{i})-\E\cro{\1_{I}(X_{i})}}\over \sqrt{P_{s}(I)}}}^{2}}^{1/2}.
\end{eqnarray*}
Taking expectations on both sides and using the concavity of the square-root, we get
\[
\w^{S}\pa{\gs,\overline \gs,y}\le {y\sqrt{6}\over \sqrt{n}}\times \sqrt{\sum_{I\in\J}\sum_{i=1}^{n}{P_{s}(I)\over P_{s}(I)}}= y\sqrt{6|\J|},
\]
which leads to the result. 

\subsection{Proof of Theorem~\ref{thm-MLE}\label{P14}}
In order to simplify the notations, when using the Hellinger distance on our model, we shall 
write $h(\theta,\theta')$ instead of $h(t_\theta,t_{\theta'})$. All along this proof, we shall denote 
by $\ab{\cdot}$ the Euclidean distance on $\R^{d}$ (as well as the absolute value when $d=1$) and
by $A_i$, $2\le i\le9$, constants that only depend on the structure of the parametric 
model $\overline{S}$ as described by Assumption~\ref{MLE}.

Since the parametric family $\{t_\theta, \theta\in\Theta'\}$ is regular it has a continuous Fisher 
Information matrix $I(\theta)$ which is also invertible on the compact set $\Theta$ by 
Assumption~\ref{MLE}-$(ii)$. Therefore its eigenvalues are bounded away from zero and 
infinity on $\Theta$ which implies --- see (7.20) p.82 of the book by Ibragimov and Has{'}minski{\u\i}~\citeyearpar{MR620321} --- that
\begin{equation}\label{metric-h}
A_2\ab{\overline\theta-\theta}\le h\left(\overline\theta,\theta\right)\le A_3\ab{\overline\theta-\theta}
\quad\mbox{with }0<A_2<A_3\quad\mbox{for all }\overline\theta,\,\theta\in \Theta.
\end{equation}
It then follows from Assumption~\ref{MLE}-$(iii)$ that
\begin{equation}
\norm{\sqrt{t_{\theta}\over t_{\theta'}}-\sqrt{t_{\overline\theta}\over t_{\theta'}}}_{\infty}\le
\frac{A_1}{A_2}h\left(\theta,\overline{\theta}\right)=A_4h\left(\theta,\overline{\theta}\right)
\quad\mbox{for all }\theta,\,\overline\theta\mbox{ and }\theta'\in\Theta
\label{Eq-control-0}
\end{equation}
and $\overline{S}$ is therefore identifiable.
Using the triangular inequality together with the facts that $\psi$ is 1.15-Lipschitz and satisfies $\psi(1/x)=-\psi(x)$ for all $x>0$, we get for all $\theta,\overline \theta,\theta',\overline \theta'$ in $\Theta$,
\begin{equation}\label{control-1}
\norm{\psi\pa{\sqrt{t_{\theta'}\over t_{\theta}}}-\psi\pa{\sqrt{t_{\overline \theta'}\over t_{\overline \theta}}}}_{\infty}\le 1.15A_4\cro{h\left(\theta,\overline{\theta}\right)+h\left(\theta',\overline{\theta}'\right)}.
\end{equation}
Moreover 
%
\begin{lem}\label{L-cont}
The function
\[
(t,t')\mapsto\rho\left(t',\frac{t+t'}{2}\right)-\rho\left(t,\frac{t+t'}{2}\right)
\]
is uniformly continuous on $\overline{S}\times\overline{S}$ with respect to the Hellinger distance. 
\end{lem}
%
%
\begin{proof} 
It is clearly enough to show the continuity of $(t,t')\mapsto h\left(t,(t+t')/2\strut\right)$ and, since 
\[
\left|h\left(t,(t+t')/2\strut\right)-h\left(u,(u+u')/2\strut\right)\right|\le h(t,u)+h\left((t+t')/2,(u+u')/2\strut\right),
\]
it is enough to bound the second term. By the classical inequalities between the Hellinger and variation distances,
\begin{eqnarray*}
h^2\left(\frac{t+t'}{2},\frac{u+u'}{2}\right)&\le&\frac{1}{2}\int\left|\frac{t+t'}{2}-\frac{u+u'}{2}\right|d\mu\\
&\le&\frac{1}{4}\int\left[|t-u|+|t'-u'|\right]d\mu\;\;\le\;\;\frac{1}{\sqrt{2}}[h(t,u)+h(t',u')],
\end{eqnarray*}
which concludes the proof.
\end{proof}
Together with (\ref{metric-h}) and (\ref{control-1}) the lemma shows that $(\theta,\theta')\mapsto 
\gT(\gX,\gt_{\theta},\gt_{\theta'})$ is continuous from $\Theta\times\Theta$ into $\R$ with probability 1,
uniformly with respect to $\gX$.

Recalling that $s=t_\vartheta\in\overline{S}$, let us set, for $\Gamma\ge1$, $J\in\N$ and $n\ge 1$,
\[
\delta=\sqrt{\Gamma d/n}\quad\mbox{and}\quad
\CC(\Gamma,J)=\ac{(\theta,\theta')\in \Theta^{2}\ \mbox{with}\ h(\vartheta,\theta)\le\delta\;\mbox{ and }\;h(\vartheta,\theta')>2^{J/2}\delta}.
\]
We want to establish the following intermediate result.
\begin{prop}\label{etape0001}
Under Assumption~\ref{MLE}-(i),(ii) and (iii), there exist a positive constant $C$ and a  positive integer $J_{0}$, both depending on $\overline{S}$ only, such that, for all $J\ge J_{0}$ and $\Gamma\in[1,n/d]$,
\[
\P_{\gs}\cro{\sup_{(\theta,\theta')\in\CC(\Gamma,J)}\gT(\gX,\gt_{\theta},\gt_{\theta'})< 0}\ge 
1-\exp[-C2^{J}\Gamma d].
\]
\end{prop}
%
\begin{proof}
First of all, let us note that $\sup_{(\theta,\theta')\in\CC(\Gamma,J)}\gT(\gX,\gt_{\theta},
\gt_{\theta'})$ is measurable since $\CC(\Gamma,J)$ is separable and $(\theta,\theta')\mapsto \gT(\gX,\gt_{\theta},\gt_{\theta'})$ is continuous. Let us then set $B=\ac{\theta\in \Theta\left|\,h(\vartheta,\theta)\le\delta\right.}$,
\[
C_{j}=\ac{\theta\in \Theta\left|\,2^{(J+j)/2}\delta\le h(\vartheta,\theta)<
2^{(J+j+1)/2}\delta\right.}\quad\mbox{for all }j\in\N
\]
and, for $k\in\N$, let $B_{k}\subset B$ and $C_{j,k}\subset C_j$ be $2^{-k/2}\delta$-nets for $B$ and $C_{j}$ respectively. Since, by~\eref{metric-h}, $(\overline S,h)$ and $(\Theta,\ab{\ })\subset (\R^{d},\ab{\ })$ are isometric (up to constants), we may choose $B_{k}$ and $C_{j,k}$ in such a way that 
\begin{equation}
\log\ab{B_{k}}\le A_5dk\qquad\mbox{and}\qquad\log\ab{C_{j,k}}\le A_5d(J+j+1+k)\quad\mbox{for all }k,j\in\N,
\label{card-res}
\end{equation}
as would be the case for Euclidean balls. 

For $\theta\in \Theta$ and $j,k\in\N$, we denote by $\theta_{k}$ and $\theta_{j,k}$ minimizers of the function $\theta'\mapsto h(\theta,\theta')$ over $B_{k}$ and $C_{j,k}$ respectively.
For all $j\in \N$ and $(\theta_{0},\theta'_{j,0})\in B_{0}\times C_{j,0}$, by Proposition~\ref{variance},
\[
\E\cro{\psi^{2}\pa{\sqrt{t_{\theta'_{j,0}}/t_{\theta_0}}}}\le6\cro{h^{2}\left(\vartheta,\theta_{0}\right)+h^{2}\left(\vartheta,\theta'_{j,0}\right)}\le6\delta^2\pa{1+2^{J+j+1}}\le2^{J+j+4}\delta^2.
\]
Since $\ab{\psi\pa{\sqrt{t_{\theta'_{j,0}}/t_{\theta_{0}}}}}\le1$, we may use Bernstein's inequality with $x_{j,0}=2^{J+j}\Gamma d/100$, then \eref{card-res} to derive that 
\begin{eqnarray*}
\lefteqn{\P_{\gs}\cro{\sup_{(\theta_{0},\theta_{j,0}')\in B_{0}\times C_{j,0}}\gT\pa{\gX,t_{\theta_{0}},t_{\theta'_{j,0}}}-\E\cro{\gT\pa{\gX,t_{\theta_{0}},t_{\theta'_{j,0}}}}>x_{j,0}}}\quad
\\&\le& \sum_{(\theta_0,\theta'_{j,0})\in B_{0}\times C_{j,0}}\P_{\gs}\cro{\gT
\pa{\gX,t_{\theta_{0}},t_{\theta'_{j,0}}}-\E\cro{\gT\pa{\gX,t_{\theta_{0}},t_{\theta'_{j,0}}}}>x_{j,0}}
\\&=& \sum_{(\theta_0,\theta'_{j,0})\in B_{0}\times C_{j,0}}\P_{\gs}\cro{\sum_{i=1}^{n}\left(
\psi\pa{\sqrt{{t_{\theta'_{j,0}}\over t_{\theta_{0}}}}(X_{i})}-\E\cro{\psi\pa{\sqrt{{t_{\theta'_{j,0}}
\over t_{\theta_{0}}}}(X_{i})}}\right)>\sqrt{2}\,x_{j,0}}\\&\le& 
\exp\cro{A_5d\pa{J+j+1}-{2x_{j,0}^{2}\over
2\left(2^{J+j+4}\Gamma d+\sqrt{2}\,x_{j,0}/3\right)}}\\&\le& \exp\cro{-C2^{J+j+1}\Gamma d}
\;\;\le\;\; \exp\cro{-(j+1)-C2^{J}\Gamma d}
\end{eqnarray*}
for some $C>0$ and $J_0$ large enough (depending on the $A_i$, which means on $\overline{S}$) since $J\ge J_0$. 

For $(\theta,\theta')\in B\times C_{j}$ and $k\in\N$, let
\begin{eqnarray*}
\lefteqn{\Delta\gT\pa{\gX,\gt_{\theta_{k}},\gt_{\theta'_{j,k}},\gt_{\theta_{k+1}},\gt_{\theta'_{j,k+1}}}}
\hspace{28mm}\\&=&\ac{\gT\pa{\gX,\gt_{\theta_{k+1}},\gt_{\theta'_{j,k+1}}}-
\E\cro{\gT\pa{\gX,\gt_{\theta_{k+1}},\gt_{\theta'_{j,k+1}}}}}\\&&\mbox{}-
\ac{\gT\pa{\gX,\gt_{\theta_{k}},\gt_{\theta'_{j,k}}}-\E\cro{\gT\pa{\gX,\gt_{\theta_{k}},\gt_{\theta'_{j,k}}}}}\\
&=&\frac{1}{\sqrt{2}}\sum_{i=1}^{n}\left[\psi\pa{\sqrt{t_{\theta'_{j,k+1}}\over t_{\theta_{k+1}}}(X_{i})}
-\psi\pa{\sqrt{t_{\theta'_{j,k}}\over t_{\theta_{k}}}(X_{i})}\right]\\&&\mbox{}- \frac{1}{\sqrt{2}}
\sum_{i=1}^{n}\E\cro{\psi\pa{\sqrt{t_{\theta'_{j,k+1}}
\over t_{\theta_{k+1}}}(X_{i})}-\psi\pa{\sqrt{t_{\theta'_{j,k}}\over t_{\theta_{k}}}(X_{i})}}.
\end{eqnarray*}
It follows from \eref{control-1} and \eref{metric-h} that
\begin{eqnarray*}
\norm{\psi\pa{\sqrt{t_{\theta'_{j,k+1}}\over t_{\theta_{k+1}}}}-\psi\pa{\sqrt{t_{\theta'_{j,k}}
\over t_{\theta_{k}}}}}_{\infty}&\le& A_6\cro{h\left(\theta'_{j,k+1},\theta'_{j,k}\right)+h\left(\theta_{k+1},\theta_{k}\right)}\\&\le&A_62^{1-k/2}\delta\pa{2^{-1/2}+1}\;\;<\;\;7A_62^{-k/2-1}\delta,
\end{eqnarray*}
therefore,
\[
\E\left[\rule{0mm}{9mm}\left[\psi\pa{\sqrt{t_{\theta'_{j,k+1}}\over t_{\theta_{k+1}}}}
-\psi\pa{\sqrt{t_{\theta'_{j,k}}\over t_{\theta_{k}}}}\right]^2\right]<
A_72^{-k}\delta^2.
\]
For $x_{j,k}=(k+1)2^{-k/2+J+j}\Gamma d/100$ and $\Gamma d\le n$, we deduce from Bernstein's inequality and~\eref{card-res} that 
\begin{eqnarray*}
\lefteqn{\P_{\gs}\cro{\sup_{\left.
\begin{array}{c}
\scriptstyle(\theta_{k},\theta_{k+1})\in B_{k}\times B_{k+1}\\
\scriptstyle(\theta_{j,k}',\theta_{j,k+1}')\in B_{j,k}\times B_{j,k+1}\\
\end{array}\right.}
\Delta\gT\pa{\gX,\gt_{\theta_{k}},\gt_{\theta'_{j,k}},\gt_{\theta_{k+1}},\gt_{\theta'_{j,k+1}}}>x_{j,k}}}\hspace{20mm}\\
&\le& \exp\cro{2A_5d\pa{J+j+2k+2}-{2x_{j,k}^{2}\over 2\pa{A_72^{-k}\Gamma d+(7/6)x_{j,k}A_62^{-k/2}\delta}}}\\
&\le& \exp\cro{-C\pa{(k+1)2^{J+j+1}\Gamma d}}\;\;\le\;\;\exp\cro{-(k+1)-(j+1)-C2^{J}\Gamma d},
\end{eqnarray*}
for some $C>0$ and $J_0$ large enough (depending on $\overline{S}$). 

Putting all these bounds together, we get, for $J\ge J_0$ large enough and with probability at least 
\[
1-e^{-C2^{J}\Gamma d}\pa{\sum_{j\ge 1}e^{-j}+\sum_{j\ge 1}e^{-j}\sum_{k\ge 1}e^{-k}}\ge 1-e^{-C2^{J}\Gamma d},
\] 
for some $C>0$, that for all $j\in\N$, $\theta\in B$ and $\theta'\in C_{j}$, 
\begin{eqnarray*}
\gT(\gX,\gt_{\theta},\gt_{\theta'})&=&\E\cro{\gT(\gX,\gt_{\theta},\gt_{\theta'})}+\gT(\gX,\gt_{\theta},\gt_{\theta'})-\E\cro{\gT(\gX,\gt_{\theta},\gt_{\theta'})}\\&=& 
\E\cro{\gT(\gX,\gt_{\theta},\gt_{\theta'})}+\lim_{k\rightarrow+\infty}
\left\{\gT\pa{\gX,\gt_{\theta_{k}},\gt_{\theta'_{j,k}}}-
\E\cro{\gT\pa{\gX,\gt_{\theta_{k}},\gt_{\theta'_{j,k}}}}\right\}\\&=&
\E\cro{\gT(\gX,\gt_{\theta},\gt_{\theta'})}+\gT\pa{\gX,\gt_{\theta_{0}},\gt_{\theta'_{j,0}}}-
\E\cro{\gT\pa{\gX,\gt_{\theta_{0}},\gt_{\theta'_{j,0}}}}\\&&\mbox{}+\sum_{k\in\N}
\Delta\gT\pa{\gX,\gt_{\theta_{k}},\gt_{\theta'_{j,k}},\gt_{\theta_{k+1}},\gt_{\theta'_{j,k+1}}}\\
&\le&\E\cro{\gT(\gX,\gt_{\theta},\gt_{\theta'})}+x_{j,0}+\sum_{k\in\N}x_{j,k}\\&\le& 
\E\cro{\gT(\gX,\gt_{\theta},\gt_{\theta'})}+{2^{J+j}\Gamma d\over 100}\pa{1+\sum_{k\ge 0}(k+1)2^{-k/2}}.
\end{eqnarray*}
Finally, with probability at least $1-e^{-C2^{J}\Gamma d}$, for all $(\theta,\theta')\in B\times C_{j}$ and $j\in\N$,
\begin{equation}\label{eq-step}
\gT(\gX,\gt_{\theta},\gt_{\theta'})<\E\cro{\gT(\gX,\gt_{\theta},\gt_{\theta'})}+0.13\left(2^{J+j}n\delta^2\right).
\end{equation}
We conclude by using (\ref{rhoM-approx}) which implies that, if $(\theta,\theta')\in B\times C_{j}$,
\begin{eqnarray*}
n^{-1}\E\cro{\gT(\gX,t_{\theta},t_{\theta'})}&=&\varrho(s,t_{\theta'},t_{\theta})-
\varrho(s,t_{\theta},t_{\theta'})\;\;\le\;\; \varrho(s,t_{\theta'},t_{\theta})-\rho(s,t_{\theta})\\
&\le&\varrho(s,t_{\theta'},t_{\theta})-\rho(s,t_{\theta'})+\rho(s,t_{\theta'})-\rho(s,t_{\theta})\\
&\le&{1\over \sqrt{2}}\cro{h^{2}(\vartheta,\theta')+h^{2}(\vartheta,\theta)}+
h^{2}(\vartheta,\theta)-h^{2}(\vartheta,\theta')\\&=& 
-\cro{\pa{1-{1\over \sqrt{2}}}h^{2}(\vartheta,\theta')-\pa{1+{1\over\sqrt{2}}}h^{2}(\vartheta,\theta)}\\
&\le&-\cro{2^{J+j}\pa{1-{1\over \sqrt{2}}}-\pa{1+{1\over\sqrt{2}}}}\delta^2\;\;<\;\; 
-0.255\times2^{J+j}\delta^2
\end{eqnarray*}
provided that $J_{0}$ is large enough since $J\ge J_{0}$.
\end{proof}

Let us now proceed with the proof of Theorem~\ref{thm-MLE}. By Assumption~\ref{MLE}-$(iv)$, 
the MLE $\widetilde\theta_n$ converges towards the true parameter $\vartheta$ and, since the model is regular, it converges at rate $1/\sqrt{n}$ by Corollary 5.53 of van der Vaart~\citeyearpar{MR1652247}. 
Therefore, given $\varepsilon>0$, for $\Gamma$ large enough depending on $\varepsilon$, 
$h\pa{\vartheta,\widetilde\theta_{n}}\le\delta=\sqrt{\Gamma d/n}$ with probability larger than 
$1-\varepsilon/2$. We may now apply Proposition~\ref{etape0001} with this particular value of 
$\Gamma$, provided that $n\ge\Gamma d$. It follows that, for a suitable choice of $J\ge J_0$,
\[
\P_{\gs}\cro{\sup_{(\theta,\theta')\in\CC(\Gamma,J)}\gT(\gX,\gt_{\theta},\gt_{\theta'})< 0}\ge 
1-\exp[-C2^{J}\Gamma d]\ge1-\varepsilon/2.
\]
Therefore
\[
\P_{\gs}\cro{\sup_{\theta'\in\mathcal B^c}\gT\pa{\gX,\gt_{\widetilde\theta_{n}},\gt_{\theta'}}<0}
\ge 1-\varepsilon\;\;\mbox{ with }\;\;
\mathcal B=\left\{\theta'\in\Theta\;\mbox{ such that }h\pa{\vartheta,\theta'}\le2^{J/2}\delta\right\}.
\]
From now on, we shall work on the event of probability larger than $1-\varepsilon$ on which 
\begin{equation}\label{eq88}
h\pa{\vartheta,\widetilde\theta_{n}}\le\delta\qquad\mbox{and}\qquad
\sup_{\theta'\in\mathcal B^c}\gT\pa{\gX,\gt_{\widetilde\theta_{n}},\gt_{\theta'}}<0.
\end{equation}
It remains to evaluate 
$\sup_{\theta'\in\mathcal B}\gT\pa{\gX,\gt_{\widetilde\theta_{n}},\gt_{\theta'}}$ on this event.
For all $\theta'\in\mathcal B$, $h\pa{\widetilde\theta_{n},\theta'}\le2^{(J+2)/2}\delta$.
Moreover, using the inequalities
\[
0\le \sqrt{{a+b\over 2}}-\frac{\sqrt{a}+\sqrt{b}}{2}\le {\pa{\sqrt{b}-\sqrt{a}}^{2}\over 4\sqrt{a}}
\quad\mbox{for all }a,b>0,
\]
which both derive from $2xy\le x^2+y^2$, we get
\begin{eqnarray}
\lefteqn{\rho\pa{t_{\theta'},{t_{\widetilde\theta_{n}}+t_{\theta'}\over 2}}-\rho\pa{t_{\widetilde\theta_{n}},{t_{\widetilde\theta_{n}}+t_{\theta'}\over 2}}}\hspace{40mm}
\nonumber\\&=&\int \pa{\sqrt{t_{\theta'}}-\sqrt{t_{\widetilde\theta_{n}}}}\sqrt{{t_{\widetilde\theta_{n}}+t_{\theta'}\over 2}}\,d\mu\nonumber\\&=&
\int \pa{\sqrt{t_{\theta'}}-\sqrt{t_{\widetilde\theta_{n}}}}\cro{\sqrt{{t_{\widetilde\theta_{n}}+t_{\theta'}\over 2}}-{\sqrt{t_{\theta'}}+\sqrt{t_{\widetilde\theta_{n}}}\over 2}}d\mu
\nonumber\\&\le& {1\over4}
\int\ab{\sqrt{t_{\theta'}\over t_{\theta'}}-\sqrt{t_{\widetilde\theta_{n}}\over t_{\theta'}}}\times \pa{\sqrt{t_{\theta'}}-\sqrt{t_{\widetilde\theta_{n}}}}^{2}d\mu\nonumber\\&\le&{1\over2}
\norm{\sqrt{t_{\theta'}\over t_{\theta'}}-\sqrt{t_{\widetilde\theta_{n}}\over t_{\theta'}}}_\infty
h^2\left(\widetilde\theta_{n},\theta'\right)\;\;\le\;\;\frac{A_4}{2}h^3\left(\widetilde\theta_{n},\theta'\right)
\label{Eq-MLE1}
\end{eqnarray}
by (\ref{Eq-control-0}). Besides, when $u$ converges to 0, $\psi(1+u)=\left(1/\sqrt{2}\right)\log(1+u)+O(u^{3})$.
Setting $u=\sqrt{t_{\theta'}/t_{\widetilde\theta_{n}}}-1$ so that by (\ref{Eq-control-0}) $|u|\le A_4
h\pa{\widetilde\theta_{n},\theta'}$, then using the fact that $\widetilde \theta_{n}$ maximizes the likelihood,  we derive that 
\[
2\sqrt{2}\sum_{i=1}^{n}\psi\!\pa{\sqrt{t_{\theta'}\over t_{\widetilde\theta_{n}}}(X_{i})}\!\le
\sum_{i=1}^{n}\log t_{\theta'}(X_{i})-\sum_{i=1}^{n}\log t_{\widetilde\theta_{n}}(X_{i})+
A_{8}nh^3\!\pa{\widetilde\theta_{n},\theta'}\!\le A_{8}nh^3\!\pa{\widetilde\theta_{n},\theta'}\!.
\]
Together with (\ref{eq88}) and (\ref{Eq-MLE1}) this shows that, with probability larger than $1-\varepsilon$,
\[
\sup_{\theta'\in\mathcal B}\gT\pa{\gX,\gt_{\widetilde\theta_{n}},\gt_{\theta'}}\le
\frac{A_4+A_{8}}{4}nh^3\pa{\widetilde\theta_{n},\theta'},
\]
hence
\begin{equation}\label{eq-topo}
\sup_{\theta'\in\Theta}\gT\pa{\gX,\gt_{\widetilde\theta_{n}},\gt_{\theta'}}
<A_{9}\left(2^J\Gamma d\right)^{3/2}n^{-1/2}.
\end{equation}
Since the mapping $(\theta,\theta')\mapsto \gT\pa{\gX,t_{\theta},t_{\theta'}}$ is uniformly continuous on 
$\Theta\times \Theta$, 
\[
{\mathcal E}=\ac{t_{\theta}\in \overline S,\ \quad\sup_{\theta'\in\Theta}\gT\pa{\gX,\gt_{\theta},\gt_{\theta'}}<A_{9}\left(2^J\Gamma d\right)^{3/2}n^{-1/2}}
\]
is an open subset of $\overline S$ hence $S\cap{\mathcal E}$ is also dense in ${\mathcal E}$.
Besides,  $S\cap{\mathcal E}\subset \EE(\gX, S)$ for $n$ large enough. Then, using~\eref{eq-topo} 
we get with probability at least $1-\eps$
\[
\gt_{\widetilde\theta_{n}}\in{\mathcal E}={\mathcal E}\cap{\rm Cl}(S\cap{\mathcal E})\subset {\rm Cl}(S\cap{\mathcal E})\subset {\rm Cl}(\EE(\gX, S)),
\]
showing that $\gt_{\widetilde\theta_{n}}$ is a $\rho$-estimator.

\subsection{Proof of Theorem~\ref{T-crochet}}\label{P15}
Inequality~\eref{eq-BR-rho-MLE} is obtained by combining~\eref{eq-Krobuste} (Assumption~\ref{A-MLE}-$(i)$ corresponds to~\eref{hypo-KL} in the density context) and~\eref{borneD-S}. Consequently, it suffices to prove~\eref{borneD-S} and to do so  we may assume with no loss of generality that $s=\overline s\in S$, which we shall do in the remaining part of this proof.

Let us consider the symmetric family $\FF=\FF(y)$ defined for $y=\sigma\sqrt{n}>0$ by 
\[
\FF=\FF(y)=\ac{\left.\psi(\sqrt{t/s})\,\right|\,\gt\in \B^{S}(\gs,y)}\bigcup \ac{\left.-\psi(\sqrt{t/s})\,\right|\,\gt \in \B^{S}( \gs,y)}.
\]
For all $f\in\FF$, $\ab{f}\le 1$ and it follows from Proposition~\ref{variance} that for all integers $k\ge 2$, $\E\cro{|f(X_{1})|^{k}}\le \E\cro{f^{2}(X_{1})}\le(6\sigma^{2})\wedge1$. Since $\psi$ is increasing and Lipschitz with Lipschitz constant $L<\sqrt{3}$, it follows from Assumption~\ref{A-MLE}-$(ii),(iii)$ that the family of pairs $\I^{\psi}(s,\sigma,\epsilon)$ given by
\[
\ac{\pa{\psi(\sqrt{t_{L}/s}),\psi(\sqrt{t_{U}/s})},\pa{-\psi(\sqrt{t_{U}/s}),-\psi(\sqrt{t_{L}/s})}, \ (t_{L},t_{U})\in \I\left(s,\sigma,L^{-1}\epsilon/\sqrt{2}\right)}
\]
covers $\FF$ with at most $2\exp\cro{\HH_{[\ ]}^{S}\left(s,\sigma,L^{-1}\epsilon/\sqrt{2}\right)}\le 
\exp\cro{2\HH_{[\ ]}^{S}\left(s,\sigma,L^{-1}\epsilon/\sqrt{2}\right)}$ brackets and that for all integers $k\ge 2$
\begin{eqnarray*}
\lefteqn{\E\cro{\pa{\psi(\sqrt{t_{U}/s})(X_{1})-\psi(\sqrt{t_{L}/s})(X_{1})}^{k}}}\hspace{30mm}\\
&\le & 2^{k-2}\E\cro{\pa{\psi(\sqrt{t_{U}/s})(X_{1})-\psi(\sqrt{t_{L}/s})(X_{1})}^{2}}\\
&\le& 2^{k-2}L^{2}\int_{\X}\pa{\sqrt{t_{U}}-\sqrt{t_{L}}}^{2}d\mu\;\;\le\;\;\epsilon^{2}\times 2^{k-2}
\;\;\le\;\;\frac{k!}{2}\epsilon^{2}.
\end{eqnarray*}
Note that
\[
\w^{S}(\gs,\gs,y)\le\E\cro{\sup_{f\in\FF}\ab{\sum_{i=1}^{n}\pa{f(X_{i})-\E\cro{f(X_{i})}}}}=\E\cro{\sup_{f\in\FF}\pa{\sum_{i=1}^{n}\pa{f(X_{i})-\E\cro{f(X_{i})}}}}.
\]
We may therefore apply to this last expectation the bound (6.25) of Theorem~6.8 in Massart~\citeyearpar{MR2319879} with $\sigma^2$ replaced by $(6\sigma^{2})\wedge1$, $b=1$, 
$\delta=\epsilon$, $H(\delta)=2\HH_{[\ ]}^{S}\left(s,\sigma,L^{-1}\delta/\sqrt{2}\right)$, 
$\eps=L/\sqrt{3}\in (0,1]$ and $A=\Omega$. It leads to
\begin{eqnarray*}
\w^{S}(\gs,\gs,y)&\le& {27L^{-1}\sqrt{6n}}\int_{0}^{\sigma L\sqrt{2}}
\sqrt{\HH_{[\ ]}^{S}\left(s,\sigma,L^{-1}\epsilon/\sqrt{2}\right)}\,d\epsilon+
8\HH_{[\ ]}^{S}\left(s,\sigma,\sigma\sqrt{3}/L\right)\\&=& 54\sqrt{3n}\int_{0}^{\sigma}
\sqrt{\HH_{[\ ]}^{S}(s,\sigma,z)}\,dz+ 8\HH_{[\ ]}^{S}\left(s,\sigma,\sigma\sqrt{3}/L\right).
\end{eqnarray*}
Since $z\mapsto\HH_{[\ ]}^{S}(s,\sigma,z)$ is non-increasing, $\HH_{[\ ]}^{S}
\left(s,\sigma,\sigma\sqrt{3}/L\right)\le\HH_{[\ ]}^{S}(s,\sigma,\sigma)\le \sigma^{-2}\phi^{2}(\sigma)$. 
Let us now choose some $\lambda_0>1$. It follows from the definition of $\tau_{n}$ and the monotonicity 
of $\sigma\mapsto\phi(\sigma)/\sigma$ that for all 
$\lambda'\in ]1,\lambda_{0}]$ and $\sigma\ge \lambda_{0} \tau_{n}$ 
\[
{\phi(\sigma)\over \sigma}\le {\phi(\lambda'\tau_{n})\over \lambda' \tau_{n}}\le \lambda'\tau_{n}\sqrt{n}\le {\lambda'\over \lambda_{0}}\sigma\sqrt{n}.
\]
Letting $\lambda'$ tend to 1 we get $\phi(\sigma)\le \sigma^{2}\sqrt{n}/\lambda_{0}$. Putting these bounds 
together we get that, for all $y=\sqrt{n}\sigma\ge \lambda_{0}\sqrt{n}\tau_{n}$ with $\lambda_0=2555$,
\[
\w^{S}(\gs,\gs,y)\le 54\sqrt{3n}\phi(\sigma)+8\sigma^{-2}\phi^{2}(\sigma)\le 
\pa{{54\sqrt{3}\over \lambda_{0}}+{8\over \lambda_{0}^{2}}}n\sigma^{2}\le c_0y^{2}.
\]
Finally, $\sup_{\gs\in S}D^{S}(\gs,\gs)\le (\lambda_{0}^{2}n\tau_{n}^{2})\vee 1$.

\subsection{Proof of Proposition~\ref{P-dimpara}\label{P4}}
If $s\ne\overline s$, let $J\in\N$ be such that $h^{2}(s,\overline s)=2^{-J}$ and $\Omega_{J}(\gX)=
\{\omega\in\Omega\,|\,\ell(X_{i})\le J\mbox{  for }i=1,\ldots,n\}$. Since $s\in \overline S$, there exists 
$\theta^{\star}\in \Theta$ such that $s=s_{\theta^{\star}}$ and for $y\ge 1$, let us us set 
\[
\Theta[\theta^{\star},y]=\ac{\theta\in \Theta\ \telque\ \gs_{\theta}\in\B^{S}(\gs,\overline \gs,y)}.
\]
We decompose $\w^{S}(\gs,\overline \gs,y)$ as 
\[
\w^{S}(\gs,\overline \gs,y)=\E\cro{\sup_{\gt\in\B^{S}(\gs,\overline \gs,y)}\ab{\gZ(\gX,\overline \gs,\gt)}}
=\E\cro{\sup_{\theta\in \Theta[\theta^{\star},y]}\ab{\gZ(\gX,\overline \gs,\gs_{\theta})}}=E_1+E_2
\]
with 
\[
E_1=\E\cro{\sup_{\theta\in \Theta[\theta^{\star},y]}\ab{\gZ(\gX,\overline \gs,\gs_{\theta})}
\1_{\Omega_{J}(\gX)}}\quad\mbox{and}\quad
E_2=\E\cro{\sup_{\theta\in \Theta[\theta^{\star},y]}\ab{\gZ(\gX,\overline \gs,\gs_{\theta})}
\1_{\left(\Omega_{J}(\gX)\right)^{c}}}.
\]
Let us bound each of these last two terms from above. 
On the event $\Omega_{J}(\gX)$, $s(X_{i})=\overline s(X_{i})$ for all $i=1,\ldots,n$, hence
\[
E_1\le\E\cro{\sup_{\theta\in \Theta[\theta^{\star},y]}\ab{\sum_{i=1}^{n}\pa{\psi\pa{\sqrt{s_{\theta}(X_{i})\over s(X_{i})}}-\E\cro{\psi\pa{\sqrt{s_{\theta}(X_{i})\over s(X_{i})}}}}}}.
\]
For $\theta\in\Theta$ and $i=1,\ldots,n$, either $\ell(X_{i})\le J(\theta,\theta^{\star})$ and
$s_{\theta}(X_{i})=s(X_{i})=2^{-\ell(X_{i})}$ in which case,  $\psi\pa{\sqrt{s_{\theta}(X_{i})/s(X_{i})}}=\psi(1)=0$, or $\ell(X_{i})>J(\theta,\theta^{\star})$, $s_{\theta}(X_{i})=0$ and then $\psi\pa{\sqrt{s_{\theta}(X_{i})/s(X_{i})}}=\psi(0)=-1$. In both cases $\psi\pa{\sqrt{s_{\theta}(X_{i})/s(X_{i})}}=-\1_{\ell(X_{i})>J(\theta,\theta^{\star})}$. Let us now introduce $n$ Rademacher random variables $\eps_{1},\ldots,\eps_{n}$, independent of the $X_{i}$. By a symmetrization argument,
\[
E_1\le2\E\cro{\sup_{\theta\in \Theta[\theta^{\star},y]}\ab{\sum_{i=1}^{n}\eps_{i}\psi\pa{\sqrt{s_{\theta}(X_{i})\over s(X_{i})}}}}=2\E\cro{\sup_{\theta\in \Theta[\theta^{\star},y]}\ab{\sum_{i=1}^{n}\eps_{i}\1_{\ell(X_{i})>J(\theta,\theta^{\star})}}}.
\]

Let us now work conditionally on $X_{1},\ldots,X_{n}$ and denote by $\mathbb{E}_{\eps}$ the 
corresponding conditional expectation. Up to a re-ordering of the $\eps_{i}$, we may assume with 
no loss of generality that $\ell(X_{1})\ge \ell(X_{2})\ge \ldots\ge \ell(X_{n})$. Then 
$\sum_{i=1}^{n}\eps_{i}\1_{\ell(X_{i})>J(\theta,\theta^{\star})}$ is necessarily of the form 
$\sum_{i=1}^{k}\eps_{i}$ for some non-negative integer $k=k(\theta,\theta^{\star},\gX)$ corresponding 
to the number of $X_{i}$ of length $\ell(X_{i})$ larger than $J(\theta,\theta^{\star})$. By~\eref{f-h}, for 
$\theta\in \Theta[\theta^{\star},y]$, $J(\theta,\theta^{\star})\ge\log_{2}(n/y^{2})$, hence $k$ cannot 
exceed the number $\widehat N$ of $X_{i}$ of length not smaller than $\log_{2}(n/y^{2})$. 
We deduce that 
\[
\sup_{\theta\in \Theta[\theta^{\star},y]}\ab{\sum_{i=1}^{n}\eps_{i}\1_{\ell(X_{i})>J(\theta,\theta^{\star})}}
\le\max_{0\le k\le \widehat N}\ab{\sum_{i=1}^{k}\eps_{i}}
\]
with the convention $\sum_{i=1}^{0}=0$. Taking the expectation (conditionnaly on $\gX$) and using 
Doob's maximal inequality we get
\begin{eqnarray*}
\lefteqn{\mathbb{E}_{\eps}\cro{\sup_{\theta\in \Theta[\theta^{\star},y]}\ab{\sum_{i=1}^{n}
\eps_{i}\1_{\ell(X_{i})>J(\theta,\theta^{\star})}}\,}}\hspace{40mm}\\&\le&\mathbb{E}_{\eps}
\cro{\max_{0\le k\le \widehat N}\ab{\sum_{i=1}^{k}\eps_{i}}\,}\;\;\le\;\;\pa{\mathbb{E}_{\eps}
\cro{\max_{0\le k\le \widehat N}\ab{\sum_{i=1}^{k}\eps_{i}}^{2}}}^{1/2}\\&\le& 
2\pa{\mathbb{E}_{\eps}\cro{\ab{\sum_{i=1}^{\widehat N}\eps_{i}}^{2}}}^{1/2}\;\;=\;\;2\sqrt{\widehat N}.
\end{eqnarray*}
Taking the expectation with respect to $X_{1},\ldots,X_{n}$ finally leads to
\begin{eqnarray*}
E_1&\le&4\E\cro{\sqrt{\widehat N}}\;\;\le\;\;4\sqrt{\E\pa{\sum_{i=1}^{n}
\1_{\ell(X_{i})>\log_{2}(n/y^{2})}}}\\
&=&4\sqrt{n\P_{\gs}\cro{\ell(X_{1})>\log_{2}\left(\frac{n}{y^2}\right)}}\;\;=\;\;
4\cro{n\sum_{j>\log_{2}(n/y^{2})}2^{-j}}^{1/2}\;\;\le\;\;4\sqrt{2y^{2}}.
\end{eqnarray*}

As to $E_2$, since $\psi$ is bounded by one and $y\ge 1$, it satisfies
\[
E_2\le2n\P_{\gs}\cro{\Omega_{J}(\gX)^{c}}\le 2n^{2}\P_{\gs}\cro{\ell(X_{1})>J}= 2n^{2}2^{-J}\le 2n^{2}2^{-J}y= 2n^{2}h^{2}(s,\overline s)y.
\]
Putting these bounds together, we get 
\begin{equation}
\w^{S}(\gs,\overline \gs,y)=E_1+E_2\le 2\left(2\sqrt{2}+n^2h^2(s,\overline s)\right)y\quad\mbox{for all }y\ge1,
\label{Eq-aux7}
\end{equation}
which leads to the bound on $D^{S}(\gs,\overline \gs)$.

If $s=\overline{s}$, we proceed in the same way with $J=+\infty$ which means that $\Omega_{J}(\gX)^{c}=\varnothing$, $E_2=0$ and (\ref{Eq-aux7}) remains valid.

\subsection{Proof of Proposition~\ref{approx-q}}\label{P-prop10}
Note that the set of densities $\{p^\beta,\beta>0\}$ is a regular statistical model with respect to  the parameter $\beta$. Following Theorem 2.1 p. 121 (equation~2.9) and Section~5 p.133 of the book by Ibragimov and Has{'}minski{\u\i}~\citeyearpar{MR620321}, for $\beta'>\beta>0$
\begin{equation}
h^2\left(p^\beta,p^{\beta'}\right)\le
\left[\sup_{\beta\le b\le\beta'}\frac{I(b)(\beta'-\beta)^2}{8}\right]\bigwedge1,
\label{Eq-Fisher}
\end{equation}
where $I$ denotes the Fisher Information of this parametric model which is given by
\[
I(b)=\int_\R\frac{\left[\dot{p^b}(x)\right]^2}{p^b(x)}dx=
2\int_0^\infty\frac{\left[\dot{p^b}(x)\right]^2}{p^b(x)}dx
\]
and $\dot{p^b}(x)$ is the derivative of $p^b(x)$ with respect to $b$. It follows from
(\ref{Eq-Gamma1}) that
\[
\dot{p^b}(x)=p^b(x)\left[-\frac{1}{b}-\frac{\Gamma'(b)}{\Gamma(b)}+
\frac{1}{b^2}x^{1/b}\log x\right]\quad\mbox{for }x>0,
\]
hence
\[
\frac{\left[\dot{p^b}(x)\right]^2}{p^b(x)}\le3p^b(x)\left[\frac{1}{b^2}+
\left(\frac{\Gamma'(b)}{\Gamma(b)}\right)^2+\frac{1}{b^4}x^{2/b}(\log x)^2\right]
\]
and
\[
I(b)\le3\left[\frac{1}{b^2}+\left(\frac{\Gamma'(b)}{\Gamma(b)}\right)^2+
\frac{1}{b^5\Gamma(b)}\int_0^\infty e^{-x^{1/b}}x^{2/b}(\log x)^2dx\right].
\]
Using again a change of variables we get 
\[
\int_0^\infty e^{-x^{1/b}}x^{2/b}(\log x)^2dx=b^3\int_0^\infty e^{-u}u^{b+1}(\log u)^2du=
b^3\Gamma''(b+2),
\]
so that finally,
\[
I(b)\le3J\quad\mbox{with}\quad J\le\frac{1}{b^2}+\left(\frac{\Gamma'(b)}{\Gamma(b)}\right)^2+
\frac{\Gamma''(b+2)}{b^2\Gamma(b)}=\frac{1}{b^2}+\left(\frac{\Gamma'(b)}{\Gamma(b)}\right)^2
+\frac{(b+1)\Gamma''(b+2)}{b\Gamma(b+2)}.
\]
Binet's formula for $\log\Gamma$ (see Whittaker and Watson~\citeyearpar{MR1424469}  page 251)  tells us that
\begin{equation}
\frac{\Gamma'(b)}{\Gamma(b)}=\log b-\frac{1}{2b}-2k(b)\quad\mbox{with}\quad 
k(b)=\int_0^\infty\frac{x}{\left(x^2+b^2\right)\left(e^{2\pi x}-1\right)}\,dx
\label{Eq-gamma1}
\end{equation}
hence
\[
\frac{\Gamma''(b)}{\Gamma(b)}-\left(\frac{\Gamma'(b)}{\Gamma(b)}\right)^2=\frac{1}{b}
+\frac{1}{2b^2}+4b\int_0^\infty\frac{x}{\left(x^2+b^2\right)^2\left(e^{2\pi x}-1\right)}\,dx.
\]
One should then observe that, since $e^u\ge1+u$,
\begin{equation}
0\le2k(b)\le\frac{1}{\pi}\int_0^\infty\frac{dx}{\left(x^2+b^2\right)}=\frac{1}{2b}
\label{Eq-gamma2}
\end{equation}
and
\[
0\le\int_0^\infty\frac{x}{\left(x^2+b^2\right)^2\left(e^{2\pi x}-1\right)}\,dx\le\frac{1}{2\pi}\int_0^\infty
\frac{dx}{\left(x^2+b^2\right)^2}=\frac{1}{2\pi b^3}\int_0^\infty\frac{dx}{\left(x^2+1\right)^2}=
\frac{1}{8b^3}.
\]
It follows that
\begin{equation}
J\le\frac{1}{b^2}+\left(\frac{\Gamma'(b)}{\Gamma(b)}\right)^2+\frac{b+1}{b}
\left[\left(\frac{\Gamma'(b+2)}{\Gamma(b+2)}\right)^2+\frac{1}{b+2}+\frac{1}{(b+2)^2}\right].
\label{Eq-gamma3}
\end{equation}
Moreover, since $\log\Gamma$ is a strictly convex function (Whittaker and Watson~\citeyearpar{MR1424469}  page 250) with a minimum value at $b_0\in(1,2)$, by 
(\ref{Eq-gamma1}) and (\ref{Eq-gamma2}), $0<\Gamma'(b)/\Gamma(b)<\log b-(2b)^{-1}$ 
for $b>b_0$ and $0<-\Gamma'(b)/\Gamma(b)<b^{-1}-\log b$ for $b<b_0$. It follows that
\[
\left|\frac{\Gamma'(b)}{\Gamma(b)}\right|\le\left\{\begin{array}{ll}1.37b^{-1}&\quad\mbox{if }0<b\le1;
\\1&\quad\mbox{if }1<b\le3;\\ \log b-(2b)^{-1}&\quad\mbox{if }b>3.
\end{array}\right.
\]
Therefore, by (\ref{Eq-gamma3}),
\[
J\le\frac{2.88}{b^2}+\frac{b(b+1)}{b^2}\left[1+\frac{1}{b+2}+\frac{1}{(b+2)^2}\right]
\le\frac{52}{9b^2}\quad\mbox{for }b\le1;
\]
\[
J\le\frac{1}{b^2}+1+\frac{b+1}{b}\left[\left(\log(b+2)-\frac{1}{2(b+2)}\right)^2+\frac{1}{b+2}
+\frac{1}{(b+2)^2}\right]\le4.63\quad\mbox{for }1<b\le3
\]
and
\[
J\le\frac{1}{b^2}+(\log b)^2+\frac{b+1}{b}\left[\left[\log(b+2)\right]^2+\frac{1}{b+2}+\frac{1}{(b+2)^2}\right]
\le4.23(\log b)^2\quad\mbox{for }b>3.
\]
Finally
\[
\frac{I(b)}{8}\le\left\{\begin{array}{ll}13/(6b^2)&\quad\mbox{if }0<b\le1
\\ 7/4&\quad\mbox{if }1<b\le3\\(1.3\log b)^2&\quad\mbox{if }b>3
\end{array}\right.
\]
and our first bound then follows from (\ref{Eq-Fisher}).

Let us now turn to the second inequality. 
\[
h^{2}\pa{p^{\beta},p^0}=1-\int_{-1}^1\sqrt{p^\beta(x)/2}\,dx=
1-\frac{1}{\sqrt{\beta\Gamma(\beta)}}\int_0^1\exp\left[-{x^{1/\beta}\over 2}\right]dx.
\]
Since $\beta\Gamma(\beta)=\Gamma(\beta+1)\le1$ for $0<\beta\le1$ and
\[
\int_0^1\exp\left[-{x^{1/\beta}\over 2}\right]dx\ge1-\int_0^1{x^{1/\beta}\over 2}dx=
1-\frac{\beta}{2(\beta+1)},
\]
(\ref{Eq-gamma7}) follows.

To control $\w_{p^\beta}$ we observe that, for $\beta>0$,
\[
h^{2}\pa{p^{\beta} ,{1\over \lambda}p^{\beta} \pa{\cdot \over \lambda}}=1-\frac{1}{\beta\Gamma(\beta)\sqrt{\lambda}}\int_{0}^{+\infty}e^{-(1/2){x^{1/\beta}}\pa{1+\lambda^{-1/\beta}}}dx.
\]
Using the change of variables $z=x\pa{(1+\lambda^{-1/\beta})/2}^{\beta}$, and the assumption 
$\lambda\in[1,2]$, we get
\begin{eqnarray*}
h^{2}\pa{p^{\beta} ,{1\over \lambda}p^{\beta} \pa{\cdot \over \lambda}}&=&1-\frac{1}{\beta\Gamma(\beta)\sqrt{\lambda}}\int_{0}^{+\infty}2^{\beta}\pa{1+\lambda^{-1/\beta}}^{-\beta}e^{-{z^{1/\beta}}}dz\\
&=&1-\frac{2^{\beta}}{\sqrt{\lambda}\pa{1+\lambda^{-1/\beta}}^{\beta}}\;\;=\;\;
\left(\frac{2}{\lambda^{1/\beta}+1}\right)^\beta\left[\left(\frac{\lambda^{1/\beta}+1}{2}\right)^\beta
-\sqrt{\lambda}\right]\\
&\le&\left(\frac{\lambda^{1/\beta}+1}{2}\right)^\beta
-\sqrt{\lambda}\;\;\le\;\;\lambda-\sqrt{\lambda}\;\;<\;\;(3/5)(\lambda-1).
\end{eqnarray*}
The particular case of $\beta=0$ is straightforward.\\

\paragraph{\bf Acknowledgements} One of the authors is grateful to Vladimir Koltchinskii for stimulating discussions and especially letting him know about the nice properties of VC-subgraph classes and all authors would like to thank the referee for his/her many useful comments.


\begin{thebibliography}{}

\bibitem[Audibert and Catoni, 2011]{MR2906886}
Audibert, J.-Y. and Catoni, O. (2011).
\newblock Robust linear least squares regression.
\newblock {\em Ann. Statist.}, 39(5):2766--2794.

\bibitem[Baraud, 2002]{MR1918295}
Baraud, Y. (2002).
\newblock Model selection for regression on a random design.
\newblock {\em ESAIM Probab. Statist.}, 6:127--146.

\bibitem[Baraud, 2011]{MR2834722}
Baraud, Y. (2011).
\newblock Estimator selection with respect to {H}ellinger-type risks.
\newblock {\em Probab. Theory Related Fields}, 151(1-2):353--401.

\bibitem[Barron et~al., 1999]{MR1679028}
Barron, A., Birg{\'e}, L., and Massart, P. (1999).
\newblock Risk bounds for model selection via penalization.
\newblock {\em Probab. Theory Related Fields}, 113(3):301--413.

\bibitem[Barron, 1991]{MR1154352}
Barron, A.~R. (1991).
\newblock Complexity regularization with application to artificial neural
  networks.
\newblock In {\em Nonparametric Functional Estimation and Related Topics
  ({S}petses, 1990)}, volume 335 of {\em NATO Adv. Sci. Inst. Ser. C Math.
  Phys. Sci.}, pages 561--576. Kluwer Acad. Publ., Dordrecht.

\bibitem[Birg{\'e}, 1983]{MR722129}
Birg{\'e}, L. (1983).
\newblock Approximation dans les espaces m\'etriques et th\'eorie de
  l'estimation.
\newblock {\em Z. Wahrsch. Verw. Gebiete}, 65(2):181--237.

\bibitem[Birg{\'e}, 1984]{MR762855}
Birg{\'e}, L. (1984).
\newblock Stabilit\'e et instabilit\'e du risque minimax pour des variables
  ind\'ependantes \'equidistribu\'ees.
\newblock {\em Ann. Inst. H. Poincar\'e Probab. Statist.}, 20(3):201--223.

\bibitem[Birg{\'e}, 2006]{MR2219712}
Birg{\'e}, L. (2006).
\newblock Model selection via testing: an alternative to (penalized) maximum
  likelihood estimators.
\newblock {\em Ann. Inst. H. Poincar\'e Probab. Statist.}, 42(3):273--325.

\bibitem[Birg{\'e}, 2013]{Robusttests}
Birg{\'e}, L. (2013).
\newblock Robust tests for model selection.
\newblock In Banerjee, M., Bunea, F., Huang, J., Koltchinskii, V., and
  Maathuis, M.~H., editors, {\em From Probability to Statistics and Back:
  High-Dimensional Models and Processes}, volume~9, pages 47--64. IMS
  Collections.

\bibitem[Birg{\'e} and Massart, 1993]{MR1240719}
Birg{\'e}, L. and Massart, P. (1993).
\newblock Rates of convergence for minimum contrast estimators.
\newblock {\em Probab. Theory Related Fields}, 97(1-2):113--150.

\bibitem[Birg{\'e} and Massart, 1997]{MR1462939}
Birg{\'e}, L. and Massart, P. (1997).
\newblock From model selection to adaptive estimation.
\newblock In {\em Festschrift for Lucien Le Cam}, pages 55--87. Springer, New
  York.

\bibitem[Birg{\'e} and Massart, 1998]{MR1653272}
Birg{\'e}, L. and Massart, P. (1998).
\newblock Minimum contrast estimators on sieves: exponential bounds and rates
  of convergence.
\newblock {\em Bernoulli}, 4(3):329--375.

\bibitem[Birg{{\'e}} and Massart, 2007]{MR2288064}
Birg{{\'e}}, L. and Massart, P. (2007).
\newblock Minimal penalties for {G}aussian model selection.
\newblock {\em Probab. Theory Related Fields}, 138(1-2):33--73.

\bibitem[Dudley, 1984]{MR876079}
Dudley, R.~M. (1984).
\newblock A course on empirical processes.
\newblock In {\em \'{E}cole d'\'{\'e}t\'e de Probabilit\'es de {S}aint-{F}lour,
  {XII}---1982}, volume 1097 of {\em Lecture Notes in Math.}, pages 1--142.
  Springer, Berlin.

\bibitem[Ghosal et~al., 2000]{MR1790007}
Ghosal, S., Ghosh, J.~K., and van~der Vaart, A.~W. (2000).
\newblock Convergence rates of posterior distributions.
\newblock {\em Ann. Statist.}, 28(2):500--531.

\bibitem[Gin{\'e} and Koltchinskii, 2006]{MR2243881}
Gin{\'e}, E. and Koltchinskii, V. (2006).
\newblock Concentration inequalities and asymptotic results for ratio type
  empirical processes.
\newblock {\em Ann. Probab.}, 34(3):1143--1216.

\bibitem[Grenander, 1981]{MR599175}
Grenander, U. (1981).
\newblock {\em Abstract inference}.
\newblock John Wiley \& Sons, Inc., New York.
\newblock Wiley Series in Probability and Mathematical Statistics.

\bibitem[H{{\'a}}jek, 1972]{MR0400513}
H{{\'a}}jek, J. (1972).
\newblock Local asymptotic minimax and admissibility in estimation.
\newblock In {\em Proceedings of the {S}ixth {B}erkeley {S}ymposium on
  {M}athematical {S}tatistics and {P}robability ({U}niv. {C}alifornia,
  {B}erkeley, {C}alif., 1970/1971), {V}ol. {I}: {T}heory of statistics}, pages
  175--194. Univ. California Press, Berkeley, Calif.

\bibitem[Huber, 1964]{MR0161415}
Huber, P.~J. (1964).
\newblock Robust estimation of a location parameter.
\newblock {\em Ann. Math. Statist.}, 35:73--101.

\bibitem[Huber, 1981]{MR606374}
Huber, P.~J. (1981).
\newblock {\em Robust Statistics}.
\newblock John Wiley \& Sons, Inc., New York.
\newblock Wiley Series in Probability and Mathematical Statistics.

\bibitem[Ibragimov and Has{'}minski{\u\i}, 1980]{Ibrag-Hasm}
Ibragimov, I.~A. and Has{'}minski{\u\i}, R.~Z. (1980).
\newblock On estimate of the density function.
\newblock {\em Zap. Nauchn. Semin. LOMI}, 98(61--85).

\bibitem[Ibragimov and Has{'}minski{\u\i}, 1981]{MR620321}
Ibragimov, I.~A. and Has{'}minski{\u\i}, R.~Z. (1981).
\newblock {\em Statistical Estimation. Asymptotic Theory}, volume~16.
\newblock Springer-Verlag, New York.

\bibitem[Klein and Rio, 2005]{MR2135312}
Klein, T. and Rio, E. (2005).
\newblock Concentration around the mean for maxima of empirical processes.
\newblock {\em Ann. Probab.}, 33(3):1060--1077.

\bibitem[Kolmogorov and Tihomirov, 1961]{MR0124720}
Kolmogorov, A.~N. and Tihomirov, V.~M. (1961).
\newblock {$\varepsilon $}-entropy and {$\varepsilon $}-capacity of sets in
  functional space.
\newblock {\em Amer. Math. Soc. Transl. (2)}, 17:277--364.

\bibitem[Koltchinskii, 2006]{MR2329442}
Koltchinskii, V. (2006).
\newblock Local {R}ademacher complexities and oracle inequalities in risk
  minimization.
\newblock {\em Ann. Statist.}, 34(6):2593--2656.

\bibitem[Le~Cam, 1970]{MR0267676}
Le~Cam, L. (1970).
\newblock On the assumptions used to prove asymptotic normality of maximum
  likelihood estimates.
\newblock {\em Ann. Math. Statist.}, 41:802--828.

\bibitem[Le~Cam, 1973]{MR0334381}
Le~Cam, L. (1973).
\newblock Convergence of estimates under dimensionality restrictions.
\newblock {\em Ann. Statist.}, 1:38--53.

\bibitem[Le~Cam, 1975]{MR0395005}
Le~Cam, L. (1975).
\newblock On local and global properties in the theory of asymptotic normality
  of experiments.
\newblock In {\em Stochastic processes and related topics ({P}roc. {S}ummer
  {R}es. {I}nst. {S}tatist. {I}nference for {S}tochastic {P}rocesses, {I}ndiana
  {U}niv., {B}loomington, {I}nd., 1974, {V}ol. 1; dedicated to {J}erzy
  {N}eyman)}, pages 13--54. Academic Press, New York.

\bibitem[Le~Cam, 1986]{MR856411}
Le~Cam, L. (1986).
\newblock {\em Asymptotic Methods in Statistical Decision Theory}.
\newblock Springer Series in Statistics. Springer-Verlag, New York.

\bibitem[Le~Cam, 1990]{Lecam-MLE}
Le~Cam, L. (1990).
\newblock Maximum likelihood: An introduction.
\newblock {\em Inter. Statist. Review}, 58(2):153--171.

\bibitem[Le~Cam and Yang, 1990]{MR1066869}
Le~Cam, L. and Yang, G.~L. (1990).
\newblock {\em Asymptotics in Statistics. Some Basic Concepts}.
\newblock Springer Series in Statistics. Springer-Verlag, New York.

\bibitem[Massart, 2007]{MR2319879}
Massart, P. (2007).
\newblock {\em Concentration Inequalities and Model Selection}, volume 1896 of
  {\em Lecture Notes in Mathematics}.
\newblock Springer, Berlin.
\newblock Lectures from the 33rd Summer School on Probability Theory held in
  Saint-Flour, July 6--23, 2003.

\bibitem[Massart and N{{\'e}}d{{\'e}}lec, 2006]{MR2291502}
Massart, P. and N{{\'e}}d{{\'e}}lec, {\'E}. (2006).
\newblock Risk bounds for statistical learning.
\newblock {\em Ann. Statist.}, 34(5):2326--2366.

\bibitem[Sart, 2014]{Sart2014}
Sart, M. (2014).
\newblock Estimation of the transition density of a markov chain.
\newblock {\em Annales de l'I.H.P. Probabilit{\'e}s et statistiques},
  50(3):1028--1068.

\bibitem[Sart, 2015]{refId0}
Sart, M. (2015).
\newblock Model selection for poisson processes with covariates.
\newblock {\em ESAIM: PS}, 19:204--235.

\bibitem[Sart, 2016]{sart2016}
Sart, M. (2016).
\newblock Robust estimation on a parametric model via testing.
\newblock {\em Bernoulli}, 22(3):1617--1670.

\bibitem[van~de Geer, 1995]{MR1324688}
van~de Geer, S. (1995).
\newblock The method of sieves and minimum contrast estimators.
\newblock {\em Math. Methods Statist.}, 4(1):20--38.

\bibitem[van~der Vaart and Wellner, 2009]{MR2797943}
van~der Vaart, A. and Wellner, J.~A. (2009).
\newblock A note on bounds for {VC} dimensions.
\newblock In {\em High Dimensional Probability {V}: the {L}uminy volume},
  volume~5 of {\em Inst. Math. Stat. Collect.}, pages 103--107. Inst. Math.
  Statist., Beachwood, OH.

\bibitem[van~der Vaart, 1998]{MR1652247}
van~der Vaart, A.~W. (1998).
\newblock {\em Asymptotic statistics}, volume~3 of {\em Cambridge Series in
  Statistical and Probabilistic Mathematics}.
\newblock Cambridge University Press, Cambridge.

\bibitem[van~der Vaart and Wellner, 1996]{MR1385671}
van~der Vaart, A.~W. and Wellner, J.~A. (1996).
\newblock {\em Weak Convergence and Empirical Processes. With Applications to
  Statistics}.
\newblock Springer Series in Statistics. Springer-Verlag, New York.

\bibitem[Whittaker and Watson, 1996]{MR1424469}
Whittaker, E.~T. and Watson, G.~N. (1996).
\newblock {\em A Course of Modern Analysis}.
\newblock Cambridge Mathematical Library. Cambridge University Press,
  Cambridge.
\newblock An introduction to the general theory of infinite processes and of
  analytic functions; with an account of the principal transcendental
  functions, Reprint of the fourth (1927) edition.

\bibitem[Yang and Barron, 1999]{MR1742500}
Yang, Y. and Barron, A. (1999).
\newblock Information-theoretic determination of minimax rates of convergence.
\newblock {\em Ann. Statist.}, 27(5):1564--1599.

\end{thebibliography}
\bibliographystyle{apalike}

\end{document}